\documentclass{article}
\usepackage[T1]{fontenc}
\usepackage[latin9]{inputenc}
\usepackage[tmargin=4cm,bmargin=4cm,lmargin=3.2cm,rmargin=3.2cm]{geometry}
\usepackage{amsmath,amsthm,amssymb}
\usepackage{mathrsfs}
\usepackage{tikz-cd}
\usepackage[all]{xy}
\usepackage{xcolor}
\usepackage[pdfusetitle,bookmarks=true,colorlinks=true,
            linkcolor=purple,citecolor=magenta]{hyperref}

\DeclareFontFamily{U}{mathx}{\hyphenchar\font45}
\DeclareFontShape{U}{mathx}{m}{n}{
      <5> <6> <7> <8> <9> <10>
      <10.95> <12> <14.4> <17.28> <20.74> <24.88>
      mathx10
      }{}
\DeclareSymbolFont{mathx}{U}{mathx}{m}{n}
\DeclareFontSubstitution{U}{mathx}{m}{n}
\DeclareMathAccent{\widecheck}{0}{mathx}{"71}

\theoremstyle{plain}
\newtheorem{thm}{Theorem}[subsection]
\newtheorem{prop}[thm]{Proposition}
\newtheorem{lem}[thm]{Lemma}
\newtheorem{cor}[thm]{Corollary}
\newtheorem{conjecture}[thm]{Conjecture}
\theoremstyle{definition}
\newtheorem{defn}[thm]{Definition}
\newtheorem{example}[thm]{Example}
\newtheorem{property}[thm]{Property}
\theoremstyle{remark}
\newtheorem{rem}[thm]{Remark}

\providecommand{\bysame}{\leavevmode\hbox to3em{\hrulefill}\thinspace}

\title{The Categorical Fargues-Scholze Correspondence at Generic Parameters}
\author{Nikolas Castro}
\date{}

\begin{document}

\maketitle

\begin{abstract}

Let $G$ be a quasi-split reductive group over a finite extension $F$ of $\mathbb{Q}_p$.
We study the Fargues-Scholze categorical Langlands correspondence at generic parameters $\phi$, namely those whose semisimplification does not coincide with that of any parameter with nontrivial monodromy.
Extending the theory of quasi-exceptional sets on the nilpotent cone developed by Bezrukavnikov to the setting of disconnected reductive groups which arise as stabilizers of semisimple generic parameters, to any irreducible representation $\rho$ of such a group we construct sheaves labeled $\overline{\nabla}_\rho$, $\overline{\Delta}_\rho$, $IC_\rho$, and $\delta_\rho$ on the moduli stack of L-parameters for $G$, which were conjectured to exist by Hansen.
Using forthcoming work of Hamann, Hansen, and Mann on geometric Eisenstein series compatibility, the (conjectural, in general) behavior of the correspondence at generous parameters, and the (conjectural, in general) agreement between Fargues-Scholze and classical semisimple L-parameters we prove that these sheaves correspond under the Fargues-Scholze functor to the sheaves $i_{b!}\pi$, $i_{b\sharp}\pi$, $i_{b\sharp!}\pi$, and $i_{b!*}\pi$ on $\text{Bun}_G$, where $\pi$ is a smooth irreducible representation with Fargues-Scholze parameter $\phi$.
As applications we prove t-exactness properties of the Fargues-Scholze functor specialized at a generic parameter, and that Hecke operators are perverse t-exact on the generic locus, as conjectured by Hansen and Hamann-Lee.
We also construct and study foundational properties of Hecke eigensheaves at generic parameters, compute non-semisimplified Fargues-Scholze parameters on the generic locus, and construct a Fargues-Scholze-theoretic version of the BM-O bijection at generic parameters.

\end{abstract}

\tableofcontents{}

\section{Introduction}

Let $G$ be a quasi-split reductive group over a finite extension $F$ of $\mathbb{Q}_p$ with residue field $\mathbb{F}_q$.  Let $\ell$ be a prime distinct from $p$, and let $\Lambda$ be a $\mathbb{Z}_\ell$-algebra.
We denote by $W_F$ and $I_F$ the Weil group of $F$ and its inertia subgroup, respectively.  
When $\Lambda=\overline{\mathbb{Q}}_\ell$, the classical local Langlands correspondence roughly predicts a bijection between smooth irreducible $\overline{\mathbb{Q}}_\ell$-representations of $G(F)$ and Weil-Deligne parameters of $G$ enhanced with extra data (see for example \cite{KalLLC}).
To every element $b$ in the Kottwitz set $ B(G)$ of $G$, there is an associated inner form $G_b$ of a Levi subgroup $G_{\{b\}}$ of $G$, with $G_{\{b\}}$ being $G$ itself exactly when $b$ is basic.
In the early 90s, Vogan \cite{Vog93} realized that to understand non-quasi-split groups, on the representation-theoretic side of the local Langlands correspondence one should not study $G$ in isolation but rather $G$ together with (some of) its inner forms.
There are different approaches to carrying out this program precisely, see \cite{KalLLC} for an overview, but the approach most relevant to the present paper parametrizes extended pure inner forms of $G$ by the set of basic elements of the Kottwitz set, $B(G)^{\text{bas}}$.
A further enhancement of this approach introduced in \cite{BMO} equips $\phi$ with an irreducible representation $\rho$ of its stabilizer  in $\hat G$, denoted $S_\phi$, and one of the primary results of that paper proves that if the $B(L)^{\text{bas}}$-parametrized local Langlands conjecture holds for all Levi subgroups $L$ of $G$, then there is a natural correspondence between pairs $(\phi,\rho)$ and $(b,\pi)$, where $\pi$ is a smooth irreducible $\overline{\mathbb{Q}}_\ell$-representation of $G_b(F)$, and $b$ is now \textit{any} element of $B(G)$.

In another recent development, inspired by the geometric Langlands correspondence for algebraic curves over a characteristic $0$ algebraically closed field, Fargues and Scholze \cite{FS} have built foundations for a sweeping categorical form of the local Langlands correspondence for quasi-split reductive groups $G$ which even makes sense for integral coefficients, though in this paper we will focus on $\overline{\mathbb{Q}}_\ell$-coefficients.
To describe their work, we let $\text{Par}_G$ denote the moduli stack of $\overline{\mathbb{Q}}_\ell$-valued L-parameters for $G$, and let $\text{Bun}_G$ denote the moduli stack of $G$-bundles on the Fargues-Fontaine curve.
Fargues and Scholze construct a category of $\ell$-adic sheaves on $\text{Bun}_G$, $D^{\text{lis}}(\text{Bun}_G,\overline{\mathbb{Q}}_\ell)$, as well as a spectral action of $\text{Perf}(\text{Par}_G)$ on $D^{\text{lis}}(\text{Bun}_G,\overline{\mathbb{Q}}_\ell)$, which 
they use to construct a functor $c_\psi:D^{\text{lis}}(\text{Bun}_G,\overline{\mathbb{Q}}_\ell)\to \text{QCoh}(\text{Par}_G)$, depending only on choice of Whittaker data $\psi$ for $G$, which they conjecture induces an equivalence $ D^{\text{lis}}(\text{Bun}_G,\overline{\mathbb{Q}}_\ell)^\omega\cong \text{Coh}^{b,\text{qc}}(\text{Par}_G)$ between the compact objects of $D^{\text{lis}}(\text{Bun}_G,\overline{\mathbb{Q}}_\ell)$ and the bounded derived category of coherent sheaves on $\text{Par}_G$ with quasi-compact support.
Taking ind-completions of both sides induces a conjectural equivalence $\textbf{L}_\psi:D^{\text{lis}}(\text{Bun}_G,\overline{\mathbb{Q}}_\ell)\ \tilde{\to}\ \text{IndCoh}(\text{Par}_G)$.

In groundbreaking recent work \cite{HM}, Hansen and Mann prove this functor is an equivalence for $\text{GL}(n)$, subject to a compatibility result between $c_\psi$ and geometric Eisenstein series which is an ongoing project of Hamann, Hansen, and Mann.  
By \cite[Corollary 1.5]{HSX} the Hansen-Mann method, conditional on the same geometric Eisenstein compatibility result, implies the categorical local Langlands correspondence for many other groups for which classical local Langlands is well understood.
\newline

For any $b\in B(G)$, let $D(G_b(F), \overline{\mathbb{Q}}_\ell)$ denote the symmetric monoidal derived $\infty$-category of smooth $\overline{\mathbb{Q}}_\ell$-representations of $G_b(F)$.  
In \cite{FS}, it is shown that $D(G_b(F), \overline{\mathbb{Q}}_\ell)$ embeds fully faithfully into $D^{\text{lis}}(\text{Bun}_G,\overline{\mathbb{Q}}_\ell)$ in a canonical way.
In fact, the points of $\text{Bun}_G$ are exactly $B(G)$, and shriek pushforward along the inclusion $i_b:\text{Bun}_G^b\hookrightarrow \text{Bun}_G$ induces an equivalence between $D(G_b(F), \overline{\mathbb{Q}}_\ell)$ and the full subcategory of complexes in $D^{\text{lis}}(\text{Bun}_G,\overline{\mathbb{Q}}_\ell)$ supported on $b$ \cite[Proposition VII.7.1]{FS}.
Fargues and Scholze, using this embedding along with the spectral action, are able to construct a semisimple local Langlands map $(b,\pi)\to\phi_\pi^{\text{FS}}$ from pairs $(b,\pi)$ where $b\in B(G)$ and $\pi$ is a smooth irreducible representation of $G_b(F)$, to semisimple L-parameters \cite[Definition IX.0.4]{FS}.
It has been proven in many cases \cite{BMHN} \cite{CF} \cite{{DvHKZ26}} \cite{FS} \cite{HamGanTak} \cite{Pen} that this map agrees with the semisimplification of classical local Langlands correspondences.

On the spectral side of the correspondence, an ordered pair $(\phi,\rho)$ as appearing in the BM-O bijection can be associated to the sheaf $i_{\phi*}\rho\in\text{Coh}(\text{Par}_G)$, where $i_\phi:BS_\phi\hookrightarrow \text{Par}_G$ is the natural inclusion.
So we seemingly have the (conjectural, in general) BM-O bijection appearing in the Fargues-Scholze categorical correspondence.

But it is in fact not always the case that $c_\psi(i_{b!}\pi)=i_{\phi*}\rho$; in general more complicated sheaves will be involved.  To go into more detail we stratify $\text{Par}_G$ using the behavior of its fibers along the coarse moduli map $q:\text{Par}_G\to\text{Par}_G^{\text{spec}}$.
By \cite[Proposition VIII.3.2]{FS} and independently \cite[Theorem 4.13]{DHKM} the closed points of $\text{Par}_G^{\text{spec}}$ correspond to \textit{semisimple} L-parameters (Definition \ref{ssLparam}).
For any closed point $\phi^{\text{ss}}\in\text{Par}_G^{\text{spec}}$, by Proposition \ref{fibstruct} (see also \cite[Proposition 1.3.1]{BeijingNotes}) $q^{-1}(\phi^{\text{ss}})^{\text{red}}$ 
consists of the variety of ordered pairs $(u,N)\in \mathcal{U}_{S_{\phi^{\text{ss}}}}\times \hat{\mathfrak{g}}^{\text{ad}\phi(I_F)}$ such that $\text{ad}(\text{Fr})\cdot N = q^{-1}N$ and $\text{ad}(u)\cdot N = N$, quotiented (stack-theoretically) by $S_{\phi^{\text{ss}}}$ under simultaneous conjugation, where $\text{Fr}$ denotes geometric Frobenius,
$\mathcal{U}_{S_{\phi^{\text{ss}}}}$ is the unipotent variety of $S_{\phi^{\text{ss}}}$, and $\hat{\mathfrak{g}}$ is the Lie algebra of $\hat{G}$, the dual group of $G$.
We say a point $\phi$ in $\text{Par}_G$ is \textit{generous} if the reduced fiber of $q$ above $q(\phi)$ is $BS_\phi$.
In other words, $\phi$ is semisimple and no other $L$ parameters have $\phi$ as their semisimplification.
More generally, $\phi$ is \textit{generic} if the reduced fiber of $q$ above $q(\phi)$ is isomorphic to $\mathcal{U}_{S_{q(\phi)}}/S_{q(\phi)}$.
In other words, $\phi$ does not share a semisimplification with any L-parameter with nontrivial monodromy.
These two conditions were introduced in Hansen's Beijing notes \cite[Definition 2.1.5, Proposition 2.3.1]{BeijingNotes}, and Hamann-Lee introduced the conditions of Langlands-Shahidi type/weakly Langlands-Shahidi type \cite[Definition 6.2]{HL} which are respectively equivalent (and moreover behave well for torsion coefficients, where the two sets of definitions are no longer equivalent in general).
Prior to these Hamann \cite[Definition 3.6]{HamEis} defined the notion of a generic toral parameter, which agrees with Hansen's notion of a generous toral parameter (as opposed to Hansen's notion of a generic toral parameter).

We illustrate these conditions with a concrete example:
\begin{example}

Consider $G=\hat G=\text{GL}(n)$, and let $\phi$ be a semisimple L parameter.
Then there exist positive integers $n_1,...,n_k$ such that $n=\sum\limits_{i=1}^kn_i$ and $\phi = \bigoplus\limits_{i=1}^k\phi_i$, where each $\phi_i$ is a \textit{supercuspidal} parameter of $\text{GL}(n_i)$, meaning it is semisimple and does not factor through any proper Levi subgroups of $\text{GL}(n_i)$.  Then $\phi$ is generic exactly when $\phi_i\ncong \phi_j(1)$ for all $i\neq j$ where $(1)$ denotes tensoring with the $\ell$-adic cyclotomic character, and $\phi$ is generous exactly when it is generic and furthermore $\phi_i\ncong \phi_j$ for all $i\neq j$.
\end{example}

At generous parameters the Fargues-Scholze categorical local Langlands conjecture is expected to be quite well behaved.
To start, we have the following conjecture, stated roughly:

\begin{conjecture}\label{generousstalks}(Compare \cite[Conjecture 2.1.8(i)]{BeijingNotes})
Let $\phi$ be a generous parameter of a quasi-split reductive group $G$.
Then for all $b\in B(G)$ and smooth irreducible representations $\pi$ of $G_b(F)$ with Fargues-Scholze parameter $\phi$, there exists an irreducible representation $\rho$ of $S_\phi$ such that $c_{\psi}(i_{b!}\pi)=i_{\phi*}(\rho)$.  Furthermore, the map $(b,\pi)\mapsto(\phi,\rho^\vee)$ agrees with the BM-O bijection coming from classical local Langlands.
\end{conjecture}

\noindent See Conjecture \ref{GenerousConj} for the precise statement.  
So far this is known for $G=\text{GL}(n)$ \cite[Theorem 8.6(1)]{Zou25} and tori \cite{Zou22}.  Linus Hamann has also informed the author that this conjecture is likely within reach for some classical groups such as unitary groups.

When $\phi$ is semisimple the stabilizer of $\phi$ in $\hat G$ is a (possibly disconnected) reductive group $S_\phi$ by Proposition \ref{ssLparamprop}.
So when $\phi$ is generous (and therefore semisimple) Conjecture \ref{generousstalks} together with the categorical conjecture imply that the abelian category of finite (see Section \ref{FSSect}) lisse $\overline{\mathbb{Q}}_\ell$-sheaves on $\text{Bun}_G$ with Fargues-Scholze parameter $\phi$ is also semisimple.
So in particular if $\pi$ and $\pi'$ are smooth irreducible representations of $G_b(F)$ and $G_{b'}(F)$ with generous Fargues-Scholze parameter $\phi$ for $b,b'\in B(G)$, then there are no nontrivial relations between sheaves of the form $i_{b!}\pi$ and $i_{b'!}\pi'$ when $b\neq b'$, or (in the case $b=b'$) when $\pi\ncong\pi'$.

Fargues and Scholze prove a semiorthogonal decomposition of $D^{\text{lis}}(\text{Bun}_G)$:

\begin{thm}\cite[Proposition VII.7.3]{FS}
$D^\textup{lis}(\textup{Bun}_G,\overline{\mathbb{Q}}_\ell)$ has a semiorthogonal decomposition indexed by $B(G)$ into the categories $D^\textup{lis}(\textup{Bun}^b_G,\overline{\mathbb{Q}}_\ell)$ supported at a single point $b\in B(G)$.
\end{thm}

While this decomposition is not orthogonal in general, as discussed above it is expected to be orthogonal after specializing to a generous L-parameter.
In this case another closely related conjecture of Hansen follows:

\begin{conjecture}\cite[Conjecture 2.1.8(ii)]{BeijingNotes}\label{identicalpushforwards}
If $\pi$ is a smooth irreducible $\overline{\mathbb{Q}}_\ell$-representation of $G_b(F)$ with generous Fargues-Scholze parameter then $i_{b\sharp}\pi\cong i_{b!}\pi\cong i_{b*}\pi$.
    Here $i_{b\sharp}$ is the left adjoint of $i_b^*$ constructed in \cite[Proposition VII.7.2]{FS}.
\end{conjecture}

Again this is known for $\text{GL}(n)$ \cite[Theorem 8.6(2)]{Zou25}.
Similar conjectures about the behavior of the categorical correspondence at a generous parameter were made by Hamann and Lee \cite[Conjecture 6.4]{HL}.
\newline

\noindent With the shape of the categorical correspondence at generous parameters relatively well understood (conjecturally, in general), the goal of this paper is to lay foundations for the study of the categorical correspondence at generic parameters, with a particular eye towards introducing the objects needed to precisely formulate an analogue of Conjecture \ref{generousstalks}.
More specifically, we seek to understand $c_\psi(i_{b\sharp}\pi), c_\psi(i_{b!}\pi)$, and related sheaves when the Fargues-Scholze parameter of $\pi$ is generic.
So let $\phi$ from now on be a semisimple generic parameter with centralizer $S_\phi$.
By Proposition \ref{ssLparamprop} $S_\phi$ has reductive identity component.
For each irreducible representation $\rho$ of $S_\phi$, in Section \ref{SpringerSect} we construct sheaves $\overline{\nabla}_\rho$, $\overline\Delta_\rho$, $IC_\rho$, and $\delta_\rho$ on $\mathcal{U}_{S_\phi}/S_\phi$.
Our main theorem is then, roughly: 
\begin{thm} (Theorem \ref{MainThm})\label{IntroMainThm}
Let $\pi$ be a smooth irreducible representation of $G_b(F)$ with generic Fargues-Scholze parameter $\phi$. 
Suppose that the Fargues-Scholze parameter of $\pi$ agrees with its classical semisimple local Langlands parameter, Conjecture \ref{generousstalks} is known for all Levi subgroups of $G$, $c_\psi$ is compatible with geometric Eisenstein series, and $c_\psi$ is fully faithful on $D(\textup{Bun}_G)^{\omega,\phi}$.
Then, letting $\nu_{\phi}$ denote the inclusion of $\mathcal{U}_{S_\phi}/S_\phi$ into $\textup{Par}_G$, $c_\psi(i_{b!}\pi)\cong \nu_{\phi*}\overline{\nabla}_\rho$, $c_\psi(i_{b\sharp}\pi)\cong\nu_{\phi*}\overline{\Delta}_\rho$, $c_\psi(i_{b\sharp!}\pi)\cong\nu_{\phi*}IC_\rho$, and $c_\psi(i_{b!*}\pi)\cong\nu_{\phi*}\delta_\rho$.
   
\end{thm}

\begin{rem}

Regarding the conditions of the theorem, classical local Langlands compatibility is known quite broadly \cite{BMHN} \cite{CF} \cite{{DvHKZ26}} \cite{FS} \cite{HamGanTak} \cite{Pen}, and up to geometric Eisenstein series compatibility the full categorical correspondence is also known in many cases \cite{HM}\cite{HSX}.
Geometric Eisenstein series compatibility is forthcoming work of Hamann, Hansen, and Mann \cite{HHM}, and as mentioned above Conjecture \ref{generousstalks} is known for $\text{GL}(n)$\cite[Theorem 8.6]{Zou25}.
So for $G=\text{GL}(n)$ the theorem is conditional only on \cite{HHM}, and for many classical groups the only further obstruction is Conjecture \ref{generousstalks}, which should therefore be viewed as the primary bottleneck to making the theorem unconditional in more cases.
    
\end{rem}

The $i_{b\sharp!}$ and $i_{b!*}$ functors will be defined below.
They are analogues of intermediate extension functors from the theory of perverse sheaves in algebraic geometry.  
The sheaves $\overline{\nabla}_\rho$, $\overline{\Delta}_\rho$, $IC_\rho$, and $\delta_\rho$ were constructed in \cite[Section 2.3]{BeijingNotes} when $\phi$ is the trivial parameter.

To construct $\overline{\nabla}_\rho$, $\overline{\Delta}_\rho$, $IC_\rho$, and $\delta_\rho$, let $\hat{M}$ be the smallest Levi which $\phi$ factors through. 
We will fix for all time a maximal torus and Borel of $\hat{G}$, which will be used to define dominance of characters and standardness of Levi and parabolic subgroups, and we can assume without loss of generality that $\hat{M}$ is standard by replacing $\phi$ by one of its conjugates if necessary.
By the representation theory of disconnected reductive groups \cite{AHR}, $\rho$ corresponds to a pair $(\lambda, E_\rho)$, well defined up to conjugacy, where $\lambda$ is a dominant weight of the maximal torus $S_\phi^\circ\cap T$ of $S_\phi^\circ$, and $E_\rho$ is an irreducible representation of $k^\alpha[A^\lambda] $, a certain twisted algebra related to the component group of $S_\phi$.
Let $S_{\phi,M}$ denote the stabilizer of $\phi$ in $\hat{M}$.
We will show that the component group of $S_{\phi,M}$ is normal in that of $S_\phi$, and therefore there will be an irreducible representation $\tau$ of $S_{\phi,M}$, unique up to $S_\phi$-conjugacy, such that $\rho$ is a direct summand of $\text{Ind}_{\hat{P}\cap S_\phi}^{S_\phi}(\tau)$.

In one of the central technical sections of the paper, we will extend some aspects of Springer Theory to the setting of disconnected reductive groups appearing as the stabilizers of semisimple generic parameters.
In particular, using the Springer resolution we will construct $\overline{\nabla}_\rho$ as a certain direct summand of $\mu_*\eta^*\tau$.
Its "dual" $\overline{\Delta}_\rho$ will be constructed similarly, using $^{w_0}\tau$ in place of $\tau$, where $w_0$ is the longest-length Weyl group element, and $IC_\rho$ will be constructed from $\overline{\nabla}_\rho$ and $\overline{\Delta}_\rho$ using the theory of perverse-coherent sheaves \cite{PervCoh}.
Finally, $\delta_\rho$ is just $i_{\phi*}\rho$, where $i_\phi$ is the inclusion of $*/S_\phi$ into $\text{Par}_G$ at $\phi$.

The relevance of Springer theory to the study of $\text{Par}_G$ stems from the spectral parabolic induction functor $\text{Eis}^{\text{spec}}:\text{Coh}(\text{Par}_M)\to \text{Coh}(\text{Par}_G)$, which after specializing to a generic parameter will be computable using coherent Springer Theory (Proposition \ref{SpecEisCalc}).

Continuing with the proof overview of Theorem \ref{IntroMainThm}, we prove that if $\pi$ is a smooth irreducible representation of $G_b(F)$ with generic Fargues-Scholze parameter and $(M,\sigma)$ is its supercuspidal support, then $\sigma$ has generous Fargues-Scholze parameter.
On the automorphic side, the sheaves $i_{b!}\pi$ and $i_{b\sharp}\pi$ can, by results from \cite{HHS}, be expressed as direct summands of certain geometric Eisenstein series applied to sheaves of the form $i_{b_M!}\pi$.
And as mentioned above, $\nu_{\phi*}\overline{\nabla}_\rho$ and $\nu_{\phi*}\overline{\Delta}_\rho$ are direct summands of certain spectral Eisenstein series functors applied to $i_{\phi_M*}\tau$ (where $\phi_M$ is a certain lift of $\phi$ to $\text{Par}_M$).
By Conjecture \ref{generousstalks}, for all $b_M\in B(M)$ mapping to $b$, $c_\psi(i_{b_M!}\sigma)\cong i_{\phi*}\tau$ for some irreducible representation $\tau$ of $S_{\phi,M}$.
The isomorphisms $c_\psi(i_{b!}\pi)\cong\nu_{\phi*}\overline{\nabla}_\rho$ and $c_\psi(i_{b\sharp}\pi)\cong\nu_{\phi*}\Delta_{\rho}$ are then proven using the compatibility of the Fargues-Scholze functor with geometric Eisenstein series.

To prove $c_\psi(i_{b\sharp!}\pi) \cong \nu_{\phi*}(IC_\rho)$, we prove a t-exactness result for $c_\psi$, namely that it is t-exact with respect to the \textit{hadal} t-structure on $D(\text{Bun}_G,\overline{\mathbb{Q}}_\ell)^\phi$ and the \textit{perverse coherent} t-structure on $\text{Coh}(\text{Par}_G^\phi)$.
Since $IC_\rho=\text{Im}(\Delta_\rho\to\nabla_\rho)$ and $i_{b\sharp!}=\text{Im}(i_{b\sharp}\pi\to i_{b!}\pi)$ in the perverse-coherent and hadal t-structures respectively the desired result follows from the first two calculations.
Finally, to show $c_\psi(i_{b!*}\pi)\cong\nu_{\phi*}(\delta_\rho)$ we exploit that $i_{\phi*}\rho$ is an irreducible quotient of $\nabla_\rho$ in the standard t-structure on $\text{Coh}(\text{Par}_G^\phi)$, and $i_{b!*}$ is the unique irreducible quotient of $i_{b!}\pi$ in the perverse t-structure on $D(\text{Bun}_G,\overline{\mathbb{Q}}_\ell)$, together with Corollary \ref{PTruncStability}, which states that $c_\psi$ is t-exact with respect to those t-structures on $\text{Coh}(\text{Par}_G^\phi)$ and $D(\text{Bun}_G,\overline{\mathbb{Q}}_\ell)$.
\newline

Theorem \ref{MainThm} has a number of interesting implications.
As mentioned briefly above, it is closely related to many t-exactness theorems on the generic locus of the correspondence.

\begin{cor}(Theorem \ref{MainThm}(ii), Corollary \ref{PTruncStability})
\newline
\indent  i)\ $c_\psi:D(\textup{Bun}_G,\overline{\mathbb{Q}}_\ell)^{\textup{fin},\phi}\to\textup{Coh}(\textup{Par}_G^\phi)$ is t-exact with respect to the hadal t-structure on $D(\textup{Bun}_G,\overline{\mathbb{Q}_\ell})^\phi$ and the perverse-coherent t-structure on $\textup{Coh}(\textup{Par}_G^\phi)$.

ii)  $c_\psi:D(\textup{Bun}_G,\overline{\mathbb{Q}}_\ell)^{\textup{fin,gen}}\to\textup{Coh}(\textup{Par}_G^{\textup{gen}})$ is t-exact with respect to the perverse t-structure on $D(\textup{Bun}_G,\overline{\mathbb{Q}}_\ell)^{\textup{gen}}$ and the standard t-structure on $\textup{Coh}(\textup{Par}_G^{\textup{gen}})$.

\end{cor}

\noindent These properties were conjectured by Hansen \cite[Conjectures 2.5.3, 2.5.1]{BeijingNotes}.  Furthermore, 

\begin{thm} (Proposition \ref{HeckePExactness})
For all $V\in\textup{Rep}(^L{G})$, the Hecke operator $$T_V: D(\textup{Bun}_G)^{\textup{ULA, gen}}\to D(\textup{Bun}_G)^{\textup{ULA, gen}}$$ is perverse t-exact.
\end{thm}

\noindent This was previously proven by Hansen \cite[Theorem 2.4.4]{BeijingNotes} for the trivial parameter, and conjectured in general by Hansen \cite[Conjecture 2.4.1]{BeijingNotes}, and Hamann-Lee \cite[Conjecture 6.4]{HL}.

In a different direction, the sheaves $IC_\rho$ together with the Fargues-Scholze functor can be used to construct Hecke eigensheaves at generic parameters.
A Hecke eigensheaf for an L-parameter $\phi$ is a sheaf $\mathcal{F}$ on $\text{Bun}_G$ such that $T_V\mathcal{F}\cong\mathcal{F}\boxtimes V\circ\phi$ for all $V\in\text{Rep}(^LG)$ (where Rep denotes finite dimensional representations)
Letting $\mathcal{R}(\hat{G})$ denote the left regular representation of $\hat{G}$, for any $W\in\text{Rep}(\hat{G})$ there is a canonical isomorphism $\mathcal{R}(\hat{G})\otimes W\cong \mathcal{R}(\hat{G})\otimes W^{triv}$, where $W^{triv}$ denotes the underlying vector space of $W$ equipped with the trivial representation of $\hat{G}$.
When $\phi$ is generous, one can use this isomorphism to prove $$a_\psi(i_{\phi*}(\mathcal{R}(S_\phi)))\cong\bigoplus\limits_{(b,\pi)}i_{b!}\pi^{\text{dim}(\rho_{(b,\pi)})}$$ is a Hecke eigensheaf, where the sum is over $(b,\pi)$ such that $\pi\in\text{Irr}(G_b(F))$ has Fargues-Scholze parameter $\phi$, and $\rho_{(b,\pi)}$ is the representation of $S_\phi$ such that $c_\psi(i_{b!}\pi)\cong i_{\phi*}\rho$. 
We are using here that $\phi$ is a smooth point of $\text{Par}_G$, so $i_{\phi*}(\mathcal{R}(S_\phi))$ is an ind-perfect complex.

Using the perverse-coherent theory of intersection cohomology sheaves \cite{PervCoh}, we can construct Hecke eigensheaves at parameters with generic semisimplification.
Part of the theorem statement uses the notion of the full (ie, non semisimplified) Fargues-Scholze parameter of a representation $\pi$ with generic Fargues-Scholze parameter (Definition \ref{fullFSParam}).
It is defined to agree with the usual Fargues-Scholze parameter upon semisimplification, and the unipotent component of its Frobenius element is determined by perverse-coherent sheaf theory on $\text{Par}_G^\phi$.
Its construction is of independent interest.

\begin{thm} (Theorem \ref{HeckeEigenThm})
Let $\phi$ be an L-parameter with generic semisimplification $\phi^{\textup{ss}}$ and centralizer $S_\phi$.  
Let $O$ be the $S_{\phi^{\textup{ss}}}$-orbit of the unipotent component of $\phi(\textup{Fr})$ inside the unipotent variety of $S_{\phi^{\textup{ss}}}^\circ$.  
Then $H_\phi:=a_\psi(\nu_{\phi*}IC(O,\mathcal{R}(S_\phi)))[-d]$ is a Hecke eigensheaf for $\phi$.  
It has an isotypic decomposition into indecomposable components $H_{\phi,\rho}$ indexed by the irreducible representations $\rho$ of $S_\phi$.
For all $\rho$, $H_{\phi,\rho}$ is (ind-)hadal, and admits a filtration by (ind-)hadal subsheaves $H_{\phi,\rho}^k$ indexed by nonnegative integers.  The $k$th graded complex $G^k_{\phi,\rho}:=H_{\phi,\rho}^k/H_{\phi,\rho}^{k-1}$ surjects (as a hadal sheaf) onto a finite direct sum of complexes of the form $i_{b\sharp!}\pi$, where $\pi$ has full Fargues-Scholze parameter $\phi$, and the kernel of this map lies in the subcategory of $D(\textup{Bun}_G)^{\phi^{\textup{ss}}}$ generated by complexes of the form $i_{b\sharp!}\pi'$, where the full Fargues-Scholze parameter of $\pi'$ lies in the boundary of the closure of $\phi$ in $\textup{Par}_G$.
\end{thm}

Lastly, we briefly note the following theorem from the mostly self-contained appendix, which is proven using Harish-Chandra's theory of intertwining operators and is likely of independent interest:

\begin{thm} (Theorem \ref{GenCompRed})\label{IntroGenCompRed}
Let $\pi$ be a smooth irreducible representation of $G(F)$ with supercuspidal support $[M,\sigma]$, and let $P=MU$ be a parabolic with Levi $M$.
Assume that the classical L-parameter $\phi_\pi$ of $\pi$ is generic.  
Then $i_P^G(\sigma)$ is completely reducible.
\end{thm}

We now briefly overview the sections of the paper.  Section 2 begins with background material relevant for the main theorem and applications, including a review of the necessary inputs from the classical and Fargues-Scholze local Langlands correspondences, as well as the representation theory of disconnected reductive groups.  
It then extends some results from Springer theory and the theory of quasi-exceptional sets to the types of disconnected reductive groups appearing as centralizers of semisimple generic L-parameters, and contains the constructions of $\overline{\nabla}_\rho$, $\overline{\Delta}_\rho$, $IC_\rho$, and $\delta_\rho$.
Section 3 begins with the definition of generous and generic parameters and proofs of foundational properties of their locus in $\text{Par}_G$, and culminates in the proof of Theorem \ref{IntroMainThm}.  
In section 4 we prove the applications described above of the main theorem to t-exactness results for $c_\psi$ and $T_V$, and to the construction of Hecke eigensheaves.
Finally, in the appendix we highlight some results from the theory of intertwining operators of parabolic inductions of supercuspidal representations, and prove Theorem \ref{IntroGenCompRed}.

\subsection{Notation} \label{notation}
\begin{itemize}

\item $F$ denotes a finite extension of $\mathbb{Q}_p$ with absolute Galois group $\Gamma = \text{Gal}(\overline{F}/F)$.  Its Weil group will be denoted $W_F$, and $I_F$ and $P_F$ are its inertia and wild inertia subgroups, respectively.  We fix a lift of geometric Frobenius to $W_F$, denoted $\text{Fr}$.

\item $\ell$ will denote a fixed prime different from $p$, and we fix an isomorphism $\iota:\overline{\mathbb{Q}}_\ell\cong \mathbb{C}$

\item Let $G$ be a connected reductive group over $F$.  We let $\hat{G}$ denote the dual reductive group of $G$ defined over $\overline{\mathbb{Q}}_\ell$.

\item We fix a $\Gamma$-stable pinning of $G_{\overline{F}}$, with Borel $B_{\overline{F}}$ and maximal torus $T_{\overline{F}}$.
We fix a maximal split torus $A_0$ of $G$ and minimal parabolic $P_0$ containing $A_0$, such that $A_{0,\overline{F}} \subset T_{\overline{F}}$, $B_{\overline{F}}\subset P_{0,\overline{F}}$.
We use $A_0$ and $P_0$ to define positivity of relative (co)roots and (semi)standardness of Levis and parabolics of $G$, and $T_{\overline{F}}$ and $B_{\overline{F}}$ to define the same notions in $G_{\overline{F}}$.
If $G$ is quasisplit we will write $B$ for $P_0$, and if $G$ is split we will furthermore write $T$ for $A_0$.
If $M$ is a standard Levi subgroup of $G$ then it inherits the maximal split torus $A_0$ from $G$ as well as the minimal standard parabolic $P_0\cap M$.

\item  We let $\Delta_G$ denote the set of simple (relative) roots of $G$ with respect to $(A_0,P_0)$.  
The simple roots of $(T_{\overline{F}}, B_{\overline{F}})$ will be denoted $\Delta_{G, \text{abs}}$.  

\item Let $W_G$ be the relative Weyl group of $G$, ie $N_G(A_0)(F)/C_G(A_0)(F)$.

\item $^LG$ denotes the Langlands dual of $G$, $\hat{G}\rtimes W_F$, defined with respect to the pinning of $G$ fixed above.
We denote by $Q$ the quotient of $W_F$ by the stabilizer of its action on $\hat{G}$.

\item We follow the convention introduced in \cite[Section VIII.3.1]{FS}, where by a parabolic or Levi subgroup of $^L G$, we refer specifically to a subgroup of the form $\hat{P}\rtimes W_F$ or $\hat{M}\rtimes W_F$, respectively, where $\hat{P}$ is the dual of a standard parabolic $P$ of the unique quasi-split inner form of $G$, and similarly $\hat{M}$ is the dual of a standard Levi $M$ of that form.  The inclusion of $^LM$ into $^LG$ will be denoted $\iota_M$.

\item If $G$ is a reductive group with split torus $T$ we let $X^*(T)$ and $X_*(T)$ be the group of characters and cocharacters of $T$, respectively.  

\item The category of finite dimensional (algebraic) representations of a reductive group will be denoted $\text{Rep}(G)$.  The set of irreducible representations of $G$ will be denoted $\text{Irr}(G)$.

\item Let $B(G)$ be the Kottwitz set of $G$, defined as the set of Frobenius-conjugacy classes of $G(\breve{F})$ where $\breve{F}$ is the maximal unramified extension of $F$ contained in $\overline{F}$.  
We denote the Newton map by $\nu: B(G)\to (X_*(T_{\overline{F}})^+_{\mathbb{Q}})^\Gamma$ and the Kottwitz map by $\kappa: B(G)\to \pi_1(G)_\Gamma$ (where $\pi_1$ is the Borovoi fundamental group).

\item Recall that the Newton and Kottwitz maps together uniquely determine the elements of $B(G)$, and that there is a subset of so called \textit{basic} elements, $B(G)^{\text{bas}}\subset B(G)$, which is in bijection with $\pi_1(G)_\Gamma$ via the Kottwitz map (and defined as the set of elements of $B(G)$ whose image under the Newton map is a cocharacter with central image in $G$).

\item For every Levi subgroup $M$ of $G$ we denote $i_M:B(M)\to B(G)$ the natural map.  If $b_M$ is a basic element of $B(M)$ we use the definitions of dominance and anti-dominance appearing in \cite[Definition 2.2.8]{HHS}.

\item  If $G$ is quasi-split with Borel $B$ and $U$ is the unipotent radical of $B$, we fix a nondegenerate character $\psi:U(F)\to \overline{\mathbb{Q}}_\ell$.
We will usually abbreviate the Whittaker data $(U,\psi)$ as simply $\psi$ ($B$, and hence $U$, having already been fixed).
Fix a $\Gamma$-stable pinning of $G_{\overline{F}}$.
We say $\psi$ is \textit{compatible} with the $\Gamma$-stable pinning if it may be expressed in the form $$\psi_F(\sum\limits_{\alpha\in\Delta_{G,\text{abs}}}x_\alpha)$$ where $\psi_F$ is a character of $F$ and $x_\alpha$ denotes the isomorphism between $U_\alpha$ and $\mathbb{A}^1_{\overline{F}}$ specified by the pinning.
We will always take $\psi$ and the $\Gamma$-stable pinning of $G$ to be compatible.
For each $w\in W_{G_{\overline{F}}}$ let $n_w$ be the Tits representative of $w$ in $N(T_{\overline{F}})(\overline{F})$ (see for example \cite[Section 12]{JA}).
By \cite[Proposition 12.3]{JA}, if $\alpha$ and $\beta$ are simple absolute roots such that $w(\alpha)=\beta$ then $n_w(n_\alpha) = n_\beta$ and $n_w(x_\alpha) = x_\beta$.
If $M$ is a standard Levi subgroup of $G$ then we denote the Whittaker datum $(U\cap M, \psi|_{U\cap M})$ by $\psi_M$.
It follows that $n_w(\psi_M) = \psi_{w(M)}$ for all $w\in W_{G_{\overline{F}}}$ mapping $\Delta_M$ to $\Delta_{w(M)}$ whenever $\psi$ is compatible with a chosen $\Gamma$-stable pinning of $G_{\overline{F}}$.

\item  If $G$ is a linear algebraic group we will denote by $\mathcal{U}_G$ the unipotent variety of $G$, and by $\mathcal{N}_G$ the nilpotent cone of its Lie algebra.

\item When $G$ is an algebraic group we will sometimes use the notation $g\in G$ to denote a closed point of $G$, and refer to $g$ as an "element" of $G$.

\end{itemize}

\subsection{Acknowledgements}

First and foremost, I would like to thank my advisor, Kiran Kedlaya, for providing an enormous amount of support throughout my time working on this project and more broadly as a graduate student.  
The author also owes the existence of this paper to David Hansen, whose suggestion to formalize the ideas about generic parameters in his Beijing lecture notes on the categorical local Langlands correspondence was the genesis of this project.
Section \ref{HeckEigenSect} came about in response to some questions Alexander Bertoloni Meli asked the author at the Categorical Local Langlands Summit in Seoul in June 2026.
We thank Alexander Bertoloni Meli and Linus Hamann for helpful conversations about an earlier draft of this paper.

\textbf{AI Disclosure:}
AI was used during the research phase of this project for help finding references, in the writing phase for creating tikz diagrams, and in the proofreading stage for help finding errors and improving the quality of exposition.
AI contributed the main idea to the proof of Lemma \ref{FinGenPres}, which was introduced to fix an oversight in the original paper draft, as well as Lemmas \ref{AllEigenReg} and \ref{SocIrred}.

\section{Technical foundations}

\subsection{L-Parameters}

On the spectral side of the local Langlands correspondence, the fundamental object of study is the L-parameter.  In this paper, we are exclusively interested in the study of L parameters over $\overline{\mathbb{Q}}_\ell$, which we will often drop from notation.

\begin{defn}
An \textit{L-parameter} is a continuous homomorphism $\phi:W_F\to\!{}^LG(\overline{\mathbb{Q}}_\ell)$ such that $\pi_2\circ\phi(w)=w$ for all $w\in W_F$, where $\pi_2:\!{}^LG\to W_F, (g, w)\mapsto w$ is the natural projection.  Here $^LG(\overline{\mathbb{Q}}_\ell)$ is given the product topology of $\hat{G}(\overline{\mathbb{Q}}_\ell)$ and $W_F$, and $\hat{G}(\overline{\mathbb{Q}}_\ell)$ is given the $\ell$-adic topology inherited from any inclusion $\hat{G}\hookrightarrow \text{GL}(n)$.
\end{defn}

In the context of the categorical local Langlands program we will be interested in the moduli stack of $L$ parameters:

\begin{defn}
We denote by $\mathrm{Par}_G$ the moduli stack of $L$ parameters up to $\hat{G}$-conjugacy, as defined in \cite[Chapter VIII]{FS}, \cite{DHKM}, and \cite{ZhuLParam}.
Its coarse moduli space will be denoted $\text{Par}_G^{\text{spec}}$.
\end{defn} 

If $\phi$ is an L-parameter, we will often abuse notation and also use $\phi$ to denote its $\hat{G}$-conjugacy class, i.e. the point in $\text{Par}_G$ it maps to.
Every property of L-parameters to be studied in this paper will be invariant under conjugation, so this should not cause confusion.

The action of $W_F$ on $\hat{G}$ factors through a finite quotient $Q$.  
In this way we can consider an L parameter to be a map from $W_F$ to $\hat{G}(\overline{\mathbb{Q}}_\ell)\rtimes Q$ (compatible with projection onto the second factor).  

\begin{defn}\label{ssLparam}
An L-parameter is \textit{semisimple} if $\phi(w)$ is a semisimple element of $\hat{G}\rtimes Q$ for all $w\in W_F$.  It is \textit{Frobenius-semisimple} if $\phi(w)$ is semisimple for some (equivalently any) element $w\not\in I_F$.
    
\end{defn}

\begin{prop}\label{ssLparamprop}
The following conditions on an L-parameter $\phi$ are equivalent:

1)  $\phi$ is semisimple.

2)  $\phi$ is Frobenius-semisimple and has open kernel.

3)  Whenever $\phi$ factors through a parabolic subgroup of $^LG$, it factors through a Levi subgroup of that parabolic.

4)  The Zariski closure of $\phi(W_F)$ is reductive.

In particular, if $\phi$ is semisimple then its centralizer in $\ ^LG$ is (possibly disconnected) reductive.
\end{prop}

\begin{rem}
As in \cite[VIII.3.1]{FS} we define a parabolic subgroup of $^LG$ to be a subgroup of the form $\hat{P}\rtimes Q$ where $P$ is a standard parabolic of the quasi-split inner form of $G$; every parabolic of $\hat{G}\rtimes Q$ surjecting onto $Q$ is of this form up to $\hat{G}$-conjugacy.
\end{rem}

\begin{proof}
$1)\implies  2)$: Follows immediately from Grothendieck's $\ell$-adic monodromy theorem and the fact that $I_F$, being compact, has finite (hence semisimple) image if its kernel is open.

$2)\implies 3)$: Let $V$ be a $\overline{\mathbb{Q}}_\ell$-vector space such that $\hat{G}\rtimes Q$ embeds into $GL(V)$, and fix such an embedding.  Since $\mathrm{Fr}$ normalizes $I_F$ and $\phi(I_F)$ has finite image, 
$\phi(\mathrm{Fr})^k$ centralizes $\phi(I_F)$ for some sufficiently large integer $k$.  Since $\phi(I_F)$ is a finite group it acts completely reducibly on $V$, and therefore so does $\langle \phi(I_F), \phi(\mathrm{Fr})^k\rangle$.  
We claim $W_F$ must act semisimply on $V$, which follows from a Maschke-style argument.  
Namely if $W$ is a $W_F$-invariant subspace of $V$, then it has $\langle \phi(I_F), \phi(\mathrm{Fr})^k\rangle$ invariant complement $W'$.  
Letting $p$ be the projection from $V$ to $W$ along $W'$, one checks $$v\mapsto\frac{1}{k}\sum\limits_{i=0}^{k-1}\text{Fr}^i\cdot p(\text{Fr}^{-i}\cdot v)$$ is a $W_F$-equivariant projection from $V$ onto $W$.  
We conclude by applying \cite[Corollaire 4.3]{SerCompRed} to the Zariski closure of $\phi(W_F)$ in $\hat{G}\rtimes Q$.

$3) \implies 4)$:  Let $H$ be the Zariski closure of $\phi(W_F)$ in $\hat{G}\rtimes Q$.  Then $H$ is a completely reducible subgroup of $\hat{G}\rtimes Q$, and by \cite[Proposition 4.2]{SerCompRed} it is reductive.

$4)\implies 1)$:  Note $\phi(W_F)$ is always solvable, which follows because $\phi(P_F)$ is a $p$-group, hence solvable, combined with the tame inertia and Weil group short exact sequences.  So $H$ is also solvable.  Then, being solvable and reductive, its identity component is a torus $T$.  Since $H$ has finitely many connected components, for any $w\in W_F$, there is an integer $k$ such that $\phi(w)^k\in T$.  But if $\phi(w)^k$ is semisimple then $\phi(w)$ must be semisimple as well.

\end{proof}

The relevance of semisimple L-parameters comes from the fact that their conjugacy classes correspond to the closed points of $\text{Par}_G$ (and therefore its coarse moduli space):

\begin{prop}\cite[Proposition VIII.3.2]{FS}, \cite[Theorem 4.13]{DHKM}\label{ParGClosedOrbits}

An L-parameter $\phi$ is semisimple if and only if its orbit under conjugation by $\hat{G}$ is a closed point of $\textup{Par}_G$.
    
\end{prop}

The tame inertia group $I_F^t$ of $F$ is canonically isomorphic to $$\prod\limits_{\ell\neq p}\mathbb{Z}_\ell(1) $$ and from now on we fix an isomorphism between this product and $$\prod\limits_{\ell\neq p}\mathbb{Z}_\ell$$ in order to define $t: I_F^t\to\mathbb{Z}_\ell$ as the composition of this isomorphism with the projection to $\mathbb{Z}_\ell$.

By Grothendieck's $\ell$-adic monodromy theorem there exists an L-parameter $\tilde\phi$ (depending on $\text{Fr}$ and $t$) with open kernel such that $\phi|_{I_F}(w) = \tilde\phi|_{I_F}(w)e^{t(w)N}$ for all $w\in I_F$ and some nilpotent element $N\in \mathfrak{g}^{\text{ad}(\phi(I_F))}$ such that $\phi(\text{Fr})\cdot N= q^{-1}N$.
It is defined by $\tilde\phi(\text{Fr}^nw) = \phi(\text{Fr})^n\phi(w)_{\text{ss}}$ for $w\in I_F$.

\begin{lem}\label{uNcommutation}

Let $u$ be the unipotent component of $\phi(\textup{Fr})$.  Then $u\in C_{\hat{G}\rtimes Q}(\phi(I_F))$, and it is independent of the choice of $\textup{Fr}$.
    
\end{lem}

\begin{proof}

It suffices to show $u$ centralizes $\tilde \phi(I_F)$ and $N$.
The former follows because $\tilde\phi(I_F)$ is finite, hence has finite automorphism group.
So $u^n$ commutes with $\tilde\phi(I_F)$ for some positive integer $n$, and since $u$ is unipotent it follows $u$ commutes with $\tilde\phi(I_F)$ as well.
Regarding $N$, since $\text{ad}(\phi(\text{Fr}))\cdot N = q^{-1}N$ it follows that $N$ is an eigenvector of $\text{ad}(\phi(\text{Fr}))$ acting on $\mathfrak{g}$, so the semisimple part of $\phi(\text{Fr})$ acts on $N$ by multiplication by $q^{-1}$ and the unipotent part acts trivially.
Therefore $u$ commutes with $e^{t(w)N}$ for all $w\in I_F$.

For the final claim, we start by showing $\phi(\text{Fr})_{\text{ss}}\phi(w)$ is semisimple for all $w\in I_F$.
Again since $\tilde\phi(I_F) = \phi(I_F)_{\text{ss}}$ is finite, some power $\phi(\text{Fr})_{\text{ss}}^n$ commutes with $\tilde\phi(w)$ for all $w\in I_F$, and hence $\phi(\text{Fr})_{\text{ss}}^n\tilde\phi(w)$ is semisimple for all such $w$.
Since $(\phi(\text{Fr})_{\text{ss}}\tilde\phi(w))^n = \phi(\text{Fr})_{\text{ss}}^n\tilde{\phi}(w')$ for some $w'\in I_F$, which is semisimple, it follows that $\phi(\text{Fr})_{\text{ss}}\tilde\phi(w)$ is as well.
Now $\text{ad}(\phi(\text{Fr})_{\text{ss}}\tilde\phi(w))\cdot e^{t(w)N} = e^{q^{-1}t(w)N}$, and it follows that $e^{kt(w)N} \phi(\text{Fr})_{\text{ss}}\tilde\phi(w) e^{t(w)N}e^{-kt(w)N} = \phi(\text{Fr})_{\text{ss}}\tilde\phi(w)$ where $k = 1/(1-q)$.
So $\phi(\text{Fr})_{\text{ss}}\tilde\phi(w) e^{t(w)N}$ is semisimple as well.

Now any two choices of $\text{Fr}$ and $\text{Fr}'$ of Frobenius lift differ by an element $w\in I_F$.
As shown above $\phi(\text{Fr})_{\text{ss}}\phi(w)$ is semisimple for all $w\in I_F$, and together with $u$ lying in the centralizer of $\phi(I_F)$ we conclude that the Jordan decomposition of $\phi(\text{Fr}')$ is $\phi(\text{Fr})_{\text{ss}}\phi(w)u$.

\end{proof}

\begin{defn}\label{Semisimplification}
Given $\phi$ we define its \textit{Frobenius-semisimplification} $\phi'$ by $\phi'(\text{Fr})= \phi(\text{Fr})_{\text{ss}}$ and $\phi'(w)=\phi(w)$ for all $w\in I_F$.
We define its \textit{semisimplification} $\phi_{\text{ss}}$ to be the Frobenius-semisimplification of $\tilde\phi$.  
We will show below that it lies in the closure of the $\hat{G}$-orbit of $\phi$.
\end{defn}

Let $q:\text{Par}_G\to\text{Par}_G^{\text{spec}}$ be the coarse moduli map.  
We will sometimes abuse notation and write $q(\phi)$ to denote $q$ applied to the $\hat{G}$-conjugacy class of an L-parameter $\phi$.  

\begin{prop}\label{fibstruct}
Fix a semisimple L-parameter $\phi$. The reduced fiber of $q$ over $q(\phi)\in \mathrm{Par}_G^{\mathrm{spec}}$ is isomorphic to the subvariety of $\mathcal{U}_{S_{\phi}}\times\hat{\mathfrak{g}}^{\mathrm{ad}\phi(I_{F})}$ consisting of pairs $(u,N)$ satisfying the relations $\mathrm{ad}\phi(\mathrm{Fr})\cdot N=q^{-1}N$ and $\mathrm{ad}(u)\cdot N=N$, (stack-)quotiented by the diagonal conjugation action of $S_\phi$.
\end{prop}

\begin{proof}

Given a closed point $(u,N)$ as in the proposition statement define an L-parameter $\phi_{u,N}$ by $\phi_{u,N}(\text{Fr}^nw) = (\phi(\text{Fr})u)^n\phi(w)e^{t(w)N}$. 
Let $D$ be the Zariski closure of $\langle\phi(\mathrm{Fr})\rangle$, a diagonalizable group commuting with $u$, whose identity component centralizes $\phi(I_F)$ (as some power of
$\phi(\mathrm{Fr})$ does). 
Since $\mathrm{ad}(\phi(\mathrm{Fr}))N = q^{-1}N$ and $q^{-1}$ is not a root of unity, $N$ is a weight vector for a character of $D$ of infinite order, so some cocharacter $\lambda$ of $D^\circ$ satisfies $\mathrm{Ad}(\lambda(z))N = z^mN$ with $m > 0$. 
Letting $S_\phi$ be the stabilizer of $\phi$ in $\hat{G}$, it follows that $\lambda$ is $S_\phi$-valued, commutes with $u$, and that $\lim_{z\to 0}\lambda(z)\cdot\phi_{u,N} = \phi_u$.
Now we apply Jacobson-Morozov to $u=\phi(\text{Fr})_u\in S_{\phi}^\circ$ to find a $S_{\phi}^\circ$ valued cocharacter $\lambda'$ such that $\lim\limits_{t\to0}\lambda'(t)\cdot u = I$, the identity.
It follows that $\lim\limits_{t\to0}\lambda'(t)\cdot \phi_u = \phi$.
Therefore $\phi$ is in the closure of the orbit of $\phi_{u,N}$.

Now let $R$ be a reduced finite type $\overline{\mathbb{Q}}_\ell$-algebra and let $u$ and $N$ be $S:=\text{Spec}(R)$-valued points of $\mathcal{U}_{S_\phi}$ and $\hat{\mathfrak{g}}^{\text{ad}\phi(I_F)}$ satisfying the hypotheses of the proposition.
Denote by $\Phi$ the corresponding $S$-valued parameter.
Letting $\Phi_s$ be the specialization of $\Phi$ to $s$ for any closed point $s\in S$, the closure of the orbit of $\Phi_s$ contains $\phi$.
Therefore $S$ lands in $q^{-1}(q(\phi))$.

Conversely, consider an $S$-point $\Phi$ of $q^{-1}(q(\phi))$.
Then for every $\overline{\mathbb{Q}}_\ell$-point $s\in S$, $\Phi_s(\text{Fr})_{\text{ss}}$ lives in a fixed $\hat{G}(\overline{\mathbb{Q}}_\ell)$-conjugacy class in $\hat{G}(\overline{\mathbb{Q}}_\ell)\rtimes Q$ as $s$ varies, and in particular its eigenvalues are constant after fixing any embedding of $\hat{G}\rtimes Q$ into $\text{GL}(V)$.
Therefore there exists an $S$-point $\Phi(\text{Fr})_{\text{ss}}$ of $\hat{G}\rtimes Q$ restricting to $\Phi_s(\text{Fr})_{\text{ss}}$ for all $s$, and consequently a Jordan decomposition of $\Phi(\text{Fr})$ at the level of $S$-points.

By the relative Grothendieck $\ell$-adic monodromy theorem \cite[Lemma 3.10]{ZhuLParam}, $\Phi|_{I_F}$ can be written as a product $\Phi = \tilde\Phi e^{t(-)N}$ for some $S$-point $N$ of $\mathfrak{g}$ such that $\tilde\Phi$ has open kernel, $\text{ad}(\tilde\Phi(I_F))\cdot N = N$ and $\text{ad}(\tilde\Phi(\text{Fr}))\cdot N = q^{-1}N$.
By Lemma \ref{uNcommutation} $u$ and $N$ commute after restriction to all closed points of $S$, and by the hypotheses on $S$ it follows that they commute.
We can therefore define the semisimplification $\Phi_{\text{ss}}$ of $\Phi$ as in Definition \ref{Semisimplification}.
Possibly replacing $S$ by an etale cover if necessarily, we may trivialize the $S_\phi$-torsor $\Phi_{\text{ss}}$, so that we may assume $\Phi_{ss}$ is the constant map from $S$ to $\phi$.
For a closed point $s\in S$ let $u_s$ and $N_s$ denote the restrictions of $u$ and $N$ to $s$.
Then $u_s$ and $N_s$ satisfy the conditions of the proposition for all $s$, and hence $u$ and $N$ do as well.
Then $\phi$ lies in the closure of the orbit of $\phi_s$ for all $s$.

\end{proof}

 \subsection{Review of Classical Local Langlands}\label{ClassicalLLC}

In this section we review some aspects of the local Langlands correspondence.  
For more detailed overviews see \cite{KalLLC}, \cite{Tai}.

Let $G$ be a reductive group over a finite extension $F$ of $\mathbb{Q}_p$, and let $\ell$ be a prime different from $p$.  
The goal of the local Langlands correspondence is, roughly, to establish a bijection between smooth irreducible $\overline{\mathbb{Q}}_\ell$-representations of $G(F)$, and L-parameters (up to conjugacy) equipped with additional data.
After fixing an isomorphism between $\overline{\mathbb{Q}}_\ell$ and $\mathbb{C}$ we may equivalently consider smooth irreducible complex representations of $G(F)$ on the representation-theoretic side, and complex Weil-Deligne representations $W_F\times SL_2(\mathbb{C})\to\ ^LG(\mathbb{C})$ on the Galois-theoretic side.
When $G$ is not quasi-split, in ideas originally going back to work of Vogan \cite{Vog93} it was realized that on the representation-theoretic side it makes sense to consider representations of $G(F)$ not in isolation but together with those of various inner forms of $G$.
The form of this conjecture most relevant to us is the following for tempered representations due to Kottwitz:
\begin{conjecture} \cite[Conjecture F]{KalLLC}\label{basicLLC}
There is a canonical bijection $\iota_\psi$ between $\bigsqcup\limits_{b\in B(G)^{\textup{bas}}}\textup{Irr}(G_b(F))$ and pairs $(\phi,\rho)$ making the diagram below commute, where $\phi$ is a Weil-Deligne parameter and $\rho$ is an irreducible representation of $S_\phi^\natural:=S_\phi/(S_\phi\cap\hat{G}^{\text{der}})^\circ$, where $S_\phi$ is the stabilizer of $\phi$ in $\hat{G}$.  
Fixing $b$ and $\phi$, there are finitely many smooth irreducible representations of $G_b(F)$ with L-parameter $\phi$.  We denote this set $\Pi_\phi(G_b)$, which will be referred to as the L-packet of $\phi$ for $G_b$.

\[
\xymatrix{\underset{b\in B(G)_{\mathrm{bas}}}{\coprod}\Pi_{\phi}(G_{b})\ar[r]^{\,\;\;\;\;\;\;\;\iota_{\psi}}\ar[d] & \mathrm{Irr}(S_{\phi}^{\natural})\ar[d]\\
B(G)_{\mathrm{bas}}\ar[r]^{\kappa_{G}} & X^{\ast}(Z(\hat{G})^{\Gamma})
}
\]

\end{conjecture}

\begin{rem}

In \cite{KalLLC} the local Langlands correspondence is restricted to tempered representations and parameters, but by the proof of \cite[Theorem 2.5]{BMHN} this restriction can be removed (though the restriction to tempered packets is needed for other properties to hold, such as the endoscopic character identities, which will not be used in this paper).
     
\end{rem}

If $\phi$ is an L-parameter, we will use the notation $\Pi_\phi(G_b)$ to denote the L-packet of the Weil-Deligne parameter associated to $\phi$.

In later work, Bertoloni Meli and Oi extended this result as follows:
\begin{thm}\cite[Theorem 1.2]{BMO}\label{BMO}
Assume Conjecture \ref{basicLLC} holds for all Levi subgroups of a reductive group $G$.  Then there is an induced canonical bijection between $\bigsqcup\limits_{b\in B(G)}\textup{Irr}(G_b(F))$ and pairs $(\phi,\rho)$ where $\rho$ is an irreducible representation of $S_\phi$, making the diagram below commute.  The set of irreducible representations of $G_b(F)$ with Weil-Deligne parameter $\phi$ is denoted $\Pi_\phi(G_b)$.

\[
\xymatrix{\underset{b\in B(G)}{\coprod}\Pi_{\phi}(G_{b})\ar[r]^{\;\;\;\;\;\;\iota_{\psi}}\ar[d] & \mathrm{Irr}(S_{\phi})\ar[d]\\
B(G)\ar[r]^{\kappa_{G}} & X^{\ast}(Z(\hat{G})^{\Gamma})
}
\]

\end{thm}

In general $\Pi_\phi(G_b)$ is not an L-packet itself, but a finite union over the L-packets of $G_b$ which map to $\phi$ under the map from $\text{Par}_{G_b}$ to $\text{Par}_G$ induced by the inclusion $\ ^LG_b\hookrightarrow\ ^LG$.

\begin{rem}
We have not defined what it means for the above bijections to be canonical, and indeed this is a subtle question.
We content ourselves below to merely reviewing the properties of the correspondence which will be relevant to later sections.
\end{rem}
\ 
\newline

It is conjectured that the local Langlands correspondence should satisfy a number of desirable properties.
Before going further we introduce some definitions.

\begin{defn}
A Weil-Deligne parameter $\phi$ is \textit{tempered} or \textit{bounded} if the image of $\phi|_{W_F}$ in $\ ^LG$ is bounded after projecting to $\hat{G}$.
We say $\phi$ is \textit{essentially discrete} if it does not factor through any proper Levi subgroup of $\ ^LG$.
It is \textit{supercuspidal} if it is essentially discrete and its restriction to $\text{SL}_2$ is trivial.
\end{defn}

For later use we translate some of these conditions to the setting of L-parameters:
If $\phi$ is an L-parameter, then $\phi$ is essentially discrete if and only if it does not factor through any proper Levi subgroups of $^LG$.
It is supercuspidal if it is discrete and has open kernel.

For our specific needs we highlight the following properties of the classical local Langlands correspondence.
\begin{property}\label{ClassLLCProp}
\

\begin{enumerate}

\item $\phi$ is tempered if and only if $\bigsqcup\limits_{b\in B(G)}\Pi_\phi(G_b)$ contains a tempered representation if and only if every element of $\bigsqcup\limits_{b\in B(G)}\Pi_\phi(G_b)$ is tempered.  

\item $\phi$ is essentially discrete if and only if $\bigsqcup\limits_{b\in B(G)^{\text{bas}}}\Pi_\phi(G_b)$ contains an essentially discrete series representation if and only if every element of $\bigsqcup\limits_{b\in B(G)^{\text{bas}}}\Pi_\phi(G_b)$ is essentially discrete series.
Recall an essentially discrete series representation is a discrete series tensored with an unramified character.

\item $\phi$ is supercuspidal if and only if every element of $\bigsqcup\limits_{b\in B(G)^{\text{bas}}}\Pi_\phi(G_b)$ is supercuspidal.
It is not necessarily true that $\Pi_\phi(G_b)$ contains a supercuspidal representation if and only if every element is supercuspidal.

\item Define $\phi^{\text{inf}}:W_F\to\! ^LG$ by $\phi^{\text{inf}}(w):=\phi(w,\begin{pmatrix} ||w||^{1/2} & 0 \\ 0 & ||w||^{-1/2} \end{pmatrix})$, and recall $\iota_M$ is the inclusion of $\ ^LM$ into $\ ^LG$.  If $P$ is a parabolic subgroup of $G$ with Levi $M$, $\sigma$ is an irreducible representation of $M(F)$ with Weil-Deligne parameter $\phi_M$, and $\pi$ is an irreducible subquotient of the normalized parabolic induction $i_P^G(\sigma)$ with Weil-Deligne parameter $\phi$, then $\iota_M\circ\phi_M^{\text{inf}}=\phi_G^{\text{inf}}$.

The translation of this condition to the setting of L-parameters is more straightforward: If $\sigma$ and $\pi$ as above have L-parameters $\phi_M$ and $\phi$, then $\iota_M\circ\phi_M^{\text{ss}} = \phi^{\text{ss}}$.

When there is a compatibility for $G$ between the classical local Langlands correspondence and the Fargues-Scholze semisimple parameter $\phi^{\text{FS}}$ defined below, in the sense that $\phi_\pi^{\text{inf}}=\phi_\pi^{\text{FS}}$ for all irreducible representations $\pi$ of $G_b(F)$ and all $b$, then property 4 follows by \cite[Corollary IX.7.3]{FS}.

\end{enumerate}
\end{property}

\subsection{Review of Fargues-Scholze Theory}\label{FSSect}

In this section we briefly review Fargues and Scholze's categorification of the local Langlands correspondence \cite{FS}.  
To set notation, let $F$ be a finite extension of $\mathbb{Q}_p$, and let $X_F$ be the Fargues-Fontaine curve over $F$ \cite[Definition II.1.15]{FS}.  
Denote by $\mathrm{Bun}_G$ the moduli stack of $G$ bundles on $X_F$ \cite[Definition III.1.2]{FS}.  
Recall that $|\mathrm{Bun}_G|$ is homeomorphic to the Kottwitz set $B(G)$ of $G$ \cite{Vie}.  
Fargues and Scholze build a theory of $\ell$-adic sheaves on $\text{Bun}_G$ which works for general $\mathbb{Z}_\ell$-algebra coefficients, but we will restrict ourselves to exclusively working with $\overline{\mathbb{Q}}_\ell$ coefficients.
They construct a four functor formalism $D^{\text{lis}}(\ ^*,\ _*, \otimes, \mathscr{H}\!\text{om})$ of $\ell$-adic sheaves on Artin v-stacks \cite[Chapter VII.6]{FS}, and show \cite[Proposition VII.7.1]{FS} that for all $b\in B(G)$, the subcategory of sheaves in $D(\mathrm{Bun}_G,\overline{\mathbb{Q}}_\ell)$ supported at $b$ is equivalent to $D(G_b,\overline{\mathbb{Q}}_\ell)$, the (derived) category of smooth $\overline{\mathbb{Q}}_\ell$ representations of $G_b(F)$.
They also construct, by hand, a left adjoint $i_{b\sharp}$ to $i_b^*$ \cite[Proposition VII.7.2]{FS}, and Hansen \cite{Han21} constructs $i_{b!}$ and $i_b^!$, where $i_b$ is the inclusion of $b$ into $\text{Bun}_G$.
We will use the superscript $^\text{ren}$ on functors of the form $i^{?}$ and $i_?$ to denote the renormalized functors introduced in \cite[Section 1.1]{BeijingNotes}.
Linus Hamann pointed out to the author that the identification of $D(\text{Bun}_G^b,\overline{\mathbb{Q}}_\ell)$ with $D(G_b(F), \overline{\mathbb{Q}}_\ell)$ implicitly depends on lift of $b$ to $G(\breve{F})$, which we therefore fix for all time and all $b$.

In \cite[Section 4.1]{HM} Hansen and Mann, building off Scholze's theory of Berkovich motives \cite{SchBerk}, study a six functor formalism $D^{\text{rel}}$ for small v-stacks over $\text{Spd}\ \overline{\mathbb{F}}_p$.  
They prove standard equivalences $D^{\text{rel}}(\text{Bun}_G,\overline{\mathbb{Q}}_\ell)\cong D^{\text{lis}}(\text{Bun}_G,\overline{\mathbb{Q}}_\ell)$ \cite[Theorem 4.2.1]{HM}, and $D^{\text{rel}}(\text{Bun}^b_G,\overline{\mathbb{Q}}_\ell)\cong D(G_b(F),\overline{\mathbb{Q}}_\ell)$ \cite[Example 4.1.6, Theorem 4.1.9]{HM}, and therefore $D(\text{Bun}_G,\overline{\mathbb{Q}}_\ell)$ inherits the same functors $i_{b\sharp}$, $i_{b!}$, $i_{b*}$, $i_b^*$, and $i_b^!$ relating $D^{\text{lis}}(\text{Bun}_G,\overline{\mathbb{Q}}_\ell)$ to $D(G_b(F),\overline{\mathbb{Q}}_\ell)$.
All of these results will be discussed in detail in forthcoming work of Gleason and Hansen \cite{GH}.
We will often shorten $D^{\text{rel}}(\text{Bun}_G,\overline{\mathbb{Q}}_\ell)$ to simply $D(\text{Bun}_G)$.
\newline

For every finite dimensional representation $V$ of $\mathrm{Rep}(\hat{G}\rtimes Q)$, Fargues and Scholze \cite[Section IX.2]{FS} construct a Hecke operator $T_V:D(\mathrm{Bun}_G)\to D(\mathrm{Bun}_G)^{BW_F}$, from which they construct a spectral action of $\mathrm{Perf}(\mathrm{Par}_G)$ on $D(\mathrm{Bun}_G)^\omega$ \cite[Corollary X.1.3]{FS}, which can be extended to an action of $\mathrm{Perf}(\mathrm{Par}_G)$ on $D(\mathrm{Bun}_G)$ since $D(\mathrm{Bun}_G)$ is compactly generated \cite[Proposition VII.7.4]{FS}.

When $G$ is quasi-split, then after fixing a Whittaker datum $(U,\psi)$ for $G$, the spectral action on the Whittaker sheaf $W_\psi:= i_{1!} c\text{-}\text{Ind}_{U(F)}^{G(F)}(\psi)$ gives a functor $a_\psi:\mathrm{QCoh}(\mathrm{Par}_G)\to D(\mathrm{Bun}_G)$.
Let $c_\psi$ be the right adjoint to $a_\psi$.  

\begin{conjecture}\cite[Conjecture X.1.4]{FS}\label{FSConj}

After restricting to compact objects, $c_\psi$ induces an equivalence $c_\psi:D(\mathrm{Bun}_G)^\omega \tilde\to\textup{Coh}^{b,\textup{qc}}(\mathrm{Par}_G)$.

\end{conjecture}

The ind-extension of $c_\psi|_{D(\text{Bun}_G)^\omega}$ is denoted $\textbf{L}_\psi:D(\text{Bun}_G)\tilde\to\text{IndCoh}(\text{Par}_G)$ \cite[Definition 1.2.4]{HM}.
If $\phi$ is any $\overline{\mathbb{Q}}_\ell$-point of $\text{Par}_G$, we denote by $D(\text{Bun}_G)^\phi$ the full subcategory of $D(\text{Bun}_G)$ taken by $\textbf{L}_\psi$ to sheaves on $\text{Par}_G$ supported set-theoretically on $q^{-1}(q(\phi))$.
In particular any sheaf of the form $i_{b!}\pi$ in $D(\text{Bun}_G)^\phi$ has Fargues-Scholze parameter $\phi^{\text{ss}}$.
If $U$ is an open substack of $\text{Par}_G$ then $D(\text{Bun}_G)^U$ denotes the full subcategory of $D(\text{Bun}_G)$ with objects those taken by $\textbf{L}_\psi$ to sheaves on $\text{Par}_G$ supported on $U$.
\newline

Using the Hecke action, Fargues and Scholze develop a theory of excursion operators in their setting and construct a map $\mathfrak{z}$ from the spectral Bernstein center $\Gamma(\mathrm{Par}_G^{\text{spec}})\subset\mathfrak{Z}(\text{QCoh}(\text{Par}_G))$ (here $\Gamma(-)$ denotes the classical ring of global functions) to the geometric Bernstein center $\mathfrak{Z}(D(\text{Bun}_G))$, 
which after composing with the projection $\mathfrak{Z}(D(\text{Bun}_G))\to \mathfrak{Z}(D(G_b(F)))$ and applying Proposition \ref{ParGClosedOrbits} associates (the conjugacy class of) a semisimple L-parameter $\phi_\pi^{\text{FS}}$ to any irreducible represent-ation $\pi$ of $G_b(F)$ \cite[Proposition IX.4.1]{FS}.  
The spectral action $\text{Perf}(\text{Par}_G)\to \text{End}(D(\text{Bun}_G))$
induces a map from $\Gamma(\mathrm{Par}_G^{\text{spec}})\cong\text{End}^0(\mathcal{O}(\text{Par}_G))$ to $\text{End}^0(\text{id}_{D(\text{Bun}_G)})=\mathfrak{Z}(D(\text{Bun}_G))$, 
which one can check using the description of $\Gamma(\mathrm{Par}_G^{\text{spec}})$ given in the proof of \cite[Theorem VIII.4.1]{FS} agrees with the map $\mathfrak{z}$ defined using excursion operators.  
Since $a_\psi(F):= F\star W_\psi$ it follows that $a_\psi$ intertwines the action of $\Gamma(\mathrm{Par}_G^{\text{spec}})$ on $\text{QCoh}(\text{Par}_G)$ with its action on $D(\text{Bun}_G)$ through $\mathfrak{z}$.  
So the right adjoint $c_\psi$ of $a_\psi$ also intertwines these two actions of $\Gamma(\mathrm{Par}_G^{\text{spec}})$.
Using this together with the isomorphisms $i_b^{*\text{ren}}i_{b*}^{\text{ren}}\cong i_b^{*\text{ren}}i_{b!}^{\text{ren}}\pi\cong i_b^{*\text{ren}} i_{b\sharp}^{\text{ren}}\cong \text{id}$ and the Schur-irreducibility of $i_{b*}^{\text{ren}}\pi$, $i_{b!}^{\text{ren}}\pi$ and $i_{b\sharp}^{\text{ren}}\pi$ when $\pi$ is irreducible one can check:
\begin{prop}\label{FSsup}
$c_\psi(i_{b*}^{\textup{ren}}\pi)$, $c_\psi(i_{b!}^{\textup{ren}}\pi)$, and $c_\psi(i_{b\sharp}^{\textup{ren}}\pi)$ lie in $D(\textup{Bun}_G)^{\phi_\pi^{\textup{FS}}}$ for all $(b,\pi)$. 
\end{prop}

\ 
\newline
There are two important finiteness conditions one can place on sheaves in $D(\text{Bun}_G)$, namely \textit{compactness} and \textit{ULA-ness}.
In \cite[Proposition VII.7.4]{FS} Fargues and Scholze show a complex $A\in D(\text{Bun}_G)$ is compact if and only if it has finite support and $i_b^*A$ is a compact object of $D(G_b(F))$ for all $b\in B(G)$, where compactness in $D(G_b(F))$ is equivalent to lying in the subcategory of $D(G_b(F))$ generated under shifts, finite colimits and retracts by the representations $c\text{-}\text{Ind}_K^{G_b(F)}(\overline{\mathbb{Q}}_\ell)$ as $K$ varies over the open pro-p subgroups of $G_b(F)$.

We will not define the ULA property here, but simply note that by \cite[Proposition VII.7.9]{FS} the ULA property for a complex $A$ is equivalent to requiring $(i_b^*A)^K$ be a perfect complex of $\overline{\mathbb{Q}}_\ell$ vector spaces for all $b\in B(G)$ and open pro-p subgroups $K$ of $G_b(F)$.
We say a complex of sheaves is \textit{finite} \cite[Definition 1.6.2]{BeijingNotes} if it is compact and ULA.
Using the above definitions one can check a complex $A$ in $D(\text{Bun}_G)$ is finite if and only if it is bounded with finite support and each cohomology sheaf $H^n(i_b^*A)$ is a finite length $G_b(F)$-representation.
On the spectral side, an object of $\text{Coh}(\text{Par}_G)$ is finite \cite[Conjecture 1.7.9]{BeijingNotes} if it is supported on finitely many fibers of the coarse moduli space map $q:\text{Par}_G\to\text{Par}_G^{\text{spec}}$.
\newline

There are two primary t-structures on $D(\text{Bun}_G)^\omega$ which will be relevant to us, the \textit{perverse} t-structure $(D^{p,\leq0}, D^{p,\geq0})$ \cite[Proposition 1.2.1]{BeijingNotes} and the \textit{hadal} t-structure $(D^{h,\leq0}, D^{h,\geq0})$ \cite[Theorem 1.2.3]{BeijingNotes}.
By definition $D^{p,\leq0}$, resp. $D^{p,\geq0}$ consists of complexes $A$ such that $i_b^{*\text{ren}}A\in D^{\leq0}(G_b(F))^\omega$, resp. $i_b^{!\text{ren}}A\in D^{\geq0}(G_b(F))^\omega$ for all $b\in B(G)$, and $D^{h,\leq0}$, resp. $D^{h,\geq0}$ consists of complexes generated under finite extensions by sheaves of the form $i_{b\sharp}^{\text{ren}}A$, resp. $i_{b!}^{\text{ren}}A$ for $A\in D^{\leq0}(G_b(F))^\omega$, resp. $D^{\geq0}(G_b(F))^\omega$.
We are implicitly using the fact that $i_{b!}^{\text{ren}}$ and $i_{b\sharp}^{\text{ren}}$ preserve compact objects, which for $i_{b!}^{\text{ren}}$ follows directly from the criterion for compactness \cite[Proposition VII.7.4]{FS} mentioned above, and for $i_{b\sharp}^{\text{ren}}$ follows because it is the left adjoint of a colimit-preserving functor.
We note the definition of the perverse t-structure does not use the compactness assumption, and in particular it defines a t-structure on all of $D(\text{Bun}_G)$.
For any $A\in D(G_b(F))^\omega$ define $i_{b\sharp!}^{\text{ren}}A:=\text{Im}(H^{h,0}(i_{b\sharp}^{\text{ren}}A)\to H^{h,0}(i_{b!}^{\text{ren}}A))$ in the hadal heart $\text{Had}(\text{Bun}_G)$, and $i_{b!*}^{\text{ren}}A:=\text{Im}(H^{p,0}(i_{b!}^{\text{ren}}A)\to H^{p,0}(i_{b*}^{\text{ren}}A))$ in the perverse heart $\text{Perv}(\text{Bun}_G)$.
\newline

An important tool in this paper will be the geometric Eisenstein series functor.  
In the Fargues-Scholze setting, its study began in Hamann's thesis \cite{HamEis} and continued in \cite{HHS}.
For $P$ a parabolic subgroup of an arbitrary reductive group $G$ with Levi $M$, the natural projection from $P$ to $M$ and inclusion of $P$ into $G$ induce a diagram $\text{Bun}_M\xleftarrow{q}\text{Bun}_P\xrightarrow{p}\text{Bun}_G$.
Then $\text{Eis}_P^G$ is by definition $p_!(q^*(-)\otimes \text{IC}_{\text{Bun}_P}): D(\text{Bun}_M)\rightarrow D(\text{Bun}_G) $, where $\text{IC}_{\text{Bun}_P}$ is the intersection cohomology complex introduced in \cite{HI}.

On skyscraper sheaves it is often possible to compute the geometric Eisenstein functors explicitly, but in order to do so we must review some background and introduce more notation.
Recall that the basic elements of $B(G)$ are parametrized by the elements $\alpha$ of the Borovoi fundamental group of $G$, $\pi_1(G)_\Gamma$, where $\pi_1(G)$ is by definition the cocharacter lattice modulo the coroot lattice.  
If $b\in B(G)$ is in image of $i_M:B(M)^{\text{bas}}\to B(G)$ then by \cite[Lemma 4.25]{HI} there is an injection from $i_M^{-1}(b)$ into a certain Weyl group $W_{M,b}$ defined in \cite[Section 5.3]{BMO}.  It takes $b_\alpha$, $\alpha\in \pi_1(G)_\Gamma$, to the unique $w_\alpha\in W_{M,b}$ such that $w_\alpha(b_\alpha)$ is dominant in $B(w_\alpha(M))$.  Writing $G_{\{b\}}$ for the centralizer of the Newton point of $b$ in $G$, so that $G_b$ is an inner form of $G_{\{b\}}$, by \cite[Lemma 4.26]{HI} there is a parabolic subgroup $Q_{b,w_\alpha}$ with Levi $w_\alpha(M)_{w_\alpha(b_\alpha)}$ of $G_b$ which is the transfer of $w_\alpha(P)\cap G_{\{b\}}$ under the twisting from $G_{\{b\}}$ to $G_b$.

\begin{prop} \cite[Proposition 2.2.12, Corollary 2.2.14]{HHS}\label{EisCalc}

For any complex $A$ in $D(M_{b_\alpha}(F))$, if $b_\alpha$ is antidominant then $ \textup{Eis}_P^Gi_{b_\alpha!}A\cong i_{b!}^{\textup{ren}}i_{Q_{b,w_\alpha}}^{G_b}((A)^{w_\alpha})$, and if $b_\alpha$ is dominant then $ \textup{Eis}_P^Gi_{b_\alpha!}A\cong i_{b\sharp}^{\textup{ren}}i_{Q_{b,1}}^{G_b}(A)$.

\end{prop}

There is also a geometric Eisenstein series functor on the spectral side of the correspondence, defined by $\text{Eis}_P^{\text{spec}}:=p^{\text{spec}}_*q^{\text{spec}*}: \text{Coh}(\text{Par}_M)\to \text{Coh}(\text{Par}_G)$, where $q^{\text{spec}}: \text{Coh}(\text{Par}_P)\to\text{Coh}(\text{Par}_M)$ and $p^{\text{spec}}: \text{Coh}(\text{Par}_P)\to \text{Coh}(\text{Par}_G)$ are the natural maps.
A fundamental assumption we will need for some of the main results of the paper is the compatibility of the two Eisenstein functors with $c_\psi$:

\begin{conjecture}\label{EisCompat}

There is an equivalence of functors $\textup{Eis}_{\overline{P}}^{\textup{spec}}\circ c_{\psi_M}\cong c_\psi\circ \textup{Eis}_P$, where $\overline{P}$ is the parabolic opposite to $P$.
\end{conjecture}

This will be proven in forthcoming work of Hamann, Hansen, and Mann \cite{HHM}.
\newline

\noindent Finally, we record a proposition we will need later:

\begin{prop}\label{FSWeylEquiv}

Let $G$ be a quasi-split reductive group with Levi subgroup $M$, and choose a pinning of $G$ along with a compatible Whittaker datum as in Section \ref{notation}.  Let $w\in W_G$ be such that $w(\Delta_M)= \Delta_{w(M)}$ and let $\hat w$ be its corresponding element in $W_{\hat G}^\Gamma$.  
Then $c_{\psi_{w(M)}}(w_*(A))\cong \hat w_*(c_{\psi_M}(A))$ for all $A\in D(\textup{Bun}_M)$.
    
\end{prop}

\begin{proof}

Without loss of generality we may assume $M$ is standard.  
Since $G$ is quasi-split, $C_G(A_0)$ is a maximal torus, so no absolute root of $G$ is trivial on $A_0$, and an absolute root is positive if and only if its restriction to $A_0$ is. 
Therefore the assumption $w(\Delta_M) = \Delta_{w(M)}$ implies that $w$ takes the positive, and hence simple, absolute roots of $M$ to those of $w(M)$.
Let $n_w$ be the Tits representative of $w$, which as reviewed in the Notation section has the property that $n_w(\psi_M) = \psi_{w(M)}$, and therefore $n_{w*}(i_{1!}W_{\psi_M})\cong i_{1!}W_{\psi_{w(M)}}$.

In this proof only, we will let $\hat{M}$ denote the dual group of $M$ (endowed with $W_F$ action) defined over $\mathbb{Z}_\ell[\sqrt{q}]$, rather than $\overline{\mathbb{Q}}_\ell$. 
Let $\Lambda = \mathbb{Z}/\ell^n\mathbb{Z}[\sqrt{q}]$, and as in \cite[Chapter VI]{FS} let $\check{M}$ denote the Tannaka group of $\text{Sat}_M(\Lambda)$.
By \cite[Theorem VI.11.1]{FS} there is a canonical $W_F$-equivariant isomorphism $\iota_M:\check{M}\cong\hat{M}$, and similarly an isomorphism $\iota_{w(M)}:\widecheck{w(M)}\cong\hat{w}(\hat{M})$.
We will show that the map $n_{w}^*:\text{Sat}_{w(M)}(\Lambda)\to \text{Sat}_{M}(\Lambda)$ coming from conjugation by $n_w$ intertwines with the map from $\text{Rep}(\hat{w}(\hat{M}))$ to $\text{Rep}(\hat{M})$ coming from conjugation by $n_{\hat{w}}$ under $\iota_M$.
Let $n_{\check{w}}:\check{M}\to\widecheck{w(M)}$ be the map induced by $n_w^*$.

As in the proof of \cite[Theorem VI.11.1]{FS} we note that if $T$ is a maximal torus of $M$ then $\text{Sat}_T(\Lambda)$ is identified with $X_*(T_{\overline{F}})$-graded objects of $\text{Rep}^{\text{cont}}(\Lambda)$ equipped with a $W_F$-action, which on the grading agrees with the usual action of $W_F$ on $X_*(T_{\overline{F}})$.
This equivalence between $\text{Sat}_T(\Lambda)$ and $\text{Rep}_{\Lambda}(\hat{T}\rtimes W_F)$ intertwines the action of $w$ on the former with the action of $\hat{w}$ on the latter.
Since $w$ takes simple absolute cocharacters of $M$ to simple absolute cocharacters of $w(M)$, it takes dominant absolute coroots of $M$ to dominant absolute coroots of $w(M)$.
For any dominant absolute cocharacter $\mu$ of $M$ define $A_\mu:=j_{\mu!}^p\Lambda[\langle2\rho_M,\mu\rangle]$, the standard object associated to $\mu$ in $\text{Sat}_M(\Lambda)$.
Then $n_{w}^*$ takes $A_{w(\mu)}$ to $A_\mu$.
The fiber functors of $\text{Sat}_{w(M)}(\Lambda)$ and $\text{Sat}_M(\Lambda)$ are respectively given by $\oplus H^i(\pi_{w(M)*})$ and $\oplus H^i(\pi_{M*})$, and  $\pi_{M*}n_w^*=\pi_{w(M)*}$.
By definition $ \check{B}_M$ is the stabilizer of the increasing degree filtration on the fiber functor of $\text{Sat}_M(\Lambda)$, and similarly $\check{B}_{w(M)}$ is the stabilizer of the increasing filtration on the fiber functor of $\text{Sat}_{w(M)}$.
So $n_{\check{w}}(\check{B}_M) = \check{B}_{w(M)}$ since $n_w^*$ preserves the degree filtration.

At the end of the proof of \cite[Theorem VI.11.1]{FS}, Fargues and Scholze construct canonical (Tate twisted) trivializations of the root spaces $U_{\hat{\alpha}}$ of $\check{M}$ by identifying each root space $U_{\hat{\alpha}}$ with the $\text{Hom}$-space between the 2nd and 0th cohomology groups of the flag variety of $M_{\alpha}^{ad}\cong\text{PGL}(2)$.
These trivializations are mapped to one another under the action of $n_{\check{w}}$ because they only depend on the choice of $T$ and $B$ (and $n_w$ takes $M_\alpha$ to $M_{w(\alpha)}$).  
The isomorphism $\iota_M$ is defined to take the canonical root vectors of $\check{M}$ to the root vectors of $\hat{M}$ coming from the beforehand-chosen (Tate twisted) pinning of $\hat{M}$.

So in summary $\iota_M$ and $\iota_{w(M)}$ are pinning-preserving, conjugation by $n_{\check{w}}$ and $n_{\hat{w}}$ take the pinnings of $\check{M}$ to $\widecheck{w(M)}$ resp. $\hat{M}$ to $\widehat{w(M)}$, and $(\iota_{w(M)}\circ n_{\check{w}})|_{\check{T}}=(n_{\hat{w}}\circ\iota_M)|_{\check{T}}$.
Given that a pinning-preserving isomorphism between reductive groups is determined by its restriction to the maximal tori of the pinning, it follows that $\iota_{w(M)}\circ n_{\check{w}}=n_{\hat{w}}\circ\iota_M$.

By \cite[Theorem 4.1.7]{HM} the geometric Satake equivalence with relative $\mathbb{Z}_\ell[\sqrt{q}]$-coefficients intertwines the action of $w$ and $\hat w$ as well.
As in the proof of \cite[Proposition 4.5.6]{HM}, by the equivalence $$\text{Rep}_{\mathbb{Z}_\ell[\sqrt{q}]}(\hat{M}\rtimes Q)^{\text{fpj}}\otimes _{\text{Rep}_{\mathbb{Z}_\ell[\sqrt{q}]}(Q)^{\text{fpj}} }\text{Rep}_{\overline{\mathbb{Q}}_\ell}(Q)^{\text{dbl}}\cong   \text{Rep}_{\overline{\mathbb{Q}}_\ell}(\hat{M}\rtimes Q)^{\text{dbl}} $$ (where fpj means ``finite projective,'' and dbl means ``dualizable,'' which in our setting simply means a bounded complex of finite dimensional representations) we can extend the Hecke functor, which takes $\text{Rep}_{\mathbb{Z}_\ell[\sqrt{q}]}(\hat{M}\rtimes Q)^{\text{fpj}}$ to  $$\text{End}_{D(\mathbb{Z}_\ell[\sqrt{q}])}(D(\text{Bun}_M,\mathbb{Z}_\ell[\sqrt{q}]))
\otimes_{D(\mathbb{Z}_\ell[\sqrt{q}])} D(\text{Div}_1, \mathbb{Z}_\ell[\sqrt{q}])$$ to a functor from $\text{Rep}_{\overline{\mathbb{Q}}_\ell}(\hat{M}\rtimes Q)^{\text{dbl}}$ to $$\text{End}_{D(\overline{\mathbb{Q}}_\ell)}(D(\text{Bun}_M,\overline{\mathbb{Q}}_\ell))\otimes_{D(\overline{\mathbb{Q}}_\ell)} D(\text{Div}_1, \overline{\mathbb{Q}}_\ell)$$ (using the pullback along the map $\text{Div}_1\to */Q$ to map $\text{Rep}_{\overline{\mathbb{Q}}_\ell}(Q)^{\text{dbl}}$ to $D(\text{Div}_1, \overline{\mathbb{Q}}_\ell)$), which gives the desired intertwining between $w$ and $\hat w$ for relative $\overline{\mathbb{Q}}_\ell$ coefficients.
In summary the Hecke action $T_{\hat{w}(V)}$ is equivalent to $w(T_V)$, $V\in \text{Rep}_{\overline{\mathbb{Q}}_\ell}(\hat{M}\rtimes Q)^{\text{dbl}}$.
Then $$w(a_{\psi_M}(V)) = w(T_V*i_{1!}W_{\psi_M})\cong w(T_V)* w(i_{1!}W_{\psi_M})\cong T_{\hat w(V)}*i_{1!}W_{\psi_{w(M)}} \cong a_{\psi_{w(M)}}(\hat w(V)) $$ 
The intertwining is compatible with the monoidal structure on $\text{Rep}_{\overline{\mathbb{Q}}_\ell}(\hat{M}\rtimes Q)^{\text{dbl}}$ as well as the functoriality in finite sets $I$ of the actions of $\text{Rep}_{\overline{\mathbb{Q}}_\ell}(\hat{M}\rtimes Q)^{I,\text{dbl}}$, so $ w(a_{\psi_M}(F))\cong a_{\psi_{w(M)}}(\hat w(F)) $ for all $F\in\text{Perf}(\text{Par}_M)$.
Taking right adjoints gives the desired result.

\end{proof}

\subsection{Review of the Representation Theory of Disconnected Reductive Groups}\label{DiscRedGp}

In this section we briefly review the representation theory of disconnected reductive groups, following \cite{AHR}.
Let $G$ be an algebraic group over an algebraically closed field $k$ of characteristic $0$.  
We say $G$ is \textit{disconnected reductive} if its identity component $G^\circ$ is reductive. 
Let $A$ be the component group of $G$.  
Fix a pinning of $G^\circ$.  
We freely use the notation in \ref{notation}.

The action of $G$ on $G^\circ$ by conjugation induces a map from $A$ to $\mathrm{Out}(G^{\circ})$.  
$A$ acts on $X^*(T)$ as follows.  
For any $g\in G$ lying in component $a\in A$, let $g^\circ$ be a closed point of $G^\circ$ such that conjugation by $g^\circ g$ preserves the pinning of $G^\circ$.  
For $\lambda\in X^*(T)$ we let ${}^a\lambda$ denote $\lambda$ precomposed with the automorphism of $T$ induced from the conjugation by $(g^\circ g)^{-1}$.  
This action is independent of the pinning, in the sense that if a different pinning is chosen with maximal torus $T'$, the action of $A$ on $X^*(T')$ agrees with that of $A$ on $X^*(T)$ composed with the conjugation map from $T$ to $T'$ taking the first pinning to the second. 
The orbit of a character $\lambda$ under this $A$-action will be denoted $[\lambda]$.

Let $A^\lambda$ be the subgroup of $A$ consisting of elements such that ${}^a\lambda=\lambda$, and let $G^\lambda$ be the preimage of $A^\lambda$ in $G$ under the projection from $G$ to $A$. 
Let $L(\lambda)$ be the irreducible representation $(\rho_\lambda, V_\lambda)$ of $G^\circ$ with highest weight $\lambda$, and for any $g\in G$ we denote by ${}^gL(\lambda)$ the representation defined by precomposing $\rho_\lambda$ with conjugation by $g^{-1}$.
For every element $a\in A$, let $\iota(a)$ be an element of $G$ in the component $a$ such that conjugation by $\iota(a)$ preserves the pinning of $G^\circ$.  
If $a\in A^\lambda$, let $\theta_a$ be the unique isomorphism from ${}^{\iota(a)}L(\lambda)$ to $L(\lambda)$ which is the identity on $\lambda$-weight spaces.
For $a,b\in A$ define $z(a,b)\in G^\circ$ so that $\iota(a)\iota(b)=\iota(ab)z(a,b)$.
Since $\iota(a)$, $\iota(b)$, and $\iota(ab)$ all preserve the chosen pinning of $G^\circ$, $z(a,b)\in Z(G^\circ)$, and hence $z(a,b)$ acts on $V_\lambda$ via a nonzero scalar $\alpha(a,b):=\rho_\lambda(z(a,b))$.  
In \cite[Chapter 2.4]{AHR} it is shown that $\alpha(a,b)$ is a $2$-cocycle of $A^\lambda$ with coefficients in $k^\times$.  Let $k^\alpha[A^\lambda]$ denote the group algebra of $A^\lambda$ twisted by $\alpha$. 

For any irreducible representation $E$ of $k^\alpha[A^\lambda]$ we endow $E\otimes L(\lambda)$ with the structure of a $G^\lambda$-representation by $\iota(a)g\cdot (u\otimes v) = a\cdot u \otimes \theta_a(g\cdot v)$ \cite[Lemma 2.8]{AHR}.
An important point is that the cohomology class of $\alpha$ does not depend on the choices $\iota(a)$ made in its construction, and up to isomorphism the algebra $k^\alpha[A^\lambda]$ and $G^\lambda$-representation $E\otimes L(\lambda)$ only depend on the cohomology class of $\alpha$.

One of the main results from \cite{AHR} are the following:

\begin{prop}\cite[Proposition 2.9]{AHR}
For a dominant character $\lambda$ of $T$, the irreducible representations of $G^\lambda$ which, after restricting to $G^\circ$, are isomorphic to a direct sum of copies of $L(\lambda)$, are exactly the representations of the form $E \otimes L(\lambda)$, where $E$ is an irreducible representation of $k^\alpha[A^\lambda]$.
\end{prop}

Via induction from $G^\lambda$ to $G$, pairs $(\lambda, E)$ parametrize all irreducible representations of $G$:

\begin{thm}\cite[Theorem 2.16]{AHR}\label{DiscRedClass}
Every irreducible representation of $G$ is of the form $L(\lambda, E):=\mathrm{Ind}_{G^\lambda}^G(E\otimes L(\lambda))$ where $\lambda$ is a dominant character and $E$ is an irreducible representation of $k^{\alpha}[A^\lambda]$.  
Conversely, for all such pairs $(\lambda, E)$, $L(\lambda, E)$ is irreducible, so there is a bijection between the irreducible representations of $G$ and $A$-conjugacy classes of pairs $(\lambda, E)$.
\end{thm}

As in \cite[Section 3.1]{AHR} we define a preorder on $A$-orbits of pairs $[\lambda, E]$ by $[\lambda, E]\leq [\lambda', E']$ if and only if $[\lambda]\leq[\lambda']$, ie $^a\lambda\leq\lambda'$ for some $a\in A$ in the usual partial order on weights.
We will regularly compare irreducible representations $\rho$ and $\rho'$ of $G$, implicitly using their bijection with $A$-orbits of pairs $[\lambda, E]$.

\subsection{Finite groups acting on sheaves}\label{DiscSheaves}

The goal of this section is to extend some ideas from Clifford Theory and the representation theory of disconnected reductive groups reviewed above to the setting of sheaves equivariant under a disconnected (not necessarily reductive) group.

We begin by stating an analogue of Mackey decomposition.  
Let $G$ be an algebraic group over a characteristic $0$ algebraically closed field $k$ with finite index subgroups $H$ and $K$, and let $X$ be a variety with a $G$-action.
Let $p: X/H\to X/G$ and $q:X/K\to X/G$ be the natural projections.
Then $$X/K\times_{X/G}X/H\cong (X\times G)/(K\times H)$$ via the action $(k,h)\cdot (x,g) =(k\cdot x, kgh^{-1})$, which is in turn isomorphic to $$\bigsqcup\limits_{g\in K\backslash G/H}X/(K\cap gHg^{-1})$$ where we abuse notation slightly by letting $g$ denote a double coset representative as opposed to the double coset itself.
Let $\tilde{q}_g: X/K\cap gHg^{-1}\to X/H$ and $\tilde{p}_g: X/(K\cap gHg^{-1})\to X/K$ be the natural maps, where $\tilde{q}_g$ involves of conjugating $K\cap gHg^{-1}$ by $g^{-1}$ before including into $H$.

Then Mackey decomposition states that $$q^*p_*S \cong \bigoplus\limits_{g\in K\backslash G/H}\tilde{p}_{g*}\tilde{q}_g^*S$$ which follows by etale base change.

Now assume $H$ is a closed normal subgroup of $G$ and suppose for the moment that $X$ is equipped with an $H$ action instead of a $G$ action. 
Let $S$ be an $H$ equivariant coherent sheaf on $X$ whose classical endomorphism ring (respecting the equivariant structure) is $k$. 
The quotient stacks $X/H$ and $(G\times^HX)/G$ are canonically isomorphic, and we denote by $\tilde{S}$ the $G$ equivariant sheaf on $G\times^HX$ which is the pushforward of $S$ along this isomorphism.
Here $G\times^HX$ denotes the product $G\times X$ quotiented by the action $h\cdot (g,x) = (gh^{-1}, hx)$, and $G$ acts on $G\times^HX$ via $g\cdot(g',x) = (gg', x)$.
To write $\tilde{S}$ explicitly in terms of $S$ we first choose a set of representatives $g_1,...,g_n$ of the left cosets of $H$ in $G$, where $g_1 = e$ is the identity of $G$.  
So $$G\times^HX =  \bigsqcup g_iH\times^H X$$ and we view $X$ as a subscheme via its isomorphism with the identity coset $H\times^HX$.  
We overload notation and also denote by $S$ the pushforward of $S$ along this inclusion, which is an $H$ equivariant sheaf on $G\times^HX$.

For $h\in H$, we let $\phi_h: h_*S\to S$ be the isomorphism defined by the $H$-equivariant structure on $S$, and for any $g_i$, we endow $g_{i*}S$ with an $H$ equivariant structure via the map $$g_{i*}\phi_{g_i^{-1}hg_i}: h_*g_{i*}S\to g_{i*}S$$
The sheaf $\tilde{S}$ is then $$\tilde{S}\cong \bigoplus\limits_{i=1}^n g_{i*}S$$
To describe its $G$ equivariant structure more explicitly, let $g_ih$ be an element of $G$, and define the integers $\sigma(i,j)$ and elements $h_{ij}\in H$ by $g_ig_j=g_{\sigma(i,j)}h_{ij}$.   
Then $$(g_ih)_*\bigoplus\limits_{j=1}^n g_{j*}S = \bigoplus\limits_{j=1}^n (g_ihg_j)_*S = \bigoplus\limits_{j=1}^n (g_ig_j)_*(g_j^{-1}hg_j)_*S = \bigoplus\limits_{j=1}^n g_{\sigma(i,j)*}(h_{ij}g_j^{-1}hg_j)_*S$$
And since $h_{ij}g_j^{-1}hg_j\in H $, the $H$-equivariant structure on $S$ gives the 
map $$\bigoplus\limits_{j=1}^n g_{\sigma(i,j)*}(h_{ij}g_j^{-1}hg_j)_*S\to\bigoplus\limits_{j=1}^n g_{\sigma(i,j)*}S=\bigoplus\limits_{j=1}^n g_{j*}S $$

Suppose now that there is an action of $G$ on $X$ extending that of $H$.  We examine the pushforward $\overline S$  of $\tilde S$ 
to a $G$ equivariant sheaf on $X$ via the action map $G\times^HX\to X$ given by $(g,x)\mapsto g\cdot x$.  In this case 
$$\overline S\cong\bigoplus\limits_{j=1}^n g_{j*}S$$ with $G$-action exactly analogous to that on $\tilde S$.  
Note also that $\overline{S}$ is isomorphic to the pushforward of $S$ along the projection $X/H\to X/G$.  Suppose furthermore that for all $j$, 
$g_{j*}S$ is isomorphic to $S$ as $H$-equivariant sheaves via a map $\theta_{j}:g_{j*}S\to S$ and moreover that $\theta_{j}$ is unique up to invertible scaling.  
Let $A$ be the finite group (formed by the $k$ points of) $G/H$.  
Consider the map $$g_{\sigma(i,j)*}S\xrightarrow[]{g_{\sigma(i,j)*}\phi^{-1}_{h_{ij}}}  g_{\sigma(i,j)*}h_{ij*}S = g_{i*}g_{j*}S\xrightarrow[]{g_{i*}\theta_{j}} g_{i*}S \xrightarrow[]{\theta_{i}} S$$
 After multiplying this map by some constant $\alpha(i,j)$ it must equal $\theta_{{\sigma(i,j)}}$.

For any $g=g_ih\in G$, let $\theta_g:=\theta_i\circ g_{i*}\phi_h: g_*S\to S$

\begin{lem}
For all $g=g_ih$, $g'=g_jh'\in G$, $\theta_g\circ g_*\theta_{g'}=\alpha(i, j)^{-1}\theta_{gg'}$
\end{lem}

\begin{proof}

\begin{align*}
\theta_g\circ g_*\theta_{g'} &= \theta_{g_i}\circ g_{i*}\phi_{h}\circ (g_ih)_*\theta_{g_j}\circ (g_ihg_j)_*\phi_{h'} \\ &\cong \theta_{i}\circ (g_i)_*\theta_{j}\circ (g_ig_j)_*\phi_{g_j^{-1}hg_j}\circ (g_ihg_j)_*\phi_{h'}\\ &\cong \alpha(i,j)^{-1}\theta_{\sigma(i,j)}\circ g_{\sigma(i,j)*}\phi_{h_{ij}}\circ (g_ig_j)_*\phi_{g_j^{-1}hg_jh'}\\ &\cong \alpha(i,j)^{-1}\theta_{\sigma(i,j)}\circ g_{\sigma(i,j)*}\phi_{h_{ij}g_j^{-1}hg_jh'} \\ &= \alpha(i,j)^{-1} \theta_{gg'}
\end{align*}

\end{proof}

Define $\alpha(g, g')\in k^\times$ by the identity $\theta_g\circ g_*\theta_{g'}=\alpha(g, g')^{-1}\theta_{gg'}$.
We will sometimes also use the notation $\alpha(a_g, a_{g'})$ for $\alpha(g,g')$, where $a_g$ is the component containing $g$ and $a_{g'}$ is the component containing $g'$.
For any $g, g', g''\in G$, computing $\theta_g\circ g_*\theta_{g'}\circ(gg')_*\theta_{g''}$ in two different ways we see $\alpha$ is a 2-cocycle of $A$.

Let $k^\alpha[A]$ be the component algebra of $A$ twisted by $\alpha$.  
For each $1\leq i\leq n$ let $a_i$ be the image of $g_i$ in $A$, and
define an isomorphism of $H$-equivariant sheaves $\Theta:\overline S\  \tilde\to k^\alpha[A]\otimes S$ by $$\Theta = \bigoplus\limits_{1\leq i\leq n}a_i\otimes_k \theta_i: \overline{S}\to \bigoplus\limits_{1\leq i\leq n} a_i\otimes_k S = k^\alpha[A] \otimes S$$
$k^\alpha[A] \otimes S$ therefore inherits a $G$-equivariant structure from $\overline S$.

\begin{lem}\label{InducedLem}
For any $g=g_ih'\in G$ the equivariant action of $g$ on $k^\alpha[A]\otimes S$ is the map $k^\alpha[A]\otimes g_*S\to k^\alpha[A]\otimes S$ given by $a_i\cdot(-)\otimes \theta_g$.

\end{lem}

\begin{proof}

For any $g_j$, the action of $g\in G$ on $g_{j*}S\subset \overline{S}$ takes the form $g_{k*}\phi_h: (gg_j)_*S\to g_{k*}S$, where $gg_j = g_kh$, and then $\Theta|_{g_{k*}S} = a_k\otimes \theta_k$ maps $g_{k*}S$ to $a_k\otimes S$.
By definition $\theta_k\circ g_{k*}\phi_h = \theta_{gg_j}$, so the action by $g$ and $\Theta$ compose to $a_k\otimes\theta_{gg_j}$.

In the other direction, $g_*\Theta|_{g_{j*}S} = a_j\otimes g_*\theta_j$, and then composing with $a_i\cdot(-)\otimes \theta_g$ gives $a_ia_j\otimes \theta_g\circ g_*\theta_j = \alpha(i,j)a_k\otimes \alpha(i,j)^{-1}\theta_{gg_j} = a_k\otimes\theta_{gg_j}$.

\end{proof}

For any representation $E$ of $k^\alpha[A]$, we can analogously define a $G$-action on $E\otimes S$:

\begin{defn}\label{ECompDef}
The $G$-equivariant structure on $E\otimes S$ is defined, for any $g=g_ih\in G$, by $g\cdot(-) := a_i\cdot (-)\otimes \theta_g: E\otimes g_*S\to E\otimes S$

\end{defn}

Definition \ref{ECompDef} and Lemma \ref{InducedLem} can be combined in the following way.
As setup, suppose we have a chain of inclusions $K\trianglelefteq H \trianglelefteq G$ where $K\trianglelefteq G$ and $G/K$ is finite.
Let $g_i$, $1\leq i\leq n$ be coset representatives of $G/K$.
Suppose $G$ acts on $X$ and let $S$ be a $K$-equivariant sheaf on $X$ such that there is a $K$-equivariant isomorphism $\theta_i:g_{i*}S\tilde\to S$ unique up to scalar for all $g_i$, which defines a 2-cocycle $\alpha(i,j)$ of $A:=G/K$ as above.
Let $A_H:=H/K$ and let $E$ be an irreducible representation of $k^\alpha[A_H]$.  

\begin{lem}\label{PushForwardE}

The pushforward of $E\otimes S$ along the projection $p$ from $X/H$ to $X/G$ is isomorphic to $\textup{Ind}_{k^\alpha[A_H]}^{k^\alpha[A]}(E)\otimes S$.  
It follows that if $F$ is an irreducible representation of $k^\alpha[A]$ whose restriction to $k^\alpha[A_H]$ contains $E$, then $F\otimes S$ is a direct summand of $p_*(E\otimes S)$.
    
\end{lem}

Note this is not an immediate consequence of Lemma \ref{InducedLem}, because the needed assumption $g_{i*}(E\otimes S)\cong E\otimes S$ for all $g_i$ fails if $E$ is not fixed by $g_i$-conjugation.

\begin{proof}

Let $\mathfrak{G}_H$ be a set of coset representatives of $G/H$, with the identity chosen to represent the identity coset.
Let $g\in G$, $\mathfrak{g}\in \mathfrak{G}_H$, and write $g\mathfrak{g} = \mathfrak{g}'h$.
Then $$p_*(E\otimes S) = \bigoplus\limits_{\mathfrak{g}\in \mathfrak{G}_H}\mathfrak{g}_*(E\otimes S)$$ with $G$-equivariant structure given by $$g\cdot(-):= a_h\cdot(-)\otimes\mathfrak{g}'_*\theta_h:(g\mathfrak{g})_*(E\otimes S) = E\otimes (g\mathfrak{g})_*S\to E\otimes\mathfrak{g}_*'S $$

The induced representation $$\textup{Ind}_{k^\alpha[A_H]}^{k^\alpha[A]}(E) = k^\alpha[A]\otimes_{k^\alpha[A_H]}E = \bigoplus\limits_{\mathfrak{g}\in\mathfrak{G}_H}a_\mathfrak{g}\otimes_{k^\alpha[A_H]} E$$ has $k^\alpha[A]$ structure given by $$a_g\cdot a_\mathfrak{g}\otimes_{k^\alpha[A_H]} E = \alpha(g,\mathfrak{g})a_{g\mathfrak{g}}\otimes_{k^\alpha[A_H]} E = \alpha(g,\mathfrak{g})\alpha(\mathfrak{g}', h)^{-1}a_{\mathfrak{g}'}\otimes_{k^\alpha[A_H]} a_h\cdot E$$

We claim $\Theta: p_*(E\otimes S)\to \textup{Ind}_{k^\alpha[A_H]}^{k^\alpha[A]}(E)\otimes S$ defined by $$\Theta|_{\mathfrak{g}_*S} = a_\mathfrak{g}\otimes(-)\otimes\theta_\mathfrak{g}: \mathfrak{g}_*(E\otimes S) = E\otimes\mathfrak{g}_*S\to a_{\mathfrak{g}}\otimes E\otimes S$$ is an isomorphism of $G$-equivariant sheaves.
In one direction, acting by $g$ and then applying $\Theta$ gives the map $$a_{\mathfrak{g}'}\otimes (a_h\cdot(-)\otimes \theta_{\mathfrak{g}'}\circ\mathfrak{g}'_*\theta_h) = (\alpha(\mathfrak{g}',h)^{-1}a_{\mathfrak{g}'}\otimes a_h\cdot(-))\otimes \theta_{g\mathfrak{g}}: E\otimes \mathfrak{g}_*S\to (a_{\mathfrak{g}'}\otimes E)\otimes S$$

On the other hand, applying $\Theta$ and then acting by $g$ gives $$a_g(a_\mathfrak{g}\otimes \text{id})\otimes\theta_g\circ  g_*\theta_\mathfrak{g} = \alpha(g,\mathfrak{g})\alpha(\mathfrak{g}', h)^{-1}\alpha(g,\mathfrak{g})^{-1}a_{\mathfrak{g}'}\otimes a_h\cdot(-)\otimes\theta_{g\mathfrak{g}}$$ $$= (\alpha(\mathfrak{g}', h)^{-1}a_{\mathfrak{g}'}\otimes a_h\cdot(-))\otimes\theta_{g\mathfrak{g}}: E\otimes (g\mathfrak{g})_*S\to (a_{\mathfrak{g}'}\otimes E)\otimes S$$

The last claim follows because $F$ is a direct summand of $\textup{Ind}_{k^\alpha[A_H]}^{k^\alpha[A]}(E)$.

\end{proof}

We end with the behavior of this equivariant structure under proper morphisms:

\begin{lem}\label{ECompPush}
Let $f:X\to Y$ be a proper morphism of varieties with $G$-action.  If $E\otimes S$ is a $G$-equivariant sheaf on $X$ as defined in Definition \ref{ECompDef}, then $f_*(E\otimes S)\cong E\otimes f_*(S)$ as $G$-equivariant sheaves on $Y$.
\end{lem}

\begin{proof}

Follows from direct computation by checking both sheaves have the same sections and equivariant structure.  Properness is assumed to ensure the pushforward is coherent.

\end{proof}

\begin{rem}

The sheaf $f_*(S)$ may not satisfy the property that there is a unique isomorphism between $g_{i*}f_*(S)$ and $f_*(S)$ up to scalar multiple, as $S$ was assumed to have.  But that property, which was used to define the cocycle $\alpha$, is not necessary because $f_*(S)$ already inherits $\alpha$ from $S$ making the analogous diagram involving the compositions of the isomorphisms $f_*(\theta_i)$ commute.
    
\end{rem}

We finally note that although this section was written in the context of the abelian category of coherent sheaves, analogous results hold for certain objects in the derived category of coherent sheaves.
Namely, if $S$ is a complex in $\text{Coh}(X/H)$ (the derived category) such that $\pi_i(\text{Aut}(S))=\text{Ext}^{-i}(S,S)=0$ for all $i>0$ then the results of this section hold verbatim for $S$.

\subsection{Quasi-Exceptional Sets on the Nilpotent Cone for some Disconnected Groups}\label{SpringerSect}

In this subsection we extend some of the notions of \cite{BezQExcep} to the setting of disconnected reductive groups appearing in the study of centralizers of L-parameters.

We begin by reviewing the notion of a quasi-exceptional set, following \cite[Section 2]{BezQExcep}.

\begin{defn}
If $\textbf{C}$ is a stable $\infty$ or triangulated category and $\{\nabla_i\}$ is a set of objects indexed by a totally ordered set $I$, then $\{\nabla_i\}$ is \textit{quasi-exceptional} if $\text{Hom}(\nabla_i, \nabla_j) = 0$ for all $i<j$, $\text{Hom}^{<0}(\nabla_i,\nabla_i) = 0$, and $\text{Hom}^{0}(\nabla_i,\nabla_i)$ is a division algebra for all $i$.  
It is furthermore \textit{exceptional} if $\textbf{C}$ is $k$ linear for a field $k$ and $\text{End}(\nabla_i)=k$ for all $i$.  
A quasi-exceptional set $\{\nabla_i\}$ is \textit{dualizable} if there exists a set of objects $\{\Delta_i\}$ indexed by $I$ such that $\text{Hom}(\Delta_i, \nabla_j) = 0$ whenever $i>j$ and $\Delta_i\cong\nabla_i\ \text{mod}\ \textbf{C}^{<i}$ for all $i$, where $\textbf{C}^{<i}$ is the stable subcategory of $\textbf{C}$ generated by $\{\nabla_j:j < i\}$.
\end{defn}

One of the primary motivations for the study of dualizable quasi-exceptional sets is that they can be used to construct exotic t-structures:

\begin{prop}\cite[Proposition 2.2.1]{BezQExcep}\label{ExoticTDef}
Let $\{\nabla_i\}$ be a dualizable quasi-exceptional set generating \textbf{\textup{C}}.  
Then the two full subcategories $\textbf{\textup{C}}^{\geq 0}$, $\textbf{\textup{C}}^{\leq 0}$ generated under extensions by $\{\nabla_i[d]: i\in I, d\leq 0\}$, $\{\Delta_i[d]: i\in I, d\geq 0\}$ respectively form a t-structure on $\textbf{\textup{C}}$.  This is moreover the unique t-structure such that $\nabla_i\in \textbf{\textup{C}}^{\geq 0}, \Delta_i\in \textbf{\textup{C}}^{\leq 0}$ for all $i$.
\end{prop}

In \cite{BezQExcep}, \cite{BezExotic}, Bezrukavnikov studies certain quasi-exceptional sets and their corresponding t-structures on $\text{Coh}(\mathcal{N}/G)$ and $\text{Coh}(\tilde{\mathcal{N}}/G)$, where $G$ is a connected reductive group and $\mathcal{N}$ is the nilpotent cone of its Lie algebra.
The goal of this section is to extend some of his results to the types of disconnected reductive groups which arise as the centralizers of generic L-parameters, using the results of the previous section.

To set notation, let $H$ be a reductive group over an algebraically closed field $k$ of characteristic $0$ with a fixed maximal torus $S$.
Let $T$ be a subtorus of $S$, and fix a strictly dominant (with respect to the Borel chosen in advance) cocharacter $\nu$ of $T$.
Let $Q$ be the repelling parabolic subgroup of $\nu$ and let $L$ be the centralizer of $\nu$.  
Let $G$ be a (possibly disconnected) reductive subgroup of $H$ with component group $A$ such that $T$ is a maximal torus of $G^\circ$.  
Let $M = G\cap L$ and $P=Q\cap G$.  
Note the identity component of $M$ is $T$, $Z(M)^\circ = T$, and $P^\circ$ is a Borel of $G^\circ$ with Levi $T$.  
We fix a pinning of $G^\circ$ with opposite Borel $P^\circ$ and maximal torus $T$.

Since the conjugation action of $M$ on $T$ is the identity, $M$ must act on $G^\circ$ by inner automorphisms which fix $T$.
Hence for all $m$ in $M$, there exists an element $t_m\in T$ such that the action of $M$ on $G^\circ$ is equal to the action of $t_m$ on $G^\circ$ by conjugation.  
In particular $mt_m^{-1}$ acts trivially on $G^\circ$.
It follows that every connected component of $M$ contains an element which acts trivially on $G^\circ$.  
Conversely if $g\in G$ centralizes $G^\circ$ then it centralizes $\nu$, and hence must lie in $M$.
We conclude that $A_M$ is exactly the kernel of the map from $A$ to $\text{Out}(G^\circ)$.
 
For any weight $\lambda$ of $T$ let $A^\lambda$ be the subgroup of $A$ fixing $\lambda$.
Note $A_M\subset A^\lambda$ since it acts trivially on $T$.
Let $\mathcal{N}$ be the nilpotent cone of the  Lie algebra of $G$, $U$ be the unipotent radical of $P^\circ$, and let $\mathfrak{u}$ be its Lie algebra.
Since $U$ is the kernel of the map from $P$ to $M$, $P$ and $M$ have the same component group.

The standard example of this setup to keep in mind for later applications consists of taking $L$ to be the minimal Levi of $H$ containing the image of some semisimple generic $L$-parameter $\phi$, $G$ the centralizer of $\phi$ in $H$, $M$ the centralizer of $\phi$ in $L$, and $P$ the centralizer of $\phi$ in $Q$.
\newline

Define the stack $\tilde{\mathcal{N}}/G:=G\times^P\mathfrak{u}/G\cong\mathfrak{u}/P$.
Our overarching goal in this section is to endow the derived categories of coherent sheaves on $\tilde{\mathcal{N}}/G$ and $\mathcal{N}/G$ with respective exotic and perverse coherent t-structures, as in the connected setting \cite{BezQExcep}, \cite{BezExotic}.
We study $\tilde{\mathcal{N}}/G$ first.  
In \cite{BezExotic} a dualizable exceptional set $\{\nabla^{\mathrm{ex}}_\lambda\}$ on $\tilde{\mathcal{N}}/G^\circ$ ($\tilde{\mathcal{N}}/G^\circ$ by definition being $G^\circ\times^{P^\circ}\mathfrak{u}/G^\circ\cong\mathfrak{u}/P^\circ$) is constructed, indexed by weights (not necessarily dominant) of $T$.  
Its dual is denoted $\{\Delta^{\mathrm{ex}}_\lambda\}$.
He constructs $\{\nabla^{\mathrm{ex}}_\lambda\}$ by first proving $\{\mathcal{O}(\lambda)\}$ is a quasi-exceptional set with respect to any total order compatible with the usual partial order on weights \cite[Proposition 2.3.6]{BezExotic}, where $\mathcal{O}(\lambda)$ is the pullback to $\tilde{\mathcal{N}}/G^\circ$ of the $G^\circ$ equivariant line bundle associated to $\lambda$ on $G^\circ/P^\circ$ via Borel-Weil-Bott.
He then uses a process called mutation to construct $\{\nabla^{\mathrm{ex}}_\lambda\}$ from $\{\mathcal{O}(\lambda)\}$, changing the order to any compatible with the Bruhat order on weights \cite[Section 2.3]{BezExotic}.
We extend his construction to $\tilde{\mathcal{N}}/G$.

To do so, let $\tau$ be an irreducible representation of $M$, and let $\mathcal{O}(\tau)\in\mathrm{Coh}(\tilde{\mathcal{N}}/G)$ be $\eta^*\tau$ where $\eta$ is defined by the diagram 

\[
\begin{tikzcd}
  \mathfrak{u}/P \arrow[r] \arrow[rr, "\eta", bend left=35] & */P \arrow[r,] & */M
\end{tikzcd}
\]

By Theorem \ref{DiscRedClass} there is a character $\lambda$ of $T$ and an irreducible representation $E$ of $k^\alpha[A_M]$ such that $\tau\cong E\otimes \lambda$.  
Using the notation of Section \ref{DiscSheaves} note that $A_M^\lambda=A_M$ for all $\lambda$ since $M^\circ$ is central in $M$, and consequently $\mathcal{O}(\tau)\cong E_\tau\otimes\mathcal{O}(\lambda)$.
We define $\nabla^{\text{ex}}_\tau\in\text{Coh}(\tilde{\mathcal{N}}/G)\cong \text{Coh}(\mathfrak{u}/P)$ to be $E_\tau\otimes \nabla^{\text{ex}}_\lambda$, and $\Delta^{\text{ex}}_\tau$ to be $E_\tau\otimes \Delta^{\text{ex}}_\lambda$.
Here we are using the fact that $\iota(a)_*\nabla^{\text{ex}}_\lambda\cong \nabla^{\text{ex}}_\lambda$ for all $a\in A_M$ and the fact that the isomorphism $t_*\nabla^{\text{ex}}_\lambda\cong \nabla^{\text{ex}}_\lambda$ is given by multiplication by $\lambda(t)$ for all $t\in T$, each of which follow from the identical statements for $\mathcal{O}(\lambda)$ and the definition of $\nabla^{\text{ex}}_\lambda$ as a certain $k$-linear functor applied to $\mathcal{O}(\lambda)$.

Recall the preorder $\leq$ on irreducible representations of $M$ defined in Section \ref{DiscRedGp}.
If $\leq^b$ is the Bruhat order on weights of $T$ from \cite[Section 2.3]{BezExotic}, extend $\leq^b$ to a preorder on irreducible representations of $M$ by defining $\tau\leq^b\tau'$ whenever $\lambda\leq^b\lambda'$.

\begin{lem}\label{O(tau)vanishing}
i) If $\tau<\tau'$ then $\mathrm{Hom}(\mathcal{O}(\tau),\mathcal{O}(\tau')) = 0$
    
ii)  If $\tau\cong\tau'$ then $\mathrm{Hom}(\mathcal{O}(\tau),\mathcal{O}(\tau))\cong k$ 

iii) If $\tau\sim\tau'$ but $\tau\ncong\tau'$ then $\mathrm{Hom}(\mathcal{O}(\tau),\mathcal{O}(\tau')) = 0$

iv)  The same properties hold with $\nabla^{\textup{ex}}_\tau$ and $\nabla^{\textup{ex}}_{\tau'}$ in place of $\mathcal{O}(\tau)$ and $ \mathcal{O}(\tau')$.

In particular for any total order on the set of irreducible representations of $M$ compatible with $\leq$ resp. $\leq^b$, $\{\mathcal{O}(\tau)\}$ resp. $\{\nabla^{\textup{ex}}_\tau\}$ are exceptional sets.
The dual of $\{\nabla^{\textup{ex}}_\tau\}$ is $\{\Delta^{\textup{ex}}_\tau\}$.
\end{lem}

\noindent We remind the reader that by convention $\text{Hom}$ indicates \textit{derived} hom.

\begin{proof}

Let $\tilde{p}$ be the projection from $\mathfrak{u}/P^\circ$ to $\mathfrak{u}/P$.
By \cite[Lemma 2.3.5]{BezExotic} $\text{Hom}(\mathcal{O}(\lambda), \mathcal{O}(\lambda'))=0$ when $\lambda<\lambda'$, and $k$ when $\lambda = \lambda'$.
Since $\tilde{p}$ is finite, it follows that $\text{Hom}^i(\tilde{p}_*\mathcal{O}(\lambda), \tilde{p}_*\mathcal{O}(\lambda'))=0$ when $i\neq0$.
Since $\mathcal{O}(\tau)$ is a direct summand of $ \tilde{p}_*\mathcal{O}(\lambda)$ and $\mathcal{O}(\tau')$ is a direct summand of $\tilde{p}_*\mathcal{O}(\lambda') $ we conclude $ \text{Hom}(\mathcal{O}(\tau), \mathcal{O}(\tau'))=0$ when $\lambda<\lambda'$.  
When $\lambda=\lambda'$, $\text{Hom}^i(\mathcal{O}(\tau), \mathcal{O}(\tau'))=0$ for $i\neq0$, and $\text{Hom}^0(\mathcal{O}(\tau), \mathcal{O}(\tau'))= \text{Hom}^0(E_\tau\otimes \mathcal{O}(\lambda), E_{\tau'}\otimes\mathcal{O}(\lambda'))\cong k$ when $\tau\cong\tau'$, and $0$ when $\tau\ncong\tau'$.
An analogous argument holds for $\{\nabla^{\text{ex}}_\tau\}$.

For the last claim, since $\{\Delta^{\text{ex}}_{\lambda}\}$ is dual to $\{\nabla^{\text{ex}}_{\lambda}\}$ in $\text{Coh}(\mathfrak{u}/P^\circ)$ it follows as in the previous paragraph that $\text{Hom}(\Delta^{\text{ex}}_{\tau}, \nabla^{\text{ex}}_{\tau'})=0$ when $\lambda>^b\lambda'$ and $\Delta^{\text{ex}}_{\tau}\cong \nabla^{\text{ex}}_{\tau}$ mod $\text{Coh}^{<^b\lambda}(\mathfrak{u}/P^\circ)$ for all $\tau$.
And since $\text{Hom}(E_\tau\otimes\nabla^{\text{ex}}_\lambda, E_{\tau'}\otimes\nabla^{\text{ex}}_\lambda) = 0$ when $\tau\ncong\tau'$, in fact $\Delta^{\text{ex}}_{\tau}\cong \nabla^{\text{ex}}_{\tau}$ mod $\text{Coh}^{<^b\tau}(\mathfrak{u}/P^\circ)$.
To finish note that since $M^\circ$ is central in $M$, for each $a\in A_M$ there is an element $m_a\in a$ such that $m_a$ acts trivially on $\mathfrak{u}$.
Therefore we can take $m_{a*}\nabla^{\text{ex}}_\tau\cong \nabla^{\text{ex}}_\tau$ and $m_{a*}\Delta^{\text{ex}}_\tau\cong \Delta^{\text{ex}}_\tau$ via the identity.
For any $i$, $\text{Hom}^i(E_\tau\otimes \Delta^{\text{ex}}_\lambda, E_{\tau'}\otimes \nabla^{\text{ex}}_\lambda)\cong (\text{Hom}(E_\tau, E_{\tau'})\otimes \text{Hom}^i(\Delta^{\text{ex}}_\lambda, \nabla^{\text{ex}}_\lambda))^{A_M}$.
But $A_M$ acts on $\text{Hom}^i(\Delta^{\text{ex}}_\lambda, \nabla^{\text{ex}}_\lambda)$ trivially, so $\text{Hom}^i(E_\tau\otimes \Delta^{\text{ex}}_\lambda, E_{\tau'}\otimes \nabla^{\text{ex}}_\lambda)=0$ if $E_\tau\ncong E_{\tau'}$.
Therefore $\text{Hom}(\Delta^{\text{ex}}_\tau, \nabla^{\text{ex}}_{\tau'})=0$ if $\tau>^b\tau'$.

\end{proof}

We now shift focus to $\mathcal{N}/G$ and the perverse-coherent t-structure, following \cite{BezQExcep} where the analogous sheaves were studied in the connected setting.
Let $\rho$ be an irreducible representation of $G$, $G^\lambda$ be the preimage in $G$ of $A^\lambda$, and $G_M$ be the preimage in $G$ of $A_M$.
Fix choices $\iota(a)$ as in Section \ref{DiscRedGp} (so in particular conjugation by $\iota(a)$ fixes the chosen pinning on $G^\circ$ for all $a$) and let $\alpha$ be the associated 2-cocycle of $A$.
By Theorem \ref{DiscRedClass} $\rho$ is isomorphic to $L(\lambda, E_\rho):=\mathrm{Ind}_{G^\lambda}^G(E_\rho\otimes L(\lambda))$ for a dominant weight $\lambda$ and irreducible representation $E_\rho$ of $k^\alpha[A^\lambda]$, well defined up to conjugation by $A$.
Let $\tau\cong E_\tau \otimes \lambda$ be an irreducible subrepresentation of $\rho|_M$.

Let $$G^\circ\times^{P^\circ}\mathfrak{u}\to\mathcal{N}$$ be the Springer resolution sending $(g,u)\mapsto g\cdot u$.  
This map is $G$-equivariant, and we let $\mu^\circ$ denote the map $ \tilde{\mathcal{N}}/G^\circ\cong G^\circ\times^{P^\circ}\mathfrak{u}/G^\circ\to\mathcal{N}/G^\circ$.
For a dominant weight $\lambda$ define $\overline{\nabla}_\lambda:=\mu^\circ_*\mathcal{O}(\lambda)$ and $\overline{\Delta}_\lambda:=\mu^\circ_*\mathcal{O}(w_0(\lambda))$ where $w_0$ is the longest length element of the Weyl group of $G^\circ$.
We remind the reader that all pushforwards here are derived.
By \cite[Theorem 1]{BezQExcep} the set $\{\overline{\nabla}_\lambda\}$ is quasi-exceptional in $\text{Coh}(\mathcal{N}/G^\circ)$ with dual $\{\overline{\Delta}_\lambda\}$.

In the disconnected setting we have the following diagram of maps.

\[
\begin{tikzcd}
\mathfrak{u}\times^PG_M/G_M \arrow[r, "\cong"] \arrow[d, "\mu_M"'] & \mathfrak{u}\times^PG^\lambda/G^\lambda \arrow[r, "\cong"] \arrow[d, "\mu^\lambda"] & \mathfrak{u}\times^PG/G  \arrow[d, "\mu"] \\
\mathcal{N}/G_M \arrow[r, "p_M^\lambda"'] \arrow[rr, "p_M"', bend right=35]                   & \mathcal{N}/G^\lambda \arrow[r, "p^\lambda"']                    & \mathcal{N}/G
\end{tikzcd}
\]

From now on we refer to the stack in the top row as $\tilde{\mathcal{N}}/G$.
For every $a\in A$, we give $\iota(a)_*\overline{\nabla}_\lambda$ a $G^\circ$-equivariant structure as in Section \ref{DiscSheaves}.

\begin{lem}

For all $a\in A^\lambda$, $\iota(a)_*\overline{\nabla}_\lambda\cong \overline{\nabla}_\lambda$ and $\iota(a)_*\overline{\Delta}_\lambda\cong \overline{\Delta}_\lambda$ as $G^\circ$-equivariant sheaves.
    
\end{lem}

\begin{proof}

Since $\lambda=\! ^a\lambda = \iota(a)_*\lambda$ on $*/P^\circ$, it follows $\mathcal{O}(\lambda)=\eta^{\circ,*}\lambda\cong\eta^{\circ,*}\iota(a)_*\lambda\cong \iota(a)_*\eta^{\circ,*}\lambda = \iota(a)_*\mathcal{O}(\lambda)$, using smooth base change to swap $\iota(a)_*$ and $\eta^{\circ,*}$.

It follows that $\overline{\nabla}_\lambda = \mu^\circ_*\mathcal{O}(\lambda)\cong\mu^\circ_*\iota(a)_*\mathcal{O}(\lambda)\cong \iota(a)_*\mu^\circ_*\mathcal{O}(\lambda)= \iota(a)_*\overline{\nabla}_\lambda$.  
A similar argument holds for $\overline{\Delta}_\lambda$.
    
\end{proof}

For each $a\in A$, we normalize the isomorphism $\iota(a)_*\overline{\nabla}_\lambda\cong \overline{\nabla}_\lambda$ by taking it to be the pushforward under $\mu^\circ_*$ of the unique isomorphism $\iota(a)_*\mathcal{O}(\lambda)\cong \mathcal{O}(\lambda)$ which restricts to the identity at $0\in\mathfrak{u}$, where uniqueness follows since constants are the only $T$-invariant global sections of $\mathcal{O}(\mathfrak{u})$.
We label this isomorphism $\theta_a^\nabla$.
Define $\theta_a^\Delta:\iota(a)_*\overline{\Delta}_\lambda\tilde\to\ \overline{\Delta}_\lambda$ to be the unique isomorphism such that $\theta_a^\nabla\circ\iota(a)_*f = f\circ\theta_a^\Delta$ for any $f\in\text{Hom}^0(\overline{\Delta}_\lambda,\overline{\nabla}_\lambda)\cong k$.
The uniqueness of $\theta_a^\Delta$ follows because $\text{End}^0(\overline{\Delta}_\lambda)=k$.

\begin{lem}

$A^\lambda = A^{w_0\lambda}$, and $\overline{\nabla}_\lambda$ and $\overline{\Delta}_\lambda$ have the same associated 2-cocycle.
    
\end{lem}

\begin{proof}

Since conjugation by $\iota(a)$ preserves the pinning of $G^\circ$ and $w_0$ is the unique Weyl group element flipping the sign of all roots, conjugation by $\iota(a)$ fixes $w_0$, and hence their actions on characters commute, implying $A^\lambda = A^{w_0\lambda}$.
The matching of cocycles for $\overline{\nabla}_\lambda$ and $\overline{\Delta}_\lambda$ follows by construction.

\end{proof}

The following is one of the central definitions in this paper:

\begin{defn}\label{NablaDef}

For any irreducible representation $\rho$ of $G$ let $$\overline{\nabla}_\rho:= p^\lambda_*(E_\rho\otimes \overline{\nabla}_\lambda) \in \text{Coh}(\mathcal{N}/G)$$
Similarly define $$\overline{\Delta}_\rho:= p^\lambda_*(E_\rho\otimes \overline{\Delta}_\lambda)\in \text{Coh}(\mathcal{N}/G)$$
    
\end{defn}

\begin{rem}

Definition \ref{NablaDef} can be extended to any finite dimensional representation $\rho$ by defining $$\overline{\nabla}_\rho:=\bigoplus\limits_{\rho'}V_{\rho'}\otimes_k \overline{\nabla}_{\rho'}$$ where $$\rho\cong\bigoplus\limits_{\rho'}V_{\rho'}\otimes_k \rho'$$ is the isotypic decomposition of $\rho$.
    
\end{rem}

Since $\!^aE_\rho\otimes \overline{\nabla}_{\!^a\lambda}\cong g_*(E_\rho\otimes\overline{\nabla}_\lambda)$ for any $g$ representing the coset $a$, where the $G^\lambda$-action on the latter sheaf is given by $g_*\phi_{g^{-1}(-)g}$ in the notation of Section \ref{DiscSheaves}, $p^{^a\lambda}_*(\!^aE_\rho\otimes \overline{\nabla}_{\!^a\lambda})\cong p_*^\lambda(E_\rho\otimes\overline{\nabla}_\lambda)$, and so the definition of $\overline{\nabla}_\rho$ is independent of the choice of $\lambda$ within its $A$-orbit, and similarly for $\overline{\Delta}_\rho$.

\begin{lem}\label{SpringResDecomp}
Let $\mathcal{O}(\tau)\cong E_\tau\otimes_k\mathcal{O}(\lambda)$ as above.  
Then $\mu_{M*}\mathcal{O}(\tau)\cong E_\tau\otimes_k \overline{\nabla}_\lambda$ and $\mu^\lambda_*\mathcal{O}(\tau)\cong \textup{Ind}_{k^\alpha[A_M]}^{k^\alpha[A^\lambda]} (E_\tau)\otimes_k \overline{\nabla}_\lambda$.
It follows that $\overline{\nabla}_\rho$ is a direct summand of $\mu_*\mathcal{O}(\tau)$.
The same results hold for $\overline{\Delta}_\rho$.
\end{lem}

\begin{proof}

The first claim then follows from Lemma \ref{ECompPush} applied to $\mu_M$, and the second follows because $$\mu^\lambda_*\mathcal{O}(\tau)\cong p^\lambda_{M*}\mu_{M*}\mathcal{O}(\tau)\cong p^\lambda_{M*} (E_\tau\otimes_k\overline{\nabla}_\lambda)\cong \text{Ind}_{k^\alpha[A_M]}^{k^\alpha[A^\lambda]} (E_\tau)\otimes_k \overline{\nabla}_\lambda$$ by Lemma \ref{PushForwardE}.

Since $\rho\subset \text{Ind}_{k^\alpha[A_M]}^{k^\alpha[A^\lambda]} (E_\tau)$, $E_\rho\otimes_k\overline{\nabla}_\rho$ is a direct summand of $\mu^{\lambda}_*\mathcal{O}(\tau)$, and hence $\overline{\nabla}_\rho = p^\lambda_*(E_\rho\otimes_k\overline{\nabla}_\lambda)$ is a direct summand of $p^\lambda_*\mu^\lambda_*\mathcal{O}(\tau) = \mu_*\mathcal{O}(\tau)$.

The same arguments hold for $\overline{\Delta}_\rho$.
    
\end{proof}

\begin{lem}\label{NablaQExcep}
$\overline{\nabla}_\rho$ is a genuine coherent sheaf for all $\rho$, and the collection $\{\overline{\nabla}_\rho\ |\ \rho\!\in\!\textup{Irr}(G)\}$ generates $\textup{Coh}(\mathcal{N}/G)$, as do the complexes $\overline{\Delta}_\rho$.  
Furthermore the sheaves $\overline{\nabla}_\rho$ are nearly dualizable quasi-exceptional in that the following properties hold:

i) $\textup{Hom}(\overline{\nabla}_{\rho}, \overline{\nabla}_{\rho'}) = 0$ if $[\lambda]<[\lambda']$, $\textup{Hom}^{<0}(\overline{\nabla}_{\rho}, \overline{\nabla}_{\rho'}) = 0$ if $[\lambda]=[\lambda']$, $\textup{Hom}^{0}(\overline{\nabla}_{\rho}, \overline{\nabla}_{\rho'}) = 0$ if $[\lambda]=[\lambda']$ but $\rho\ncong\rho'$, and $\textup{End}^0(\overline{\nabla}_{\rho})=k$.

ii) $\textup{Hom}(\overline{\Delta}_{\rho}, \overline{\nabla}_{\rho'}) = 0$ if $[\lambda]\neq[\lambda']$, $\overline{\Delta}_\rho\cong\overline{\nabla}_\rho\ \textup{mod}\ \textup{Coh}^{<[\lambda]}(\mathcal{N}/G)$, 
$\textup{Hom}^0(\overline{\Delta}_{\rho}, \overline{\nabla}_{\rho'}) \cong \textup{Hom}(E_{\rho},E_{\rho'})$ if $[\lambda]=[\lambda']$ (and choosing representatives so that $\lambda=\lambda'$), $\textup{Hom}(\overline{\Delta}_{\rho}, \overline{\Delta}_{\rho'}) = 0$ if $[\lambda]>[\lambda']$, $\textup{Hom}^{\leq0}(\overline{\Delta}_{\rho}, \overline{\Delta}_{\rho'}) = 0$ if $[\lambda]=[\lambda']$ but $\rho\ncong\rho'$, $\textup{End}^0(\overline{\Delta}_{\rho})=k$ and $\textup{End}^{<0}(\overline{\Delta}_{\rho})=0$.
    
\end{lem}

\begin{proof}
Let $p$ be the projection from $\mathcal{N}/G^\circ$ to $\mathcal{N}/G$.  
By \cite[Fact 3.1]{BezQExcep} we know that $\overline{\nabla}_\lambda$ is a genuine coherent sheaf on $\mathcal{N}/G^\circ$.  
Then since $p$ is finite $p_*(\overline{\nabla}_\lambda)$ is a coherent sheaf on $\mathcal{N}/G$.  Since $\overline{\nabla}_\rho$ is a direct summand of $p_*(\overline{\nabla}_\lambda)$ it follows that $\overline{\nabla}_\rho$ is a genuine coherent sheaf.
By \cite[Proposition 3.5]{BezQExcep} the sheaves $\overline{\nabla}_\lambda$ generate $\text{Coh}(\mathcal{N}/G^\circ)$, and since $p^\circ:\mathcal{N}/G^\circ\to \mathcal{N}/G$ is a finite etale cover it follows that $\{\overline{\nabla}_\rho\}$ generates $\text{Coh}(\mathcal{N}/G)$.

For arbitrary irreducible $\rho$, $\rho'$, $$\text{Hom}(\overline{\nabla}_{\rho}, \overline{\nabla}_{\rho'})=\text{Hom}(p^\lambda_*(E_{\rho}\otimes \overline{\nabla}_{\lambda}),p^{\lambda'}_*(E_{\rho'}\otimes \overline{\nabla}_{\lambda'})) = \text{Hom}(p^{\lambda'*}p^\lambda_*(E_{\rho}\otimes \overline{\nabla}_{\lambda}),E_{\rho'}\otimes \overline{\nabla}_{\lambda'})$$

By Mackey decomposition discussed in section \ref{DiscSheaves}, and using notation from that discussion, $$p^{\lambda'*}p^\lambda_*(E_{\rho}\otimes \overline{\nabla}_{\lambda})\cong \bigoplus\limits_{g\in G^{\lambda'}\backslash G/G^{\lambda}}\tilde{p}_{g*}\tilde{q}_g^*(E_\rho\otimes \overline{\nabla}_\lambda)$$ so $$\text{Hom}(p^{\lambda'*}p^\lambda_*(E_{\rho}\otimes \overline{\nabla}_{\lambda}),E_{\rho'}\otimes \overline{\nabla}_{\lambda'})\cong \bigoplus\limits_{g\in G^{\lambda'}\backslash G/G^\lambda}\text{Hom}(\tilde{p}_{g*}\tilde{q}_g^*(E_\rho\otimes \overline{\nabla}_\lambda), E_{\rho'}\otimes \overline{\nabla}_{\lambda'})$$

For the rest of the proof, by the discussion after Definition \ref{NablaDef} we can WLOG assume that if $[\lambda]$ is less than, equal to, or greater than $[\lambda']$ then $\lambda$ is respectively less than, equal to, or greater than $\lambda'$.

\[
\begin{tikzcd}[column sep=huge, row sep=large]
\mathcal{N}/G^\circ \arrow[r, "\tilde{q}_h^\circ"] \arrow[d, "\tilde{p}_h^\circ = \text{id}"]
  & \mathcal{N}/(G^{\lambda'}\cap gG^\lambda g^{-1}) \arrow[r, "\tilde{q}_g"] \arrow[d, "\tilde{p}_g"] 
  & \mathcal{N}/G^\lambda \arrow[d, "p^\lambda"] \\
\mathcal{N}/G^\circ  \arrow[r, "p^{\circ\lambda'}"] 
  & \mathcal{N}/G^{\lambda'}   \arrow[r, "p^{\lambda'}"] 
  & \mathcal{N}/G 
\end{tikzcd}
\]

Fix a $g\in G$ representing a double coset of $G^{\lambda'}\backslash G/G^\lambda$.  
Define $p^{\circ\lambda'}$ to be the quotient from $\mathcal{N}/G^\circ$ to $\mathcal{N}/G^{\lambda'}$.  
Then $p^{\circ\lambda'}$ is finite etale, and hence $p^{\circ\lambda'*}$ is faithful (meaning, injective on all Ext groups).  
Consider $\text{Hom}(p^{\circ\lambda'*}\tilde{p}_{g*}\tilde{q}_g^*(E_\rho\otimes \overline{\nabla}_\lambda), p^{\circ\lambda'*}(E_{\rho'}\otimes \overline{\nabla}_{\lambda'}))$.
Applying Mackey decomposition to the first term, $$ p^{\circ\lambda'*}\tilde{p}_{g*}\tilde{q}_g^*(E_\rho\otimes \overline{\nabla}_\lambda)\cong \bigoplus\limits_{h\in G^\circ\backslash G^{\lambda'}/G^{\lambda'}\cap g G^\lambda g^{-1}}\tilde{p}_{h*}^{\circ}\tilde{q}_{h}^{\circ*}\tilde{q}_g^*(E_\rho\otimes \overline{\nabla}_\lambda)$$ where, noting $$G^\circ\cap h(G^{\lambda'}\cap g G^\lambda g^{-1})h^{-1} = G^\circ$$ for all $h$, $$ \tilde{q}_{h}^\circ: \mathcal{N}/G^\circ\to \mathcal{N}/(G^{\lambda'}\cap g G^\lambda g^{-1})$$ where the inclusion is precomposed with conjugation by $h^{-1}$, and $\tilde{p}^\circ_{h}: \mathcal{N}/G^\circ\to \mathcal{N}/G^\circ$ is the identity.
For the summands corresponding to a choice of $h$ and $g$ the maps are shown in the above figure (note the squares are not cartesian, as the upper left object of every square is only a direct summand of the cartesian product).

As $h$ varies through representatives of $G^{\lambda'}/(G^{\lambda'}\cap gG^\lambda g^{-1})$, the set of elements $hg$ form a set of representatives of the cosets in $G/G^\lambda$ which project to the double coset $G^{\lambda'}\backslash g /G^\lambda$.  
Therefore as $G^\circ$-equivariant sheaves,  $$\bigoplus\limits_{h\in  G^{\lambda'}/G^{\lambda'}\cap g G^\lambda g^{-1}}\tilde{p}_{h*}^{\circ}\tilde{q}_{h}^{\circ*}\tilde{q}_g^*(E_\rho\otimes \overline{\nabla}_\lambda)\cong\bigoplus\limits_{g'\in G/G^\lambda:\ G^{\lambda'} g'  = G^{\lambda'} g} E_\rho\otimes \overline{\nabla}_{\!^{g'}\!\lambda}$$
By \cite[Proposition 3.4]{BezQExcep}, $\text{Hom}(\overline{\nabla}_{\!^{g'}\!\lambda}, \overline{\nabla}_{\lambda'}) = 0$ if $\lambda<\lambda'$,  $\text{Hom}(\overline{\nabla}_{\!^{g'}\!\lambda}, \overline{\nabla}_{\lambda'}) = 0$ if $\lambda=\lambda'$ and $g'\neq G^\lambda$, $\text{End}^{<0}(\overline{\nabla}_{\lambda}) = 0$, and $ \text{End}^0( \overline{\nabla}_{\lambda}) = k$.
In order for the second fact to be true we use that, in the setting where $\lambda=\lambda'$ and $g'\neq G^\lambda$, $\!^{g'}\lambda$ is not contained in the convex hull of the Weyl group orbit of $\lambda$ in the character lattice of $G^\circ$ (as holds for any nontrivial outer automorphism acting on a dominant character).

It follows that $\text{Hom}(\overline{\nabla}_{\rho},\overline{\nabla}_{\rho'}) = 0$ if $\lambda<\lambda'$ and $\text{Hom}^{<0}(\overline{\nabla}_{\rho},\overline{\nabla}_{\rho'}) = 0$ if $\lambda = \lambda'$.  
If $\lambda = \lambda'$ but $E_\rho\ncong E_{\rho'}$, then $$\textup{Hom}^0(\overline{\nabla}_{\rho},\overline{\nabla}_{\rho'}) = \text{Hom}^0(\tilde{p}_{e*}\tilde{q}_e^*(E_\rho\otimes \overline{\nabla}_\lambda), E_{\rho'}\otimes \overline{\nabla}_{\lambda})= \text{Hom}^0(E_\rho\otimes \overline{\nabla}_\lambda, E_{\rho'}\otimes \overline{\nabla}_{\lambda}) = 0$$
If $\rho\cong\rho'$, then $\text{Hom}^0(E_\rho\otimes \overline{\nabla}_\lambda, E_{\rho}\otimes \overline{\nabla}_{\lambda}) = k$ since $E_\rho$ is irreducible and $ \text{Hom}^0(\overline{\nabla}_\lambda, \overline{\nabla}_{\lambda}) = k$ as mentioned above.

By \cite[Propositions 3.4, 3.3, Lemma 2.2]{BezQExcep}, $\text{Hom}(\overline{\Delta}_\lambda, \overline{\nabla}_{\lambda'}) = 0$ when $\lambda\neq\lambda'$ and $\text{Hom}^0(\overline{\Delta}_\lambda, \overline{\nabla}_{\lambda}) \cong k$; $\overline{\Delta}_\lambda\cong \overline{\nabla}_{\lambda}$ mod $\text{Coh}^{<\lambda}(\mathcal{N}/G^\circ)$, where $\text{Coh}^{<\lambda}(\mathcal{N}/G^\circ)$ is the full subcategory of $\text{Coh}(\mathcal{N}/G^\circ)$ generated by $\overline{\nabla}_{\lambda'}$, $\lambda'<\lambda$, under cones and shifts; and $\text{Hom}(\overline{\Delta}_\lambda, \overline{\Delta}_{\lambda'}) = 0$ if $\lambda > \lambda'$, $\text{End}(\overline{\Delta}_\lambda) \cong \text{End}(\overline{\nabla}_{\lambda})$.
Using these facts together with the same Mackey decomposition arguments as above it follows that $\text{Hom}(\overline{\Delta}_\rho, \overline{\nabla}_{\rho'}) = 0$ whenever $[\lambda]\neq[\lambda']$,$\text{Hom}^0(\overline{\Delta}_\rho, \overline{\nabla}_{\rho}) \cong \text{Hom}(E_{\rho},E_{\rho'})$ when $[\lambda]=[\lambda']$, $\overline{\Delta}_\rho\cong \overline{\nabla}_{\rho}$ mod $\text{Coh}^{<[\lambda]}(\mathcal{N}/G)$, $\text{Hom}(\overline{\Delta}_\rho, \overline{\Delta}_{\rho'}) = 0$ if $[\lambda] > [\lambda']$, and $\text{End}(\overline{\Delta}_\rho) \cong \text{End}(\overline{\nabla}_{\rho})$.
From the equivalence of $\overline{\Delta}_\rho$ and $\overline{\nabla}_\rho$ modulo lower degree terms it follows the complexes $\overline{\Delta}_\rho$ generate $\text{Coh}(\mathcal{N}/G)$ as well.

\end{proof}

In \cite{PervCoh} Arinkin and Bezrukavnikov study the perverse-coherent t-structure, which was initially introduced by Deligne.
We won't review the definition here, but it is exactly analogous to the usual definition of the perverse t-structure in the constructible setting and is given in \cite[Definition 3.2]{PervCoh}.
By \cite[Corollary 3.1.3]{BezQExcep}, when $G$ is connected reductive the perverse coherent and the exotic t-structure corresponding to the quasi-exceptional set $\{\overline{\nabla}_\lambda[d]\}$ on $\text{Coh}(\mathcal{N}/G)$ coincide, where $d=\text{dim}(\mathcal{N})/2$.
And by \cite[Lemma 3.1.9]{BezQExcep} the complexes $\overline{\nabla}_\lambda[d]$ and $\overline{\Delta}_\lambda[d]$ are perverse coherent for all $\lambda$.  
Moving to the disconnected setting considered in this section, since $\overline{\nabla}_\rho $ is a direct summand of $p^\circ_*\overline{\nabla}_\lambda$, and perverse-coherent sheaves are perserved under finite etale pushforward and taking direct summands, it follows that $\overline{\nabla}_\rho$ and $\overline{\Delta}_\rho$ are perverse coherent as well.
Together with \cite[Proposition 3.5]{BezQExcep} this establishes:

\begin{prop}\label{PerverseEqExotic}

The sheaves $\overline{\nabla}_\rho[d]$, $d\leq0$ generate $D^{p,\geq0}(\textup{Coh}(\mathcal{N}/G))$, and $\overline{\Delta}_\rho[d]$, $d\geq0$ generate $D^{p,\leq0}(\textup{Coh}(\mathcal{N}/G))$.    
\end{prop}

There are also intersection cohomology sheaves in the perverse-coherent setting \cite[Definition 4.6]{PervCoh}, and by \cite[Proposition 4.11]{PervCoh} the irreducible perverse-coherent sheaves on $\mathcal{N}/G$ are exactly $IC(\mathcal{O}, L)$ where $O$ is a $G$-orbit in $\mathcal{N}$ and $L$ is an irreducible representation of the stabilizer of $\mathcal{O}$ in $G$.
By \cite[Corollary 3.1.3]{BezQExcep}, the perverse-coherent sheaves on $\mathcal{N}/G^\circ$ form a quasi-hereditary category \cite[Section 2.1]{BezQExcep} with costandard objects $\overline{\nabla}_\lambda[d]$ and standard objects $\overline{\Delta}_\lambda[d]$.
Therefore, letting $IC_\lambda:= \text{Im}(\overline{\Delta}_\lambda\to \overline{\nabla}_\lambda)$, the irreducible perverse-coherent sheaves on $\mathcal{N}/G^\circ$ also take the form $IC_\lambda[d]$.
In the disconnected reductive setting, let $IC_\rho:= p^\lambda_*(E_\rho\otimes IC_\lambda)$.

\begin{lem}\label{IrrIC}
    $IC_\rho$ is an irreducible perverse-coherent sheaf, and every irreducible perverse-coherent sheaf is of the form $IC_\rho$ for a unique $\rho\in\textup{Irr}(S_\phi)$.
\end{lem}

\begin{proof}

We start by showing $IC_\rho$ is irreducible.
Since $IC_\lambda$ and $E_\rho$ are irreducible, so is $E_\rho\otimes IC_\lambda$.
Let $F$ be any irreducible perverse-coherent sheaf such that $F\hookrightarrow IC_\rho$. Then $\text{Hom}(F, IC_\rho) = \text{Hom}(p^{\lambda*}F, E_\rho\otimes IC_\lambda)\neq0$, and since $E_\rho\otimes IC_\lambda$ is irreducible $p^{\lambda*}F\twoheadrightarrow E_\rho\otimes IC_\lambda$ as perverse-coherent sheaves on $\mathcal{N}/G^{\lambda}$.
Pulling back under $p^\circ$,  $p^{\circ*}F\twoheadrightarrow E_\rho\otimes IC_\lambda$ as perverse-coherent sheaves on $\mathcal{N}/G^{\circ}$
Then $$p^{\circ*}g_*F = g_*p^{\circ*}F\twoheadrightarrow E_\rho\otimes IC_{\!^g\lambda}$$ for all $g\in G$.
Since $F$ is $G$-equivariant $g_*F\cong F$, so $$p^{\circ*}F\twoheadrightarrow \bigoplus\limits_{g\in G/G^\lambda}E_\rho\otimes IC_{\!^g\lambda} = p^{\circ*}IC_\rho$$ by Mackey decomposition.
Therefore $F\twoheadrightarrow IC_\rho$ as required.

Now suppose $F$ is an arbitrary irreducible perverse-coherent sheaf on $\mathcal{N}/G$.
Then $p^{\circ*}F$ has an irreducible quotient $IC_\lambda$.
By adjunction $\text{Hom}(p^{\circ*}F,IC_\lambda) = \text{Hom}(p^{\lambda*}F,p^{\circ,\lambda}_*IC_\lambda) = \text{Hom}(p^{\lambda*}F,k^\alpha[A^\lambda]\otimes IC_\lambda)$ where the last equality follows by Lemma \ref{PushForwardE}.
It follows that $\text{Hom}(p^{\lambda*}F,E\otimes IC_\lambda)\neq0$ for some irreducible representation $E$ of $k^\alpha[A^\lambda]$, and consequently $F\cong IC_{L(\lambda, E)}$ since both are irreducible.

\end{proof}

We conclude by establishing the t-exactness of $\mu_*$.

\begin{prop}

Pushforward along $\mu$ is t-exact with respect to the exotic t-structure on $\textup{Coh}(\tilde{\mathcal{N}}/G)$ and the $(-d)$-shifted perverse-coherent t-structure on $\textup{Coh}(\mathcal{N}/G)$.
    
\end{prop}

\begin{proof}

By \cite[Lemma 2.4.8]{BezExotic} for all Weyl group elements $w$ and dominant weights $\lambda$, $\mu^\circ_*(\nabla^{\text{ex}}_{w\lambda}) = \overline{\nabla}_\lambda$ and $\mu^\circ_*(\Delta^{\text{ex}}_{w\lambda}) = \overline{\Delta}_\lambda$.
By etale base change, $\mu_{M*}( \nabla^{\text{ex}}_\tau) = E_\tau\otimes \overline{\nabla}_\lambda$ and $\mu_{M*}( \Delta^{\text{ex}}_\tau) = E_\tau\otimes \overline{\Delta}_\lambda$, and therefore $\mu_{*}( \nabla^{\text{ex}}_\tau) = p_{M*}(E_\tau\otimes \overline{\nabla}_\lambda)\in \text{Coh}^{\geq 0}(\mathcal{N}/G)$ and $\mu_{*}( \Delta^{\text{ex}}_\tau) = p_{M*} (E_\tau\otimes \overline{\Delta}_\lambda)\in\text{Coh}^{\leq 0}(\mathcal{N}/G) $.

\end{proof}

\section{Main Theorem}

\subsection{ Generous L Parameters}

Generous parameters are those for which understanding the categorical local Langlands correspondence should be the most straightforward.

\begin{defn}\cite[Definition 2.1.5]{BeijingNotes}
An L parameter $\phi: W_F\to\,^LG(\overline{\mathbb{Q}}_\ell)$ is \textit{generous} if $q^{-1}(q(\phi))^{\text{red}}\cong B_{S_\phi}$.
\end{defn}

\begin{rem}

In \cite[Definition 3.6]{HamEis} Hamann defines a Galois-cohomological condition on toral L-parameters which he calls genericity (and does not require $\overline{\mathbb{Q}}_\ell$-coefficients).
It can be seen that in the setting of $\overline{\mathbb{Q}}_\ell$-coefficients his genericity condition agrees with the generosity condition defined above.
Note that it is distinct from the notion of genericity to be defined in the next section, which is weaker.
Similar Galois-cohomological conditions are defined in \cite[Definition 6.2]{HL}, called \textit{Langlands-Shahidi type} and \textit{weakly Langlands-Shahidi type}, which correspond to generosity and genericity, respectively.
    
\end{rem}

The categorical equivalence is expected to be very well behaved at generous parameters:
\begin{conjecture}\cite[Conjecture 2.1.8]{BeijingNotes}\label{GenerousConj}

Let $G$ be a quasi-split reductive group, $b\in B(G)$, and $\pi$ a smooth irreducible representation of $G_b(F)$ with generous Fargues-Scholze parameter $\phi$.
Then, $i_{b\sharp}^{\textup{ren}}\pi\cong i_{b!}^{\textup{ren}}\pi\cong i_{b*}^{\textup{ren}}\pi$ and $c_\psi(i_{b!}^{\textup{ren}}\pi) = i_{\phi*}\rho$ for an irreducible representation $\rho$ of $S_\phi$.
Varying over all generous parameters $\phi$, the associated mapping from the set of pairs $(b,\pi)$, where the Fargues-Scholze parameter of $\pi$ is generous, to the set of pairs $(\phi,\rho)$, where $\phi$ is generous and $\rho\in\textup{Irr}(S_\phi)$, is a bijection satisfying the opposite sign compatibility with the Kottwitz map as Theorem \ref{BMO}.
In cases where classical local Langlands is known, the map $(b,\pi)\mapsto (\phi,\rho^\vee)$ agrees with the classical BM-O bijection.
\end{conjecture}

This is known for $\text{GL}(n)$ by work of Zou \cite[Theorem 8.6]{Zou25}.

\begin{rem}

In \cite{BeijingNotes} this conjecture is formulated in terms of the spectral action functor $a_\psi$ rather than its right adjoint $c_\psi$.  
So more specifically, \cite[Conjecture 2.1.8]{BeijingNotes} takes the form $a_\psi(i_{\phi*}\rho ) \cong i_{b!}^{\text{ren}}\pi$.
When the categorical correspondence is known and $G$ is reasonable as in \cite[Definition A.1.1]{BeijingNotes} (which all groups are conjectured to be, though currently only some groups such as $\text{GL}(n)$ with a classical LLC compatible with the Fargues-Scholze parameter are known to have this property), 
then after specializing to $\phi$, $\text{Perf}(\text{Par}_G^\phi) = \text{Coh}(\text{Par}_G^\phi)$ since $\phi$ is a smooth point of $\text{Par}_G$, $a_\psi$ and $c_\psi$ both preserve compact objects, and they are inverse equivalences between $\text{IndCoh}(\text{Par}_G^\phi)$ and $D(\text{Bun}_G)^\phi$.

\end{rem}

\subsection{Semisimple Generic L Parameters}

\begin{defn}\label{GenLParamDef}
A semisimple L-parameter $\phi$ is \textit{generic} if it is not the semisimplification of any L parameter with nontrivial monodromy.  More generally, an L-parameter is generic if its semisimplification is.
    
\end{defn}
There are a number of equivalent characterizations of generic parameters:

\begin{prop}  \cite[Proposition 2.3.1]{BeijingNotes}\label{EquivGenDef}\ Let $\phi$ be a semisimple parameter.  
The following are equivalent conditions on genericity:

1) $q^{-1}(q(\phi))^{\textup{red}}\cong \mathcal{N}_{S_\phi}/S_\phi$.

2) $L(s,\mathrm{ad}\circ\phi)$ is holomorphic
at $s=1$.

3) The composite map $BS_{\phi}\to\mathcal{N}_{S_{\phi}}/S_\phi\overset{\nu_{\phi}}{\to}\mathrm{Par}_{G}$
defines a smooth point of $\mathrm{Par}_{G}$, where $\nu_\phi$ denotes the natural inclusion from $q^{-1}(q(\phi))^{\textup{red}}$ into $\textup{Par}_G$.

4) The stack $\mathrm{Par}_{G}$ is smooth in a Zariski neighborhood
of $\textup{Par}_G\times_{\textup{Par}_G^{\textup{spec}}}q(\phi)$.
\end{prop}

\begin{proof}

The equivalence of $1)$ and the definition of semisimple genericity is immediate from \ref{fibstruct}.

$2)\iff 1)$: By definition $$L(s, \mathrm{ad}\circ\phi)=\mathrm{det}(1-q^{-s}\mathrm{ad}(\phi(\mathrm{Fr}))|\mathfrak{g}^{\mathrm{ad}\phi(I_F)})^{-1}$$ which is holomorphic at $s=1$ if and only if $q$ is not an eigenvalue of $\mathrm{ad}(\phi(\mathrm{Fr}))|\mathfrak{g}^{I_F}$.  Equivalently there are no $N\in\mathfrak{g}^{\mathrm{ad}\phi(I_F)}$ for which $\phi(\mathrm{Fr})\cdot N = qN$, and since $\phi(\text{Fr})$ is semisimple there are no $N\in\mathfrak{g}^{\mathrm{ad}\phi(I_F)}$ for which $\phi(\mathrm{Fr})\cdot N = q^{-1}N$ either (this can be seen by considering the root space decomposition of $\mathfrak{g}^{\mathrm{ad}\phi(I_F)}$ with respect to a maximal torus containing $\phi(\text{Fr})$, noting that $N$ must be a sum of root space vectors, and then considering any element $N'$ in the corresponding direct sum of opposite root spaces).
We conclude by Proposition \ref{fibstruct}.

$1)\iff3)$: By \cite[Corollary 5.3]{DHKM} $\mathrm{Par}_G$ is smooth at $\phi$ if and only if $H^0(W^0_F, \mathrm{ad}(\phi)^*(1))=0$ for a choice of discretization $W_F^0$ of the Weil group (\cite[Section 1.2]{DHKM}).  The vanishing of $H^0$ is equivalent to $\mathrm{Fr}$ not having any $q^{-1}$ eigenvectors in $\mathfrak{g}^{\mathrm{ad}\phi(I_F)}$.

$3)\iff4)$: Clearly $4)$ implies $3)$, so assume $3)$.  We then know $\mathrm{Par}_G$ is smooth in a neighborhood of $\phi$.
Since $\phi$ is in the closure of every point in $q^{-1}(q(\phi))$, 4) follows.

\end{proof}

\begin{prop} 
The locus $\mathrm{Par}_{G}^{\textup{spec},\mathrm{gen}}\subset \mathrm{Par}_{G}^{\textup {spec}}$ parametrizing semisimple generic parameters is open and dense, and contains all generous parameters.

\end{prop}

\begin{proof}
The last claim is immediate.  
Openness follows because by GIT theory the image of the singular locus of $\text{Par}_G$ is closed in $\text{Par}_G^{\text{spec}}$, and by Proposition \ref{EquivGenDef} an L-parameter is generic if and only if $\text{Par}_G$ is smooth in a neighborhood of $q^{-1}(q(\phi))$.
For density, let $\phi$ be a semisimple parameter and let $S$ be the connected component of the centralizer of $\phi(I_F)$ ($S$ is reductive since $\phi(I_F)$ is finite).  
Let $T_S$ be a maximal torus of $S$, which we choose to be normalized by $\phi(\text{Fr})$. 
For any $t\in T_S(\overline{\mathbb{Q}}_\ell)$ define $\phi_t$ to be the $L$-parameter such that $\phi_t(w) = \phi(w)$ for all $w\in I_F$ and such that $\phi_t(\mathrm{Fr}) = t\phi(\mathrm{Fr})$. 
We claim the collection of $t$ such that $\text{ad}(t\phi(\mathrm{Fr}))$ has no $q^{-1}$ eigenvalues are the $\overline{\mathbb{Q}}_\ell$-points of an open dense subvariety of $T_S$. 
Let $\mathfrak{t}$ be the Lie algebra of $T_S$.
Then since a finite power of $t\phi(\text{Fr})$ acts as the identity on $\mathfrak{t}$, all the eigenvalues of $t\phi(\text{Fr})|_\mathfrak{t}$ are roots of unity, and hence cannot be $q^{-1}$ for any $t$.
To study the eigenvalues of $\phi(\text{Fr})$ on $\mathfrak{s}$ lying in the span of roots it suffices to study its eigenvalues on each subspace of the form $$\mathfrak{s}_O:=\bigoplus\limits_{\alpha\in O}\mathfrak{s}_\alpha$$ where $O$ is an orbit of roots under the action of $\phi(\text{Fr})$.
Let $v\in\mathfrak{s}_O$ be an eigenvector.  
Then the eigenvalue of $t\phi(\text{Fr})$ acting on $v$ defines a section $f_v(t)$ of $\mathcal{O}(T_S)$, and if $f_v(t)$ is not constant the requirement $f_v(t)=q^{-1}$ is a closed condition on $T_S$.
We claim that if $f_v(t)$ is constant then it must be a root of unity, which will conclude the proof of generic density.
Since $t\phi(\text{Fr})$ permutes the roots of $T_S$, so there is some large integer $k$ such that every root in $O$ is an eigenvector of $(t\phi(\text{Fr}))^k$; $k$ is independent of $t$ since $t$ doesn't permute roots.
In fact we may replace $k$ with a large multiple of itself to assume $(t\phi(\text{Fr}))^k\in T_S$ for all $t$ ($\hat{G}\rtimes Q$ has finite component group, so any $T_S$-point can be power-raised into the identity component, and $\phi(I_F)$ is finite so any $T_S$-point of its normalizer can be power-raised to its centralizer).
To show the eigenvalues of $t\phi(\text{Fr})$ are roots of unity it suffices to show the same for $(t\phi(\text{Fr}))^k$, and since $(t\phi(\text{Fr}))^k\in T_S$ we may assume $v\in\mathfrak{s}_\alpha$ for some $\alpha\in O$. 
Note that $$(t\phi(\text{Fr}))^k=(\prod\limits_{i=1}^{k}\phi(\text{Fr})^{i-1}t\phi(\text{Fr})^{-(i-1)})\phi(\text{Fr})^k $$
Since the eigenvalues of $(t\phi(\text{Fr}))^k$ and $\phi(\text{Fr})^k$ acting on $v$ are both constant as $t$ varies, their ratio must be as well.
And being a character of $T_S$ it is trivial, so we are reduced to showing $\alpha(\phi(\text{Fr})^k)$ is a root of unity.
Plugging in $t=\phi(\text{Fr})^k$, $$ 1= \alpha(\prod\limits_{i=1}^{k}\phi(\text{Fr})^{i-1}\phi(\text{Fr})^k\phi(\text{Fr})^{-(i-1)} )= \alpha(\phi(\text{Fr})^k)^{k}$$ as desired.
So the map from $T_S$ to $\text{Par}_G^{\text{spec}}$ sending $t$ to $q(\phi_t)$ sends $1$ to $q(\phi)$ and a dense open of $T_S$ into the generic locus, establishing density.

\end{proof}

A nice property of the generic locus which will simplify the study of spectral Eisenstein series is the following:
\begin{prop}
\cite[Proposition 2.3.3]{BeijingNotes}For any parabolic $P=MU\subset G$, mapping to $\textup{Par}_G^{\textup{spec}}$ via the composite $\mathrm{Par}_{P}\overset{p^{\mathrm{spec}}}{\to}\mathrm{Par}_{G}\overset{q}{\to}\textup{Par}_G^{\mathrm{spec}}$, the open substack 
\[
\mathrm{Par}_{P}^{\mathrm{gen}}:=\mathrm{Par}_{P}\times_{\textup{Par}_G^{\mathrm{spec}}}\textup{Par}_G^{\mathrm{spec,gen}}\subset\mathrm{Par}_{P}
\]
has trivial derived structure.
\end{prop}

\begin{proof}
Given in \cite{BeijingNotes}.
\end{proof}

\ 
\newline

Let $M$ be the minimal standard Levi of $G$ such that some conjugate of $\phi$ factors through $\hat{M}\rtimes Q$, and WLOG we replace $\phi$ by such a conjugate.
Let $\hat{P}$ be the opposite standard parabolic of $\hat{G}$ with Levi $\hat{M}$, $S_{\phi,P}=S_\phi\cap\hat P$, and $S_{\phi,M}=S_\phi\cap\hat M$.
We wish to show we are in the situation of Section \ref{SpringerSect}, with the following translation: 

\[
G=S_\phi,\quad M=S_{\phi,M},\quad P=S_{\phi,P},\quad H=\hat{G},\quad Q=\hat{P},\quad L=\hat{M}, \quad T=Z(\hat{M}\rtimes Q)^\circ
\]

Let $\phi_M$ be the lift of $\phi$ to $\hat M\rtimes Q$, so in particular $S_{\phi,M} = S_{\phi_M}$.
\begin{lem}\label{CenterSphiM}
$S_{\phi, M}^\circ = Z(\hat{M}\rtimes Q)^\circ$, and therefore $\phi_M$ is generous.
\end{lem}

\begin{proof}

$Z(\hat{M}\rtimes Q)^\circ\subset S_{\phi, M}^\circ$ is immediate since the image of $\phi$ is contained in $\hat{M}\rtimes Q$.
In the other direction, let $T_\phi$ be a maximal torus of $S_{\phi,M}^\circ$.  
The image of $\phi$ is contained in $C_{\hat{G}\rtimes Q}(T_\phi)$, and if $S_{\phi,M}^\circ$ is not a torus $C_{\hat{G}\rtimes Q}(T_\phi)$ is a Levi strictly contained in $\hat{M}\rtimes Q$, a contradiction.

\end{proof}

Now we study the sheaves on the centralizer $S_\phi$ of $\phi$ constructed in section \ref{SpringerSect}.  
Fix an irreducible representation $\rho$ of $S_\phi$.  
Let $\overline{\Delta}_{\rho}, IC_{\rho},$ and $ \overline{\nabla}_{\rho}$ be as defined in \ref{SpringerSect}, and let $\delta_{\rho}$ be the pushforward of $\rho$ along the inclusion of $*/S_\phi$ into $\mathcal{N}_{S_\phi}/S_\phi$ at $0$.

\begin{prop}\label{GenSheafProp}
The sheaves $\overline{\nabla}_\rho$ and $\overline{\Delta}_\rho$ satisfy the following properties, conjectured in \cite[Section 2.3]{BeijingNotes}.

1. When $\phi$ is generous, $\overline{\Delta}_{\rho}=IC_{\rho}=\overline{\nabla}_{\rho}=\delta_{\rho}=\rho$.

2. When $\rho$ factors over the quotient $S_{\phi}\to S_{\phi}^{\natural}$,
$\overline{\Delta}_{\rho}=IC_{\rho}=\overline{\nabla}_{\rho}$ is
the pullback of $\rho$ along the tautological map $\mathcal{N}_{S_{\phi}}\to BS_{\phi}$.

3. For all $\rho$, there is a unique irreducible representation $\rho'$ of $S_\phi$
such that $\delta_{\rho}\cong IC_{\rho'}[d]$, where $d=\mathrm{dim}(\mathcal{N}_{S_\phi})/2$.

4. There is a surjection of coherent sheaves $\overline{\nabla}_{\rho}\twoheadrightarrow\delta_{\rho}$.
\end{prop}

\begin{proof}

1.  If $\phi$ is generous, then $\mathcal{N}_{S_\phi}$ consists of the single point $0$, so $S_\phi^\circ = S_{\phi,M}^\circ$.
For any dominant weight $\lambda$, $\overline{\nabla}_\lambda = \lambda\in\text{Coh}(*/S_\phi^\circ)$.
Then $$\overline{\nabla}_\rho=E_\rho\otimes \lambda=\rho=\delta_\rho$$

\noindent Since $S_\phi^\circ$ is a torus its Weyl group is trivial.  
So $w_0(\lambda)=\lambda$ and  $\overline{\Delta}_\rho=\overline{\nabla}_\rho$.  
It follows that $$IC_\rho = \overline{\Delta}_\rho=\overline{\nabla}_\rho=\delta_\rho$$

2.  If $\rho$ factors through $S_\phi^\natural:=S_\phi/(S_\phi\cap\hat{G}^{\text{der}})^\circ$ then in particular $\rho|_{S_\phi^\circ}$ factors through $S^\circ_\phi/(S_\phi^\circ)^{\text{der}}$, the abelianization of $S_\phi^\circ$.
So all weights in $\rho$ are the pullbacks of characters of $S_\phi^\circ$ and thus invariant under $W_{S_\phi^\circ}$.
So $\overline{\Delta}_\rho=\overline{\nabla}_\rho$, and therefore $\overline{\Delta}_\rho=\overline{\nabla}_\rho=IC_\rho$.

Consider the following diagram (which does not commute!):

\[
\begin{tikzcd}
\tilde{\mathcal{N}}_{S_\phi}/S_\phi^\circ\cong\mathfrak{u}/S_{\phi,P}^\circ \arrow[r, "\eta^\circ"] \arrow[d, "\mu^\circ"'] & */S_{\phi, P}^\circ \arrow[r, "p_{P,M}^\circ"] & */S_{\phi,M}^\circ \arrow[d, "\iota_M^\circ"] \\ \mathcal{N}_{S_\phi^\circ}/S_\phi^\circ \arrow[rr, "\pi^\circ"'] & & */S_\phi^\circ
\end{tikzcd}
\]

Note that $\pi^\circ\mu^\circ$ is also given by $\eta^\circ$ composed with $\iota_P^\circ: */S_{\phi,P}^\circ\to*/S_\phi^\circ$.  
If $\lambda$ is a character of $S_\phi^\circ$, then although the diagram does not commute it is still true that

$$\mathcal{O}_{\tilde N}(\lambda) = (\iota^\circ_Mp^\circ_{P,M}\eta^\circ)^*\lambda = \mu^{\circ,*}\pi^{\circ,*}\lambda$$ 
because the unipotent radical of $S_{\phi,P}^\circ$ lies in the kernel of $\lambda$.  
So $$\overline{\nabla}_\lambda = \mu_*(\mu^*\pi^*\lambda) = \pi^*\lambda$$ (where the last equality follows from the projection formula, and the fact that the nilpotent cone has rational singularities).

Now consider the diagram

\[
\begin{tikzcd}
\mathcal{N}_{S_\phi}/S_\phi^\lambda \arrow[r, "\pi^\lambda"] \arrow[d, "p^\lambda"'] & */S_\phi^\lambda \arrow[d, "\tilde{p}^\lambda"] \\
\mathcal{N}_{S_\phi}/S_\phi \arrow[r, "\pi"'] & */S_\phi
\end{tikzcd}
\]

Then $$\overline{\nabla}_\rho\cong p^\lambda_*(E_\rho\otimes \overline{\nabla}_\lambda)\cong p^\lambda_*(E_\rho\otimes\pi^{\circ,*}\lambda)\cong p_*^\lambda\pi^{\lambda,*}(E_\rho\otimes\lambda)\cong \pi^*\tilde{p}^\lambda_*(E_\rho\otimes\lambda )$$ as desired, using etale base change in the last step.

3.  Recalling $d=\text{dim}(\mathcal{N}_{S_\phi})/2$, by Lemma \ref{IrrIC} and \cite[Proposition 4.11]{PervCoh} the perverse-coherent sheaves $IC_\rho[d]$ are exactly those of the form $ IC(O,L)$ for some orbit $O\subset\mathcal{U}_{S_\phi}$ and irreducible representation $L$ of its centralizer in $S_\phi$.
Taking $O$ to be the trivial orbit in $\mathcal{U}_{S_\phi}$, its centralizer is $S_\phi$, and $IC(O,\rho)=\delta_\rho$.
It follows that there exists some irreducible representation $\rho'$ of $S_\phi$ such that $IC_{\rho'}[d]\cong\delta_\rho$.

4.  The classical fiber at $0$ of $\overline{\nabla}_\lambda$ surjects onto $L(\lambda)$, so the classical fiber of $E_\rho\otimes \overline{\nabla}_\lambda$ at $0$ surjects onto $E_\rho\otimes L(\lambda)$, and hence the classical fiber of $p^\lambda_*(E_\rho\otimes \overline{\nabla}_\lambda)$ at $0$ surjects onto $\text{Ind}_{S^\lambda_{\phi}}^{S_\phi}(E_\rho\otimes L(\lambda)) = \rho$.
And then $\overline{\nabla}_\rho\twoheadrightarrow\delta_\rho$ is just the natural projection.

\end{proof}

\subsection{The Fargues-Scholze correspondence at a generic parameter}

The goal of this section is to study the Fargues-Scholze correspondence at a generic parameter $\phi$ and prove the main theorem of the paper.  We keep the notation of the previous section.

We begin by recalling the relationship between the spectral geometric Eisenstein functor and the Springer theory studied in earlier sections:

\begin{prop}\label{SpecEisCalc}

Let $\phi$ be a semisimple generic parameter and let $\phi_M$ be a lift of $\phi$ to $\textup{Par}_M$.  Then for any irreducible representation $\tau$ of $S_{\phi,M}$ with $\tau|_{S_{\phi,M}^\circ}$ dominant, the following isomorphisms hold: $$\textup{Eis}^{\textup{spec}}_{\hat P}(i_{\phi_M*}\tau) \cong  \nu_{\phi*}\overline{\nabla}_{\textup{Ind}_{S_{\phi,P}}^{S_\phi}(\tau)}$$ $$\textup{Eis}^{\textup{spec}}_{\hat P}(i_{\phi_M*}w_0\tau) \cong  \nu_{\phi*}\overline{\Delta}_{\textup{Ind}_{S_{\phi,P}}^{S_\phi}(\tau)}$$
    
\end{prop}

\begin{proof}

An analogous argument to that given in the proof of \cite[Proposition 2.2.3]{BeijingNotes} in the setting of the trivial L-parameter holds in general to show that $\textup{Eis}^{\textup{spec}}_{\hat P}(i_{\phi_M*}\tau) \cong \mu_*\mathcal{O}(\tau)$, and we conclude by using the isomorphism $$\mu_*\mathcal{O}(\tau)\cong \overline{\nabla}_{\text{Ind}_{S_{\phi,P}}^{S_\phi}(\tau)}$$ given by Lemma \ref{SpringResDecomp}.
The second claim follows similarly.

\end{proof}

\begin{thm}\label{MainThm}
Let $G$ be a quasi-split reductive group for which local Langlands is known in the sense of Section \ref{ClassicalLLC} for $G$ and all of its Levi subgroups.  
Suppose that for all $b\in B(G)$ and all smooth irreducible representations $\pi$ of $G_b(F)$, the Fargues-Scholze parameter of $\pi$ agrees with the semisimplification of its classical parameter, and furthermore the same holds for all Levi subgroups of $G$.
Fix $(b,\pi)$, let $\phi$ be the Fargues-Scholze parameter of $\pi$, and let $S_\phi$ be the stabilizer of $\phi$ in $\hat{G}$.  
Suppose $\phi$ is generic, $c_\psi$ is fully faithful on $D(\textup{Bun}_G)^{\omega,\phi}$, and that Conjectures \ref{GenerousConj} and \ref{EisCompat} hold for $G$ and all of its Levi subgroups.

i) There is a unique irreducible representation $\rho$ of $S_\phi$ such that the following isomorphisms hold: 

\begin{align}
&c_\psi(i_{b!}^{\textup{ren}}\pi)\cong\nu_{\phi*}\overline{\nabla}_\rho\\
&c_\psi(i_{b\sharp}^{\textup{ren}}\pi)\cong\nu_{\phi*}\overline{\Delta}_\rho\\
&c_\psi(i_{b\sharp!}^{\textup{ren}}\pi)\cong\nu_{\phi*}IC_\rho\\
&c_\psi(i_{b!*}^{\textup{ren}}\pi)\cong\nu_{\phi*}\delta_\rho
\end{align}

ii)  $c_\psi$ is t-exact with respect to the hadal t-structure on $D(\textup{Bun}_G)^\phi$ and the $(-d)$-shifted perverse-coherent t-structure on $\textup{Coh}(\textup{Par}_G^\phi)$.

iii)  The map given by i), from the set of pairs $(b,\pi)$ where $\pi$ has generic Fargues-Scholze parameter $\phi$, to the set of pairs $(\phi,\rho)$ where $\rho\in\textup{Irr}(S_\phi)$, is a bijection.
We conjecture it agrees with the classical BM-O bijection after dualizing $\rho$.

\end{thm}

\begin{rem}
The Fargues-Scholze and classical local Langlands parameters are known to agree in many cases, including $\text{GL}(n)$ and many classical groups \cite{BMHN} \cite{{DvHKZ26}} \cite{FS} \cite{HamGanTak} \cite{Pen}.
The full faithfulness assumption is currently known unconditionally for tori \cite{Zou22}, and for $\text{GL}(n)$ and some classical groups subject to Conjecture \ref{EisCompat} \cite{HM}, \cite{HSX}, whose proof will appear in forthcoming work \cite{HHM}.
Conjecture \ref{GenerousConj} is known for $\text{GL}(n)$ \cite[Theorem 8.6]{Zou25}.
So currently Theorem \ref{MainThm} is most significantly conditional on Conjecture \ref{EisCompat}.
Granting that, the theorem holds for $\text{GL}(n)$ and is close to being fully unconditional for many classical groups, subject to extending Conjecture \ref{GenerousConj} to more cases, which may be possible now given the significant recent advances in understanding categorical local Langlands and its compatibility with classical local Langlands for many classical groups.
    
\end{rem}

\begin{rem}

Since $\text{Coh}(\text{Par}_G^\phi)$ is generated by $\nu_{\phi*}\overline{\nabla}_\rho$ by Lemma \ref{NablaQExcep}, the theorem shows that $c_\psi|_{D(\text{Bun}_G)^{\omega,\phi}}$ is essentially surjective onto $\text{Coh}(\text{Par}_G^\phi)$.
In every case where $c_\psi|_{D(\text{Bun}_G)^{\omega,\phi}}$ is known to be fully faithful it is already known to be an equivalence, but what Theorem \ref{MainThm} shows is that essential surjectivity can also be shown by explicit computation.
    
\end{rem}

\begin{proof}
We will freely use the notation reviewed in Section \ref{FSSect} surrounding geometric Eisenstein series.
Choose an opposite standard parabolic $P$ of $G$ with Levi $M$, such that there exists an antidominant basic element $b_\alpha\in B(M)$ mapping to $b$ and a supercuspidal representation $\sigma$ of $M_{b_\alpha}$ such that the cuspidal support of $\pi$ is $$(w_\alpha(M)_{w_\alpha(b_\alpha)}, \sigma^{w_\alpha})$$
Let $\phi_M$ be the Fargues-Scholze parameter of $\sigma$.
By \cite[Corollary IX.7.3]{FS} $\phi_M$ agrees with $\phi$ after composing with the inclusion of $^LM$ into $^LG$.
From now on we will overload notation and also use $\phi$ and $\phi_M$ to denote the unique closed points of $\text{Par}_G$, resp. $\text{Par}_M$ mapping to $\phi$, resp. $\phi_M\in\text{Par}_G^{\text{spec}}$, resp. $\text{Par}_M^{\text{spec}}$.
We hope this will not cause too much confusion, and the use of the symbols should be clear from context (for example in the sentence ``$\pi$ has Fargues-Scholze parameter $\phi$,'' $\phi\in\text{Par}_G^{\text{spec}}$, and in the expression ``$i_\phi:BS_\phi\hookrightarrow\text{Par}_G$,'' $\phi\in \text{Par}_G$). 
Fix orbit representatives for $\phi$ and $\phi_M$, in other words honest L-parameters, which agree exactly under the inclusion $\!^LM\hookrightarrow\!^LG$. 
With these choices we give $S_\phi$ the structure of a closed subgroup of $\hat{G}$, and $S_{\phi,M}$ the structure of a closed subgroup of $\hat{M}$, contained in $S_\phi$ after including $^LM$ into $^LG$ (viewing $\hat{M}$ and $\hat{G}$ as the respective identity components).

By the assumed compatibility between Fargues-Scholze and semisimple classical parameters, Property \ref{ClassLLCProp}(2) and genericity of $\phi$ together imply that $\phi_M$ is essentially discrete and generic, hence supercuspidal, hence generous for $M$.  
By Conjecture \ref{GenerousConj} there is an irreducible representation $\tau$ of $S_{\phi,M}$ such that $c_{\psi_M}(i_{b_\alpha!}\sigma) = i_{\phi_M*}\tau$.  

By Conjecture \ref{EisCompat}, $$c_\psi(\text{Eis}_{\overline{P}}^G(i_{b_\alpha!}\sigma))\cong\text{Eis}^{\hat{G}, \text{spec}}_{\hat{P}}(c_{\psi_{M}}(i_{b_\alpha!}\sigma))\cong \text{Eis}^{\hat{G}, \text{spec}}_{\hat{P}}(i_{\phi_M*}\tau)$$
By Proposition \ref{EisCalc}, $$\text{Eis}_{\overline{P}}^G(i_{b_\alpha!}\sigma)\cong i_{b!}^{\textup{ren}}i_{\overline Q_{b,w_\alpha}}^{G_b}(\sigma^{w_\alpha})$$
Since the semisimple classical parameter of $\pi$ agrees with its Fargues-Scholze parameter, Theorem \ref{GenCompRed} applies, so $i_{\overline Q_{b,w_\alpha}}^{G_b}(\sigma^{w_\alpha})$ is completely reducible, and $\pi$ is the $\theta_\pi$-isotypic component of $i_{\overline Q_{b,w_\alpha}}^{G_b}(\sigma^{w_\alpha})$ for some irreducible representation $\theta_\pi$ of $\text{End}_G(i_{\overline Q_{b,w_\alpha}}^{G_b}(\sigma^{w_\alpha}))$.
By Conjecture \ref{GenerousConj} $\tau|_{S_{\phi,M}^\circ}$ is dominant, and by Proposition \ref{SpecEisCalc}, $$\text{Eis}^{\hat{G}, \text{spec}}_{\hat{P}}(i_{\phi_M*}\tau)\cong\bigoplus\limits_{\rho\in\text{Irr}(S_\phi)}\theta_\rho\otimes\nu_{\phi*}\overline{\nabla}_\rho$$ where $\theta_\rho$ is an irreducible representation of $$\text{End}^0(\text{Eis}^{\hat{G}, \text{spec}}_{\hat{P}}(i_{\phi_M*}\tau))\cong \text{End}^0(\text{Ind}_{S_{\phi,P}}^{S_\phi}(\tau))$$
where the isomorphism follows from Proposition \ref{SpecEisCalc} and Lemma \ref{NablaQExcep}.
By the assumption that $c_\psi$ is fully faithful on the generic locus of $D(\text{Bun}_G)$, $$\text{End}^0(\text{Eis}_{\overline{P}}^G(i_{b_\alpha!}\sigma))\overset{c_\psi}{\cong} \text{End}^0(\text{Eis}^{\hat{G}, \text{spec}}_{\hat{P}}(i_{\phi_M*}\tau))$$
So there exists a $\rho\in\text{Irr}(S_\phi)$ such that pullback under this isomorphism identifies $\theta_\pi$ with $\theta_\rho$.
It follows that $c_\psi$ identifies the $\theta_\pi$-isotypic component of $\text{Eis}_{\overline{P}}^G(i_{b_\alpha!}\sigma)$ with the $\theta_\rho$-isotypic component of $\text{Eis}^{\hat{G}, \text{spec}}_{\hat{P}}(i_{\phi_M*}\tau)$, ie $$c_\psi(i_{b!}^{\textup{ren}}\pi)\cong\nu_{\phi*}\overline{\nabla}_\rho$$ proving the first equivalence in part i) of the theorem.

Every step of the process done to construct $\rho$ from $(b,\pi)$ is reversible in the following way: starting with $\rho$ we can choose a highest weight $\lambda$ of $\rho|_{S_\phi^\circ}$ (unique up to $S_\phi$-conjugacy) and define $\tau$ to be an irreducible representation of $S_{\phi,M}$, unique up to $S_\phi$-conjugacy since $A_M\trianglelefteq A$, such that $\tau$ is a direct sum of $\lambda$-weight spaces and $\rho\subset\text{Ind}_{S_{\phi,P}}^{S_\phi}(\tau)$.
Then by Conjecture \ref{GenerousConj} there is a pair $(b_\alpha,\sigma)$ such that $c_\psi(i_{b_\alpha!}\sigma) = i_{\phi_M*}\tau$ and $b_\alpha$ is antidominant.
Again applying Conjecture \ref{EisCompat} it follows that $$\overline{\nabla}_{\text{Ind}_{S_{\phi,P}}^{S_\phi}(\tau)}\cong c_\psi(i_{b!}^{\textup{ren}}i_{\overline Q_{b,w_\alpha}}^{G_b}(\sigma^{w_\alpha}))$$
Again by Theorem \ref{GenCompRed} $i_{\overline Q_{b,w_\alpha}}^{G_b}(\sigma^{w_\alpha})$ is completely reducible, and so the indecomposable direct summands of $i_{b!}i_{\overline Q_{b,w_\alpha}}^{G_b}(\sigma^{w_\alpha})$ are all of the form $i_{b!}\pi$ for some irreducible representation $\pi$ of $G_b(F)$.  
To $\rho$, we assign the $\pi$ such that under the isomorphism $$\text{End}^0(i_{b!}^{\textup{ren}}i_{\overline Q_{b,w_\alpha}}^{G_b}(\sigma^{w_\alpha}))\overset{c_\psi}{\cong} \text{End}^0(\overline{\nabla}_{\text{Ind}_{S_{\phi,P}}^{S_\phi}(\tau)})$$ induced by $c_\psi$, $\pi$ is the $\theta_\rho$-isotypic component of $i_{b!}i_{\overline Q_{b,w_\alpha}}^{G_b}(\sigma^{w_\alpha})$.
This gives part iii).

Returning to the isomorphisms in part i) and the ``forward direction'' where the pair $(b,\pi)$ is the given initial input, we now use the Levi $w_\alpha(M)$ and supercuspidal representation $w_\alpha(\sigma)$.
Let $\!^\alpha P$ be the opposite standard parabolic of $G$ with Levi $w_\alpha(M)$.
Then by Proposition \ref{EisCalc} $$\text{Eis}_{\overline {\!^\alpha\! P}}^G(i_{w_\alpha(b_\alpha)!}\sigma^{w_\alpha})\cong i_{b\sharp}^{\textup{ren}}i_{\overline {\!^\alpha Q}_{b,1}}^{G_b}(\sigma^{w_\alpha})$$ 
By \cite[Corollary 6.6]{KV} and Theorem \ref{GenCompRed} the parabolic induction of $\sigma^{w_\alpha}$ is independent (up to isomorphism) of the inducing parabolic, so 
$i_{\overline Q_{b,w_\alpha}}^{G_b}(\sigma^{w_\alpha})\cong i_{\overline {\!^\alpha Q}_{b,1}}^{G_b}(\sigma^{w_\alpha})$ and therefore $$ \text{Eis}_{\overline {\!^\alpha\! P}}^G(i_{w_\alpha(b_\alpha)!}\sigma^{w_\alpha})\cong i_{b\sharp}^{\textup{ren}}i_{\overline Q_{b,w_\alpha}}^{G_b}(\sigma^{w_\alpha})\cong \bigoplus\limits_{\pi}\theta_\pi\otimes i_{b\sharp}^{\textup{ren}}\pi$$ where $\pi$ runs over the irreducible subrepresentations of $i_{\overline Q_{b,w_\alpha}}^{G_b}(\sigma^{w_\alpha})$.

Let $\lambda = \tau|_{S_{\phi,M}^\circ}$, which recall is dominant, and let $w_{0,\lambda}$ be the longest length Weyl group element of $S_\phi^\circ$ fixing $\lambda$, in other words the longest length Weyl group element of $S_\phi\cap \widehat{G_{\{b_\alpha\}}}$.
Since $w_\alpha(b_\alpha)$ is dominant and $S_{\phi,M}^\circ = Z(\hat M)^\circ$, $\hat w_\alpha w_{0,\lambda} (S_{\phi,P})$ is the standard Borel of $S_{\hat w_\alpha\phi}$, and so $\hat w_\alpha w_{0,\lambda}w_0(S_{\phi,P})$ is the opposite standard Borel of $S_{\hat w_\alpha \phi}$, which is $S_{\hat w_\alpha \phi,\ \!^\alpha P}$.
Note that $\hat w_\alpha\in W_{\hat G}$ whereas $w_0\in W_{S_\phi}$, so to multiply them we implicitly use the map $W_{S_\phi}\to W_{\hat G}$ induced by the inclusion $S_\phi\hookrightarrow\hat G$ (and the fixed choice of positivity on roots).
Then
\begin{align*}
\text{Eis}^{\hat{G}, \text{spec}}_{\widehat{\!^\alpha P}}(i_{\hat w_\alpha\phi_M*}\hat w_\alpha\tau) &\cong \text{Eis}^{\hat{G}, \text{spec}}_{\hat w_\alpha w_{0,\lambda}w_0\hat P}(i_{\hat w_\alpha\phi_M*}\hat w_\alpha\tau) \\ &= \text{Eis}^{\hat{G}, \text{spec}}_{\hat w_\alpha w_{0,\lambda}w_0\hat P}(i_{\hat w_\alpha w_{0,\lambda}w_0\phi_M*}\hat w_\alpha w_{0,\lambda}w_0w_0\tau)\\ &\cong \text{Eis}^{\hat{G}, \text{spec}}_{\hat P}(i_{\phi_M*} w_0\tau)\\ &\cong \bigoplus\limits_{\rho'\in\text{Irr}(S_\phi)}\theta_{\rho'}\otimes\nu_{\phi*}\overline{\Delta}_{\rho'}
\end{align*}
where in the second line we made use of the facts that $w_{0,\lambda}w_0\phi_M = \phi_M$ since $w_{0,\lambda},w_0\in S_\phi$, and that $w_{0,\lambda}$ acts trivially on $\tau$; in the last equivalence we used Proposition \ref{SpecEisCalc}; and we index over the variable $\rho'$ instead of $\rho$ because $\rho$ is already taken as the irreducible representation of $S_\phi$ for which $c_\psi(i_{b!}\pi)\cong\nu_{\phi*}\overline{\nabla}_\rho$.
By Proposition \ref{FSWeylEquiv} $$c_{\psi_{w_\alpha(M)}}(i_{w_\alpha(b_\alpha)!}\sigma^{w_\alpha})\cong i_{\hat{w}_\alpha(\phi_M)*}\hat{w}_\alpha(\tau)$$
Then using geometric Eisenstein series compatibility and the fact that $\text{End}^0(i_{b\sharp}\pi)=\overline{\mathbb{Q}}_\ell$ we conclude $$c_\psi(i_{b\sharp}^{\textup{ren}}\pi)\cong\nu_{\phi*}\overline{\Delta}_{\rho'}$$ for some irreducible representation $\rho'$.
Since there is a nonzero map from $i_{b\sharp}\pi\to i_{b!}\pi$, by full faithfulness of $c_\psi$ there is a nonzero map from $\overline{\Delta}_{\rho'}\to\overline{\nabla}_\rho$.
By Lemma \ref{NablaQExcep} $\rho\cong\rho'$, giving the second desired isomorphism in part i).

By Proposition \ref{PerverseEqExotic} the (-d)-shifted connective perverse coherent subcategory of $\text{Coh}(\text{Par}_G^\phi)$ is generated by $\{\overline{\Delta}_\rho[d]\ |\ d\geq0\}$ and the coconnective ($-d$)-shifted subcategory is generated by $\{\overline{\nabla}_\rho[d]\ |\ d\leq0\}$.
On $D(\text{Bun}_G)^{\text{fin},\phi}$ the subcategories of hadal connective and coconnective objects are respectively generated by $$\{i_{b\sharp}^{\textup{ren}}\pi[d]\ |\ d\geq0\}$$ and $$\{i_{b!}^{\textup{ren}}\pi[d]\ |\ d\leq0\}$$
So by (1) and (2) $c_\psi$ is t-exact with respect to the hadal t-structure on $D(\text{Bun}_G)^\phi$ and the $(-d)$-shifted perverse-coherent t-structure on $\text{Coh}(\text{Par}_G^\phi)$, showing ii).
Furthermore by the discussion above Proposition \ref{PerverseEqExotic} $\overline{\nabla}_\rho[d]$ and $\overline{\Delta}_\rho[d]$ are perverse-coherent, so $i_{b\sharp}\pi$ and $i_{b!}\pi$ are hadal.
Therefore $$c_\psi(i_{b\sharp !}^{\textup{ren}}\pi)\cong c_\psi(\text{Im}(i_{b\sharp}^{\textup{ren}}\pi\to i_{b!}^{\textup{ren}}\pi))\cong \text{Im}(\nu_{\phi*}\overline{\Delta}_\rho\to\nu_{\phi*}\overline{\nabla}_\rho)\cong\nu_{\phi*}IC_\rho$$

Finally, to prove $c_\psi(i_{b!*}^{\textup{ren}}\pi)\cong\nu_{\phi*}\delta_\rho$ we need the fact that $c_\psi$ is t-exact with respect to the perverse t-structure on $D(\text{Bun}_G)^{\text{fin},\phi}$ and the standard coherent t-structure on $\text{Coh}(\text{Par}_G^\phi)^{\text{fin}}$.
This is proven in Corollary \ref{PTruncStability} below, which uses the first two isomorphisms in part i) of this theorem.

\begin{lem}\label{FinGenPres}

Let $F\in \textup{Perv}(\textup{Bun}_G)^{\textup{fin,gen}}$, and let $G\in \textup{Perv}(\textup{Bun}_G)$ be a subobject of $F$.
Then $G$ is finite and generic.
    
\end{lem}

\begin{proof}

Let the annihilating ideal of $c_\psi(F)$ in $\text{Par}_G^{\text{spec}}$ be generated by $f_1,..., f_r$.
Since $G$ maps to $F$ via a monomorphism in $\text{Perv}(\text{Bun}_G)$, $c_\psi(G)$ is also annihilated by $(f_1,...,f_r)$ (using $\text{IndPerf}(\text{Par}_G)$ linearity of $c_\psi$), which proves genericity.
Let $\mathcal{O}$ denote the ring of global sections on $\text{Par}_G^{\text{spec}}$.
A direct computation using induction, monoidality of the spectral action, and the fact that $\mathcal{O}*G\cong G$ and $f_i$ acts on $G$ as the $0$ map shows $\bigotimes\limits_{i=1}^r(\mathcal{O}/f_i)*G$ is a finite direct sum of shifts of $G$, with the unshifted summand being precisely $G$, so $H^{p,0}(\bigotimes\limits_{i=1}^r(\mathcal{O}/f_i)*G) \cong G$.
Now express $G$ as a filtered colimit of compact objects, $G\cong \text{colim}\ G_\alpha$.
It follows that $$G\cong H^{p,0}(\bigotimes\limits_{i=1}^r(\mathcal{O}/f_i)*G)\cong \text{colim}\ H^{p,0}(\bigotimes\limits_{i=1}^r(\mathcal{O}/f_i)*G_\alpha)$$
Since $G_\alpha$ is compact so is $\bigotimes\limits_{i=1}^r(\mathcal{O}/f_i)*G_\alpha$ (\cite[Corollary X.1.3]{FS}), and since its image under $c_\psi$ is supported on the finitely many points of $\text{Par}_G^{\text{spec}}$ in the vanishing set of $(f_1,...f_r)$, $\bigotimes\limits_{i=1}^r(\mathcal{O}/f_i)*G_\alpha$ is also ULA (using \cite[Lemma 1.6.4]{BeijingNotes}), hence finite.
By Corollary \ref{PTruncStability} (which only relies on the first two isomorphisms of part i)) the image $F_\alpha$ of $\bigotimes\limits_{i=1}^r(\mathcal{O}/f_i)*G_\alpha$ in $F$ is finite for all $\alpha$, and since $G$ is a subobject of $F$ we may write $G=\text{colim}\ F_\alpha$.
By Corollary \ref{PTruncStability} $\text{Perv}(\text{Bun}_G)^{\text{fin,gen}}$ is Noetherian, so, all being subobjects of $F$, there must be some $\alpha$ for which $G\cong  F_\alpha$.
    
\end{proof}

It follows that $i_{b!*}\pi$, the unique irreducible quotient of $i_{b!}\pi$ in $\textup{Perv}(\textup{Bun}_G)$, is finite.
And by Corollary \ref{PTruncStability} and Proposition \ref{GenSheafProp} it follows that $c_\psi(i_{b!*}\pi)\cong\nu_{\phi*}\delta_\rho$.

\end{proof}

Theorem \ref{MainThm} can be used to attach nonsemisimple Fargues-Scholze parameters to representations $\pi$ such that the usual Fargues-Scholze parameter of $\pi$ is generic.

\begin{defn}\label{fullFSParam} \textup{(Full Fargues-Scholze Parameters on the Generic Locus)}\ 
\newline
Given a smooth irreducible representation $\pi$ of $G_b(F)$, define the full Fargues-Scholze parameter $\phi_\pi$ to be the unique point of $\text{Par}_G$ such that $\phi_\pi^{\text{ss}} = \phi_\pi^{\text{FS}}$, and such that $$c_\psi(i_{b\sharp!}^{\textup{ren}}\pi)\cong IC(\phi(\text{Fr})^u,\tilde{\rho})[-d]$$ for some (necessarily unique) irreducible representation $\tilde{\rho}$ of $C_{S_\phi}(\phi(\text{Fr})^u)$.
This gives a bijective map $(b,\pi)\mapsto(\phi_\pi, \tilde\rho)$ from pairs consisting of $b\in B(G)$ and $\pi$ with generic Fargues-Scholze parameter to pairs consisting of a generic parameter and an irreducible representation of its centralizer.

\end{defn}

\section{Applications}

\subsection{Exactness of Hecke Operators at Semisimple Generic Parameters}

In this section we will prove that Hecke operators are perverse t-exact at semisimple generic parameters, following \cite[Section 2.4]{BeijingNotes} which studies t-exactness at the trivial parameter.

\begin{lem}\label{CostFilt}
For any $\rho\in \mathrm{Irr}(S_\phi)$
and any $V\in\mathrm{Rep}(S_\phi)$, the sheaf $S_V\otimes \overline{\nabla}_{\rho}\in\mathrm{Coh}(\mathcal{N}_{S_\phi}/S_\phi)$
admits a finite filtration with graded pieces of the form $\overline{\nabla}_{\rho'}[n]$
where $\rho'\in \mathrm{Irr}(S_\phi)$ and $n\geq0$. 
Here $S_V$ denotes the pullback of $V$ from $*/S_\phi$ to $\mathcal{N}_{S_\phi}/S_\phi$ along the canonical projection.
\end{lem}

\begin{proof}

Let $\lambda$ be a highest weight of $\rho$.
Then $\overline{\nabla}_\rho \cong p^{\lambda}_*(E_\rho\otimes\overline{\nabla}_\lambda)$ is a direct summand of $$p^{\lambda}_*(\text{Ind}_{k^\alpha[\text{Id}_A]}^{k^\alpha[A]}\otimes\overline{\nabla}_\lambda)\cong p^\circ_*\mu^\circ_*\mathcal{O}(\lambda)$$ so it suffices to prove the desired result for $p^\circ_*\mu^\circ_*\mathcal{O}(\lambda)$.
Letting $\tilde{S}_V:=(p^\circ\mu^\circ)^*S_V$, by the projection formula $$p^\circ_*\mu^\circ_*(\mathcal{O}(\lambda))\otimes S_V\cong (p^\circ\mu^\circ)_*(\mathcal{O}(\lambda)\otimes \tilde{S}_V)$$  
So it suffices to show $\mathcal{O}(\lambda)\otimes \tilde{S}_V$ admits a finite filtration with graded pieces of the form $\mathcal{O}(\lambda_i)[n]$ where $n\geq0$ and the $\lambda_i$ are all dominant.  But this is shown in \cite[Section 2.4]{BeijingNotes}.

\end{proof}

\begin{prop}\label{HeckePExactness}\cite[Conjecture 2.4.1]{BeijingNotes}, \cite[Conjecture 6.4]{HL}
Under the conditions of Theorem \ref{MainThm}, the Hecke action is perverse t-exact on $D(\mathrm{Bun} _G)^{\mathrm{ULA},\phi}$, and therefore $D(\mathrm{Bun} _G)^{\mathrm{ULA},\mathrm{gen}} $.
    
\end{prop}

\begin{proof}
Following the proof of \cite[Theorem 2.4.4]{BeijingNotes} it suffices to show right t-exactness.
    
\noindent$D^{p,\leq0}(\text{Bun}_G)^{\text{ULA},\phi}$ is generated (under extensions and colimits) by $i_{b!}\pi[n]$ over all pairs $(b,\pi)$ such that $\pi$ has Fargues-Scholze parameter $\phi$, $n\geq0$.
For all $V\in\text{Rep}(\hat{G})$, $$c_\psi(T_V*i_{b!}\pi[n])\cong \nu_{\phi*}(S_V\otimes\overline{\nabla}_\rho)[n]$$ and we conclude by applying Lemma \ref{CostFilt} and passing back to $D(\text{Bun}_G)$ using Theorem \ref{MainThm}.  
Perverse t-exactness on $D(\mathrm{Bun} _G)^{\mathrm{ULA},\mathrm{gen}} $ follows by varying $(b,\pi)$ over all pairs for which the Fargues-Scholze parameter of $\pi$ is generic, as opposed to only those $\pi$ with Fargues-Scholze parameter $\phi$.
\end{proof}

\begin{cor}\label{PTruncStability}\cite[Conjecture 2.5.1]{BeijingNotes}
Under the assumptions of Theorem \ref{MainThm}, $D(\textup{Bun}_G)^{\textup{fin},\textup{gen}}$ is stable under perverse truncation operators, and $c_\psi|_{D(\textup{Bun}_G)^{\textup{fin},\textup{gen}}}$ is t-exact with respect to the perverse t-structure on $D(\textup{Bun}_G)^{\textup{fin},\textup{gen}}$ and the standard t-structure on $\textup{Coh}(\textup{Par}_G^{\textup{gen}})^{\textup{fin}}$.

\end{cor}

\begin{proof}

This follows from \cite[Proposition 2.5.2]{BeijingNotes}, \cite[Proposition 5.5.7]{HM}, and Proposition \ref{HeckePExactness}.
    
\end{proof}

\subsection{Hecke Eigensheaves at Generic Parameters}\label{HeckEigenSect}

We recall that a Hecke eigensheaf for an L-parameter $\phi\in\text{Par}_G$ is an object $\mathcal{F}\in D(\text{Bun}_G)$ such that for all finite sets $I$ and all $V\in\text{Rep}((\hat{G}\rtimes Q)^I)$, $$T_V*\mathcal{F}\cong \mathcal{F}\otimes_{\overline{\mathbb{Q}}_\ell} V_\phi \in D(\text{Bun}_G)\otimes\text{Rep}^{\text{cont}}_{\overline{\mathbb{Q}}_\ell}(W_F^I)$$ functorial in $V$ and $I$ and respecting the tensor product of $\text{Rep}((\hat{G}\rtimes Q)^I) $.
Here $V_\phi$ denotes the underlying vector space of $V$ equipped with the $W_F$ action coming from $\phi$.
On the spectral side, by the $\text{Perf}(\text{Par}_G)$-linearity of $\textbf{L}_\psi$ a Hecke eigensheaf for $\phi$ corresponds under the Fargues-Scholze correspondence to a sheaf $\mathcal{P}$ on $\text{Par}_G$ such that $$\mathcal{P}\otimes \mathcal{V}\cong\mathcal{P}\otimes_{\overline{\mathbb{Q}}_\ell}  V_\phi\in\text{IndCoh}(\text{Par}_G)\otimes\text{Rep}^{\text{cont}}_{\overline{\mathbb{Q}}_\ell}(W_F^I)$$ where $\mathcal{V}$ is by definition the vector bundle on $\text{Par}_G$ associated to $V$.

Let $\phi$ be an L-parameter with generic semisimplification and centralizer $S_\phi\subset\hat{G}$.
We denote its regular representation by $\mathcal{R}(S_\phi)$.
Recall $\mathcal{R}(S_\phi)$ has the property that for any representation $W$ of $S_\phi$, $$\mathcal{R}(S_\phi)\otimes W\cong \mathcal{R}(S_\phi)\otimes W_{\text{triv}}$$ where $W_{\text{triv}}$ denotes $W$ with the trivial action of $S_\phi$.  
This follows from the projection formula applied to the map $*\to */S_\phi$.  
Furthermore when $S_\phi$ is reductive every nonzero representation with this property is a direct sum of copies of $\mathcal{R}(S_\phi)$:

\begin{lem}\label{AllEigenReg}

Let $V$ be a representation of a reductive algebraic group $G$ (over an algebraically closed field of characteristic $0$) such that $V\otimes W\cong V\otimes W_{\textup{triv}}$ for all irreducible, and hence all, representations of $G$.
Then $V\cong\mathcal{R}(G)^{\oplus c}$ for some cardinal $c$.
    
\end{lem}

\begin{proof}

Let $W\cong V_\nu$ be an irreducible representation of $G$, and let $1_G$ denote the trivial representation.
For every irreducible representation $V_\lambda$, $\text{Hom}_G(1_G,V_\lambda\otimes V_\nu)\cong \text{Hom}_G(V_{\nu}^*,V_\lambda)$ is one dimensional if $V_\nu^*\cong V_\lambda$ and $0$ otherwise, so letting $m_V(\lambda)$ denote the multiplicity of $V_\lambda$ in $V$, $\text{Hom}_G(1_G, V\otimes V_\nu)$ has dimension $m_V(\nu^*)$.
On the other hand $\text{Hom}(1_G, V\otimes V_{\nu,\text{triv}})$ has dimension $m_V(1_G)\text{dim}(V_\nu) = m_V(1_G)\text{dim}(V_\nu^*)$.
Varying over all $\nu$ we find $V\cong \mathcal{R}(G)^{\oplus m_V(1_G)}$.

\end{proof}

For any L-parameter $\phi$, an application of the projection formula to the map $i_\phi:*/S_\phi\to\text{Par}_G$ shows that $i_{\phi*}\mathcal{R}(S_\phi)$ is a spectral Hecke eigensheaf for $\phi$.
If $\phi$ is semisimple then by Proposition \ref{ssLparamprop} $S_\phi$ is reductive, and so by the above Lemma every spectral Hecke eigensheaf for $\phi$ is of the form $i_{\phi*}\mathcal{R}(S_\phi)^{\oplus c}$.
But even for spectral Hecke eigensheaves which are relatively easy to describe, it is not always clear what sheaves on $\text{Bun}_G$ they should correspond to under the Fargues-Scholze correspondence.
The goal of this section is to use the main results of this paper to begin answering this question in the setting where $\phi$ is generic (and not necessarily semisimple).

To simplify notation let $\mathcal{N} := \mathcal{N}_{S_{\phi^{\text{ss}}}}$, $ S:= S_{\phi^{\text{ss}}}$, $\nu = \nu_{\phi^{\text{ss}}}$, and let $O$ be the orbit of $\phi(\text{Fr})^u$ in $\mathcal{N}$.
In \cite[Theorem 4.2]{PervCoh} Arinkin and Bezrukavnikov prove that the category of perverse-coherent sheaves on $O$ is equivalent to a certain subcategory of perverse-coherent sheaves on $\overline{O}$ satisfying some extra vanishing conditions on $\partial\overline{O}$.
Letting $j:O/S\hookrightarrow\overline{O}/S$ denote the open inclusion, the backward map inducing the equivalence is $j^*$, and $j_{!*}$ is by definition its inverse.
One can show $j_{!*}$ satisfies a projection formula.
Indeed, if $V$ is any vector bundle on $\overline{O}/S$ and $P$ is a perverse-coherent sheaf on $\mathcal{O}/S$, then $j^*(j_{!*}(P)\otimes V)$ and $j^*(j_{!*}(P\otimes j^*V))$ are both isomorphic to $P\otimes j^*V$, and therefore $j_{!*}(P)\otimes V\cong j_{!*}(P\otimes j^*V) $.
We apply this formula to the setting where $P = \mathcal{R}(S_\phi)[\text{dim}(O)/2]$ and $V = \mathcal{V}|_{\overline{O}/S}$.
Note $\mathcal{R}(S_\phi)$ is an ind-coherent sheaf, so the functors applied to it below will be those ind-extended from their coherent incarnations.
We define the intersection cohomology sheaf $IC(O,\mathcal{R}(S_\phi)):= j_{!*}\mathcal{R}(S_\phi)[\text{dim}(O)/2]$.

\begin{prop}

The complex $\nu_{*}IC(O,\mathcal{R}(S_\phi))$ is a spectral Hecke eigensheaf with eigenparameter $\phi$.  
Consequently $a_\psi(\nu_*IC(O,\mathcal{R}(S_\phi)))[-d]$ is a Hecke eigensheaf on $D(\textup{Bun}_G)$ with eigenparameter $\phi$.  
We refer to this sheaf as $H_\phi$.
\end{prop}

\begin{proof}

In the expressions below we will slightly abuse notation and use $\nu$ to denote $\nu|_{\overline{O}}$, so $\nu\circ j$ in particular will be the inclusion of $O/S$ into $\text{Par}_G$.
By the projection formula $$\nu_{*}j_{!*}\mathcal{R}(S_\phi) \otimes \mathcal{V}\cong \nu_{*}j_{!*}(\mathcal{R}(S_\phi)\otimes j^* \nu^*\mathcal{V})\cong \nu_{*}j_{!*}(\mathcal{R}(S_\phi)\otimes V_\phi|_{S_\phi})$$  $$\cong \nu_{*}j_{!*}(\mathcal{R}(S_\phi)\otimes V_{\phi,\text{triv}})\cong \nu_{*}j_{!*}(\mathcal{R}(S_\phi))\otimes_{\overline{\mathbb{Q}_\ell}}V_\phi$$ for all $V\in\text{Rep}(\hat{G}\rtimes Q)$, proving $\nu_*IC(O,\mathcal{R}(S_\phi))[-d]$ is a spectral Hecke eigensheaf.

\end{proof}

\begin{rem}

The $-d$ degree shift is a normalization that will make the comparison between $H_\phi$ and the sheaves studied in Theorem \ref{MainThm} more straightforward.
    
\end{rem}

\begin{rem}

By Proposition \ref{EquivGenDef} $\text{Par}_G^{\text{gen}}$ is smooth, so $\text{QCoh}(\text{Par}_G^{\text{gen}})$ and $\text{IndCoh}(\text{Par}_G^{\text{gen}})$ agree.
Therefore when $\textbf{L}_\psi$ is an equivalence (as is conjectured to always be the case), $a_\psi|_{\text{QCoh}(\text{Par}_G^{\text{gen}})}$ is the inverse of $\textbf{L}_\psi|_{D(\text{Bun}_G)^{\text{gen}}}$.
    
\end{rem}

The following theorem, broadly, states that the intermediate extension sheaves $i_{\sharp!}\pi$ on $D(\text{Bun}_G)$ can be used to build Hecke eigensheaves.

\begin{thm}\label{HeckeEigenThm}
Let $M_\phi$ be the quotient of $S_\phi$ by its unipotent radical $U_\phi$.  Then $H_\phi$ has a direct sum decomposition $$H_\phi\cong\bigoplus\limits_{\rho\in\textup{Irr}(M_\phi)}H_{\phi,\rho}\otimes V_{\rho^*}$$ where each $H_{\phi,\rho}$ is indecomposable and $V_{\rho^*}$ is the underlying vector space of $\rho^*$.  
Furthermore, under the conditions of Theorem \ref{MainThm}, $H_{\phi,\rho}$ is (ind-)hadal for all $\rho$, and admits a filtration by (ind-)hadal subsheaves $H_{\phi,\rho}^k$ indexed by nonnegative integers.  
The $k$th graded complex $G^k_{\phi,\rho}:=H_{\phi,\rho}^k/H_{\phi,\rho}^{k-1}$ surjects (as a hadal sheaf) onto a finite direct sum of complexes of the form $i_{b\sharp!}\pi$, where $\pi$ has full Fargues-Scholze parameter $\phi$, and the kernel of this map lies in the subcategory of $D(\textup{Bun}_G)^\phi$ generated by complexes of the form $i_{b\sharp!}\pi'$, where the full Fargues-Scholze parameter of $\pi'$ is a strict specialization of $\phi$ in $\textup{Par}_G$.
\end{thm}

\begin{proof}
    
By Levi's Theorem there is an isomorphism between $\mathcal{O}(S_\phi)$ and $\mathcal{O}(M_\phi)\otimes\mathcal{O}(U_\phi)$, with an element $s'=m'u'\in S_\phi$ acting on $g\otimes f\in\mathcal{O}(M_\phi)\otimes\mathcal{O}(U_\phi)$ via the formula $s\cdot (g(m)\otimes f(u)) = g(mm')\otimes f((m')^{-1}um'u')$.
In particular the action of $S_\phi$ on $\mathcal{O}(M_\phi)$ is the inflation of $\mathcal{R}(M_\phi)$ to $S_\phi$, which we denote $\mathcal{R}_{S_\phi}(M_\phi)$, and we denote the above action of $S_\phi$ on $\mathcal{O}(U_\phi)$ given by Levi component conjugation and unipotent component right translation by $\mathcal{R}_{S_\phi}(U_\phi)$ (so $m'u'\cdot f(u) := f((m')^{-1}um'u')$).

By the algebraic Peter-Weyl theorem, $$\mathcal{R}(M_\phi)\cong\bigoplus\limits_{\rho\in\text{Irr}(M_\phi)}\rho\otimes V_{\rho^*}$$ where the $M_\phi$ action is the usual one on $\rho$ and trivial on $V_{\rho^*}$. 
So $$\mathcal{R}(S_\phi)\cong\mathcal{R}_{S_\phi}(M_\phi)\otimes\mathcal{R}_{S_\phi}(U_\phi)\cong\bigoplus\limits_{\rho\in\text{Irr}(M_\phi)}(\rho\otimes V_{\rho^*})\otimes\mathcal{R}_{S_\phi}(U_\phi)\cong\bigoplus\limits_{\rho\in\text{Irr}(M_\phi)}(\rho\otimes\mathcal{R}_{S_\phi}(U_\phi))\otimes V_{\rho^*}$$ 
Let $H_{\phi,\rho}:=a_\psi(\nu_{*}(IC(O,\rho\otimes\mathcal{R}_{S_\phi}(U_\phi))))[-d]$.  
It is immediate that $$H_\phi \cong \bigoplus\limits_{\rho\in\text{Irr}(M_\phi)}H_{\phi,\rho}\otimes V_{\rho^*}$$

\noindent From now on we write $\mathcal{R}_{S_\phi,\rho}(U_\phi)$ as shorthand for $\rho\otimes\mathcal{R}_{S_\phi}(U_\phi) $.

\begin{lem}\label{SocIrred}

$\textup{Soc}(\mathcal{R}_{S_\phi,\rho}(U_\phi))$ is irreducible, and 
therefore $\mathcal{R}_{S_\phi,\rho}(U_\phi)$ and $H_{\phi,\rho}$ are indecomposable.
    
\end{lem}

\begin{proof}

Let $W$ be an irreducible subrepresentation of $\textup{Soc}(\mathcal{R}_{S_\phi,\rho}(U_\phi))$.
By Engel's Theorem its space of $U_\phi$-invariants $W^{U_\phi}$ is nonzero, and since $U_\phi$ is normal in $S_\phi$, $W^{U_\phi}$ is $S_\phi$-invariant, hence equals $W$.
Therefore $W$ is a subrepresentation of the $U_\phi$-invariant subspace of $\textup{Soc}(\mathcal{R}_{S_\phi,\rho}(U_\phi))$.
$U_\phi$ acts trivially on $\rho$, and $\mathcal{R}_{S_\phi}(U_\phi)^{U_\phi}$ = $\overline{\mathbb{Q}}_\ell$, so $(\mathcal{R}_{S_\phi,\rho}(U_\phi))^{U_\phi} = \rho$ is irreducible.
Therefore $W=\rho$ and since $W$ was arbitrarily chosen it follows that $\textup{Soc}(\mathcal{R}_{S_\phi,\rho}(U_\phi)) = \rho$.
Hence $\mathcal{R}_{S_\phi,\rho}(U_\phi)$, and therefore $H_{\phi,\rho}$, are indecomposable.

\end{proof}

For all $\rho$, $\mathcal{R}_{S_\phi,\rho}(U_\phi)$ has a filtration, inherited from $\mathcal{R}_{S_\phi}(U_\phi)$, indexed by the nonnegative integers.
To define it, we start by letting $U(\mathfrak{u}_\phi)$ be the universal enveloping algebra of $\mathfrak{u}_\phi$, and $\mathfrak{m}$ its augmentation ideal (not to be confused with the Lie algebra of $M_\phi$, which will not be used here).
We define $$F^k\mathcal{R}(S_\phi):=\{f\in\mathcal{O}(U_\phi):Df=0\ \forall\ D\in\mathfrak{m}^k\}$$

\begin{lem}
    For all $k\geq0$, $F^k\mathcal{R}(S_\phi)$ is $S_\phi$ stable, and the action of $S_\phi$ on $F_{k+1}\mathcal{R}(S_\phi)/F_{k}\mathcal{R}(S_\phi)$ factors through $M_\phi$.
\end{lem}

\begin{proof}

For the first part of the lemma, let $f\in F^k\mathcal{R}(S_\phi)$ and $s\in S_\phi$.  Then for all $D\in\mathfrak{m}^k$, $D\cdot (s\cdot f) = s\cdot (\text{Ad}_{s^{-1}}D\cdot f) = 0$ since $\text{Ad}_{s^{-1}}D\in\mathfrak{m}^k$.  

For the second part, we need to show the action of $U_\phi$ on $F^{k+1}\mathcal{R}(S_\phi)/F^{k}\mathcal{R}(S_\phi)$ is trivial.
Let $u\in U_\phi$ and let $X_u$ be the element of $\mathfrak{u}_\phi$ such that $\exp(X_u)=u$.
Then $$u\cdot f = \sum\limits_{i=0}^\infty \frac{1}{i!}X_u^i\cdot f$$ where the sum is finite since $X_u$ is nilpotent and $\mathcal{R}(S_{\phi})$ is locally finite.
If $f\in F^{k+1}$ we need to show $u\cdot f-f\in F^k\mathcal{R}(S_\phi)$.
Using the expansion above, $$u\cdot f-f = \sum\limits_{i=1}^\infty \frac{1}{i!}X_u^i\cdot f$$
If $D\in\mathfrak{m}^k$, $$D\cdot\sum\limits_{i=1}^\infty\frac{1}{i!} X_u^i\cdot f = \sum\limits_{i=1}^\infty \frac{1}{i!}DX_u^i\cdot f = 0$$ since $DX_u^i\in \mathfrak{m}^{k+1}$ for all $i\geq1$.

\end{proof}

Since $U_\phi$ acts trivially on $\rho$, $F^k(\mathcal{R}_{S_\phi,\rho}(U_\phi))=\rho\otimes F^k\mathcal{R}_{S_\phi}(U_\phi)$, where $F^k\mathcal{R}_{S_\phi}(U_\phi)$ is the subspace of elements of $\mathcal{R}_{S_\phi}(U_\phi)$ killed by $\mathfrak{m}^k$.
For $k\geq0$ we define $F^kH_{\phi,\rho}$ to be $a_\psi(\nu_*IC(O,F^k\mathcal{R}_{S_\phi,\rho}(U_\phi)))[-d]$.

For all $k\geq1$ and $\rho\in\text{Irr}(M_\phi)$, let $G^{k}\mathcal{R}_{S_\phi,\rho}(U_\phi):= F^{k}\mathcal{R}_{S_\phi,\rho}(U_\phi)/F^{k-1}\mathcal{R}_{S_\phi,\rho}(U_\phi)$, and let $G^{k}IC(O,\mathcal{R}_{S_\phi,\rho})$ be the quotient (in the perverse-coherent category) of the inclusion $IC(O,F^{k-1}\mathcal{R}_{S_\phi,\rho}(U_\phi))\hookrightarrow IC(O,F^{k}\mathcal{R}_{S_\phi,\rho}(U_\phi))$.

\begin{lem}
For all $\tau\in\textup{Irr}(M_\phi)$, let $$W_\tau^k:= \textup{Hom}_{M_\phi}(\tau,G^{k}\mathcal{R}_{S_\phi,\rho}(U_\phi))$$
Then for all $k\geq1$, $G^{k}IC(O, \mathcal{R}_{S_\phi,\rho})$ fits into a short exact sequence of perverse-coherent sheaves $$0\to K^k_{\phi,\rho}\to G^{k}IC(O, \mathcal{R}_{S_\phi,\rho})\to\bigoplus\limits_{\tau\in\textup{Supp}_{M_\phi}(G^{k}\mathcal{R}_{S_\phi,\rho}(U_\phi))}(IC(O,\tau))\otimes W_\tau^k\to 0$$ where $K^k_{\phi,\rho}\in\textup{IndCoh}(\overline{O}/S)$ lies in the subcategory generated under colimits and extensions by sheaves of the form  $IC(O',\tau')$ such that $O'\subset\partial O$ (and $\tau'$ is an irreducible representation of the stabilizer of $O'$ in $S_\phi$).
\end{lem}

\begin{rem}
We are slightly abusing notation, in that as defined above $IC(O',\tau')$ is technically an object of $\textup{IndCoh}(\overline{O'}/S)$ rather than $\textup{IndCoh}(\overline{O}/S)$, but pushforward under closed immersion is exact and fully faithful for perverse-coherent sheaves, so we identify $IC(O',\tau')$ with its image in $\textup{IndCoh}(\overline{O}/S)$ upon inclusion.
    
\end{rem}

\begin{proof}

Since pushforward and $IC$ commute with direct sums  $$\bigoplus\limits_{\tau\in\textup{Supp}_{M_\phi}(G^{k}\mathcal{R}_{S_\phi,\rho}(U_\phi))}IC(O,\tau)\otimes W^k_\tau\cong IC(O,\bigoplus\limits_{\tau\in\textup{Supp}_{M_\phi}(G^{k}\mathcal{R}_{S_\phi,\rho}(U_\phi))}\tau\otimes W^k_\tau)$$ $$\cong IC(O,G^{k}\mathcal{R}_{S_\phi,\rho}(U_\phi))$$
As in the constructible setting, the $IC$ functor preserves monomorphisms and epimorphisms of perverse-coherent sheaves.
So the sequence $$0\to IC(O,F^{k-1}\mathcal{R}_{S_\phi,\rho}(U_\phi))\to IC(O,F^{k}\mathcal{R}_{S_\phi,\rho}(U_\phi))\to IC(O,G^{k}\mathcal{R}_{S_\phi,\rho}(U_\phi))\to 0$$ is exact on the left and the right but potentially not in the middle.
We define $K^k_{\phi,\rho}$ to be the middle cohomology sheaf of the above sequence.
Then the desired exact sequence from the lemma statement follows since $G^{k}IC(O, \mathcal{R}_{S_\phi,\rho})= IC(O,F^{k}\mathcal{R}_{S_\phi,\rho}(U_\phi))/IC(O,F^{k-1}\mathcal{R}_{S_\phi,\rho}(U_\phi))$.

It remains to check the claim about $K^k_{\phi,\rho}$.
Since $j^*K^k_{\phi,\rho} = 0$, $K^k_{\phi,\rho} $ is supported on $\partial O$, and therefore lies in the subcategory of $\text{IndCoh}(\mathcal{N}/S)$ generated by $IC(O',\tau), O'\in\partial O$.

\end{proof}

Applying $a_\psi\circ\nu_*[-d]$ to the short exact sequence of the above lemma gives the exact sequences $$0\to F^{k-1}H_{\phi,\rho}\to F^kH_{\phi,\rho}\to G^kH_{\phi,\rho}\to 0$$ and $$0\to\mathcal{K}^k_{\phi,\rho}\to G^kH_{\phi,\rho}\to \bigoplus\limits_{\tau\in\textup{Supp}_{M_\phi}(G^{k}\mathcal{R}_{S_\phi,\rho}(U_\phi))}i_{b_{\phi,\tau},\sharp !}\pi_{\phi,\tau}\to0$$ in $\text{Had}(\text{Bun}_G)^{\phi^{\text{ss}}}$ (where the hadal t-structure \cite[Section 1.2]{BeijingNotes}, defined on compact objects, is ind-extended).
And furthermore by the above lemma $\mathcal{K}^k_{\phi,\rho}$ lies in the subcategory of $D(\text{Bun}_G)^{\phi^{\text{ss}}}$ generated by sheaves of the form $i_{b,\sharp !}\pi$ such that $c_\psi(i_{b,\sharp !}\pi)\cong \nu_*IC(O',\rho)$ for some $O'\subset \partial O$.

\end{proof}

\begin{rem}
If $\phi$ is semisimple generic, then $U_\phi$ is trivial so the filtration in the above theorem collapses.
And since the unipotent component of Frobenius is trivial every sheaf of the form $IC(\phi(\text{Fr})^u,L)$ is in fact a skyscraper sheaf of the form $i_{\phi*}\rho$ for an irreducible representation $\rho$ of $S_\phi$.
Under the conditions of Theorem \ref{MainThm}, $i_{\phi*}\rho$ is of the form $c_\psi(i_{b!*}\pi)$ for an irreducible representation $\pi$ of $G_b(F)$.
Putting these facts together, $H_\phi[d]$ is a direct sum of sheaves of the form $i_{b!*}\pi$, and in particular it is perverse.
\end{rem}

\appendix

\section{Intertwining Operators and Reducibility}

In this section we review the theory of intertwining operators and prove some results needed in the main body of the text.  We will mostly follow \cite{Wal}.  

As in Section \ref{notation} let $G$ be a reductive group defined over a finite extension $F$ of $\mathbb{Q}_p$, fix $(A_0, P_0)$ , and let $W_G$ be the relative Weyl group of $G$.
$\pi$ will always denote a smooth irreducible representation of $G(F)$, and $M$ will always denote a standard Levi subgroup of $G$.

Recall that the category of smooth representations of $G(F)$ is a direct sum of Bernstein blocks indexed by unramified equivalence classes of pairs $(M,\sigma)$ where $\sigma$ is a supercuspidal representation of $M$ \cite{BD}.
For each such pair, the Bernstein block $\mathfrak{s}(M,\sigma)$ denotes the category of smooth representations $\pi$ of $G(F)$ such that every irreducible subquotient of $\pi$ is a subrepresentation of $i_P^G(\sigma\otimes\chi)$ for some unramified character $\chi$ of $M(F)$, where $i_P^G$ denotes normalized parabolic induction and $P$ denotes any parabolic subgroup of $G$ with Levi $M$.  

Let $W(M):= N_G(M)(F)/M(F) = N_{G(F)}(M(F))/M(F)$, and $W(M,\sigma)= \{w\in W(M)| w(\sigma)=\sigma\}$, where $w(\sigma)$ is the representation of $M$ defined by precomposing $\sigma$ with conjugation by $w^{-1}$.  
For $w\in W(M)$, let $U^w:=wUw^{-1}$, and $P^w = wPw^{-1}$.  

Let $A_M$ be a maximal split torus in the center of $M$.  Let $\Sigma(P)$ denote the roots of $A_M$ positive with respect to $P$, and $\Delta(P)$ the subset of simple roots.  For any positive root $\alpha$ of $A_M$, we define the semistandard Levi subgroup $M_\alpha$ containing $M$ as the centralizer in $G$ of the torus $\text{Ker}(\alpha)^\circ$.

Let $a_M^*$ be $X^*(A_M)\otimes \mathbb{R}$, $a_M$ be its dual $X_*(A_M)\otimes \mathbb{R}$, and $\langle\cdot,\cdot\rangle$ be the standard pairing between them.
If $L$ is a semistandard Levi containing $M$ then we let $a_M^L$ be the subspace of $a_M$ generated by the coroots $\check\alpha$ of $A_M$ such that $A_L$ lies in the kernel of $\alpha$.
The orthogonal complement of $a_M^L$ in $a_M$ with respect to any $W_L(M):=N_L(A_M)/M$ invariant inner form is by standard root system theory equal to the subspace annihilated by $\langle\alpha,\cdot\rangle$ for all of the $\check\alpha$ spanning $a_M^L$.
But this is exactly $a_L$, giving an orthogonal decomposition $a_M=a_L\oplus a_M^L$.
We similarly have an orthogonal decomposition $a_M^* = a_L^*\oplus a_M^{*L}$.

We define the Harish-Chandra map \cite[Section I.1]{Wal} $H_M:M\to a_M$ via the equation $$|\chi(m)|=q^{-\langle \chi, H_M(m)\rangle}$$ for all $\chi\in a_M^*$.  
There is a surjection from $a_M^*\otimes \mathbb{C}$ to the 1-dimensional unramified complex characters of $M$, defined by sending $\lambda$ to the character defined by $\chi_\lambda(m)=q^{-\langle \lambda, H_M(m)\rangle}$.   
Given $\chi$, the set of $\lambda$ such that $\chi_\lambda = \chi$ is not unique, but the real part of $\lambda$ is, and will be denoted $\mathfrak{R}(\chi)$.
We denote the imaginary part, well defined only up to shifts by $X^*(A_M)\cdot 2\pi i/\ln(q)$, by $\mathfrak{I}(\chi)$.
If $\sigma_0$ is a unitary supercuspidal representation and $\sigma\cong\sigma_0\otimes\chi$ is a representation in its inertial equivalence class, we define $\mathfrak{R}(\sigma):=\mathfrak{R}(\chi)$, and more generally $\mathfrak{R}(\sigma)_L$ is the projection of $ \mathfrak{R}(\sigma)$ to $a_L^*$.
For a given inertial equivalence class, this definition does not depend on the choice of $\sigma_0$ within that class, since all such choices differ by unitary unramified characters, and therefore do not affect the real part of $\lambda$.
We define $$\text{Re}(\chi):= q^{-\langle \mathfrak{R}(\chi), H_M(-)\rangle}$$ $$\text{Im}(\chi):=q^{-\langle \mathfrak{I}(\chi), H_M(-)\rangle}$$

We now review the general theory of intertwining operators, closely following \cite[Chapter IV]{Wal}.  We let $P=MU$ and $P'=MU'$ be semistandard parabolics with Levi $M$. By the vanishing of $H^1$ for unipotent groups over $F$ note that $((U'\cap U)\backslash U')(F) = (U'\cap U)(F)\backslash U'(F)$.
The latter group inherits a normalized Haar measure from $U(F)$ and $(U'\cap U)(F)$ since $U'\cap U$ is closed in $U$ and they are both unimodular.
For any admissible representation $\sigma$ of $M$ define the intertwiner $J_{P'|P}(\sigma):i_P^G(\sigma)\to i_{P'}^G(\sigma)$ by $$(J_{P'|P}(\sigma)f)(g) = \int_{((U'\cap U)\backslash U')(F)}f(u'g)du'$$
This integral does not converge in general, but when $\sigma$ is finite length it does after twisting by a sufficiently positive unramified character of $M(F)$, and furthermore the integral can be meromorphically continued to the entire inertial class of $\sigma$ \cite[Th\'eor\`eme IV.1.1]{Wal}.

Letting $\overline{P}$ denote the opposite parabolic of $P$, we can compose $J_{P|\overline{P}}(\sigma)$ with $J_{\overline{P}|P}(\sigma)$ to get an endomorphism of $i_P^G(\sigma)$.  
By \cite[Section IV.3]{Wal}, $i_P^{G}(\sigma)$ is generically irreducible whenever $\sigma$ is irreducible, and hence $J_{P|\overline{P}}(\sigma)J_{\overline{P}|P}(\sigma)$ is generically an isomorphism.
We define the Harish-Chandra $\mu$-function, $\mu_G(\sigma)$, by $$J_{P|\overline{P}}(\sigma)J_{\overline{P}|P}(\sigma) = \mu_G(\sigma)^{-1} \prod\limits_\alpha \gamma(M_\alpha|M)^2$$ \cite[Section V.2]{Wal}, where the $\gamma$ factors are nonzero constants defined in \cite[Section I.1]{Wal}, and the product is over the positive reduced roots of $A_M$.
It is shown in \cite[Section IV.2]{Wal} that the composition $J_{P|\overline{P}}(\sigma)J_{\overline{P}|P}(\sigma)$ is independent of $P$ among the parabolics of G with Levi subgroup $M$, so $\mu$ is independent of $P$.
It follows that if $J_{P|\overline{P}}(\sigma)$ and $J_{\overline{P}|P}(\sigma)$ are finite, then $\mu(\sigma)$ is finite and nonzero if and only if both operators are isomorphisms.

For every reduced positive root $\alpha$ we use $\mu_{\alpha}$ as shorthand for $\mu_{M_\alpha}$.

\begin{prop}\label{proddecomp}
\cite[Section V.2]{Wal} $\mu_G=\prod\limits_{\alpha}\mu_{\alpha}$ where the product is over the reduced positive roots of $A_M$.
\end{prop}

\textbf{Warning}: We caution the reader that when defining the $\mu$ function Waldspurger restricts to the unramified unitary twists of a discrete series representation of $M$, so his definition differs slightly from ours above, but all of the specific properties of the $\mu$ function we will need do not require this restriction.

It will be important to have some conditions about where $\mu$ can have zeros and poles.

\begin{prop} \label{chiconverg}
Let $\chi$ be an unramified character of $M$ such that $\langle\mathfrak{R}(\chi),\check{\alpha}\rangle\neq0$ for all coroots $\check{\alpha}$ of $A_M$ which are positive with respect to $P$ and negative with respect to $P'$.
If $\sigma$ is a unitary supercuspidal representation of $M$, then $J_{P'|P}(\sigma\otimes\chi)$ is finite.
The same conclusion moreover holds if $\sigma$ is a direct summand of a parabolically induced unitary cuspidal.

\end{prop}

\begin{proof}

Let $P=P_0,P_1,...,P_n=P'$ be a minimal length sequence of semistandard parabolics with Levi $M$ such that for all $1\leq i\leq n$, there is exactly one reduced root $\alpha_i$ of $A_M$ which is positive for $P_{i-1}$ and negative for $P_i$.
By \cite[IV.1(12)]{Wal}, $J_{P'|P}(\sigma\otimes\chi) = J_{P'|P_{n-1}}\circ ...\circ J_{P_1|P}(\sigma\otimes\chi)$, so it suffices to show each $J_{P_i|P_{i-1}}$ is finite.
For each $i$ let $Q_i:=P_{i-1}M_{\alpha_i} = P_{i}M_{\alpha_i}$; from now on we abbreviate $M_{\alpha_i}$ by $M_i$.
Then $P_{i-1}\cap M_i$ and $P_i\cap M_i$ are opposite parabolics of $M_i$, and the canonical isomorphism of functors $i_{P_i}^G\cong i_{Q_i}^G\circ i_{P_i\cap M_i}^{M_i}$ together with a definition chase shows $J_{P_i|P_{i-1}} = i_{Q_i}^G(J_{P_i\cap M_i|P_{i-1}\cap M_i})$ when the defining integrals converge, and by rationality the operators agree globally.
So it suffices to show $J_{P_i\cap M_i|P_{i-1}\cap M_i}(\sigma\otimes\chi)$ is finite.
This follows by the assumption on $\chi$ and \cite[IV.1(13)]{Wal}.

For the more general claim suppose $\sigma$ has supercuspidal support $(\tilde M, \tilde \sigma)$ and let $\tilde P$, resp. $\tilde P'$ be a parabolic of $G$ with Levi $\tilde M$ such that $\tilde P M = P$, resp. $\tilde P' M = P'$.  
Then $J_{ P'| P}(i_{\tilde P\cap M}^M(\tilde \sigma)\otimes\chi)=J_{\tilde P'|\tilde P}(\tilde\sigma\otimes\chi)$ is finite, and since $\sigma \otimes\chi$ is a direct summand of $i_{\tilde P\cap M}^M(\tilde \sigma)\otimes\chi$, $J_{ P'| P}(\sigma\otimes\chi)$ must be as well.

\end{proof}

The following corollary is an immediate consequence of the above proposition and the definition of $\mu$.

\begin{cor}\label{munonvanishing}
    In the setting of the Proposition \ref{chiconverg} (taking $P'=\overline{P})$, $\mu(\sigma\otimes\chi)$ is nonzero.  If it is finite, then $J_{\overline{P}|P}(\sigma\otimes\chi)$ is an isomorphism.
\end{cor}

Recall the Langlands classification:

\begin{thm} \cite[Theorem VII.4.2] {Ren}\label{LangClass}

Let $P=MU$ be a standard parabolic with Levi $M$, let $\sigma$ be an irreducible tempered representation of $M$, and let $\chi$ be an unramified character of $M$ such that $\mathfrak{R}(\chi)$ is strictly dominant.  
Then $i_P^G(\sigma\otimes \chi)$ has a unique irreducible quotient, and every irreducible representation of $G$ is isomorphic to such a quotient for a unique equivalence class of triples $(P, \sigma, \chi)$.
Here two triples $(P = MU, \sigma, \chi)$ and $(P' = M'U', \sigma', \chi')$ are equivalent if $M = M'$, $\mathfrak{R}(\chi'\chi^{-1}) = 0$, and $\sigma\otimes\chi\cong\sigma'\otimes\chi'$.

\end{thm}

\begin{prop}\label{mufiniteimpliesirred}
Let $\sigma$ be an irreducible representation of $M$ which is a direct summand of a parabolically induced unitary supercuspidal, and let $\chi$ be a strictly dominant unramified character of $M$.  
Then if $\mu(\sigma\otimes\chi)$ is finite, $i_P^G(\sigma\otimes\chi)$ is irreducible.  
\end{prop}
\begin{proof}
By Theorem \ref{LangClass} $i_P^G(\sigma\otimes\chi)$ has a unique irreducible quotient $J(\sigma,\chi)$, which is isomorphic to the image $J_{\overline{P}|P}(\sigma\otimes\chi)$ in such a way that composing $J_{\overline{P}|P}(\sigma\otimes\chi)$ with this isomorphism is equal to the standard projection of $i_P^G(\sigma\otimes\chi)$ to $J(\sigma,\chi)$.  
By Corollary \ref{munonvanishing}, since $\mu(\sigma\otimes\chi)$ is finite, $J_{\overline{P}|P}(\sigma\otimes\chi)$ is an isomorphism, and hence $i_P^G(\sigma\otimes\chi)$ is irreducible.
    
\end{proof}

We will make crucial use of the following theorem of Heiermann:
\begin{thm}\label{HeiThm}\cite[Corollaire 8.7]{Hei}
Let $\sigma$ be a supercuspidal representation of $M(F)$ such that $\mathfrak{R}(\sigma)_G=0$.  Then $[M,\sigma]$ is the supercuspidal support of a discrete series representation of $G$ if and only if $\mu$ has a pole at $\sigma$ of order equal to the parabolic rank of $M$.  Furthermore this is the maximal order a pole at $\sigma$ can have.
\end{thm}

We now wish to drop the $\mathfrak{R}(\sigma)_G=0$ assumption in the above theorem.  To do so we will need some notation. 
Let $K$ be a maximal compact subgroup of $G(F)$ which is the stabilizer of a special point in the apartment of $A_0$.  Using the Iwasawa decomposition $G=MUK$ it is well known that the restriction map from $I_P^G(\sigma)$ to $I_{P\cap K}^K(\sigma|_{M\cap K})$ is an isomorphism.
Note that for any unramified character $\chi$, $I_{P\cap K}^K(\sigma|_{M\cap K})$ and $I_{P\cap K}^K((\sigma\otimes\chi)|_{M\cap K})$ have the same underlying vector spaces.
This will be used in the proof of the following lemma:
\begin{lem}\label{essentialinvariance}
Let $\chi$ be an unramified character of $G(F)$ restricted to $M(F)$. Then $J_{\overline{P}|P}$ has the same order of vanishing at $\sigma$ and $\sigma\otimes\chi$.
\end{lem}

\begin{proof}

Since $\chi$ is trivial on unipotent elements of $G(F)$, $J_{\overline{P}|P}(\sigma\otimes\chi) = J_{\overline{P}|P}(\sigma)\otimes\chi$.
Then triviality of $\chi$ on $K$ and the isomorphisms $I_P^G(\sigma)\cong I_{P\cap K}^K(\sigma|_{M\cap K})$ and $I_P^G(\sigma\otimes\chi)\cong I_{P\cap K}^K((\sigma\otimes\chi)|_{M\cap K})\cong I_{P\cap K}^K(\sigma|_{M\cap K})$ show that $J_{\overline{P}|P}(\sigma)=J_{\overline{P}|P}(\sigma\otimes\chi)$ as operators on $I_{P\cap K}^K(\sigma|_{M\cap K})$.
Since $\sigma$ was arbitrary, $J_{\overline{P}|P}(\sigma)$ and $J_{\overline{P}|P}(\sigma\otimes\chi)$ have the same order of vanishing at all points.
    
\end{proof}

\begin{cor}\label{HeiCor}
Let $\sigma$ be a supercuspidal representation of $M(F)$.  Then $[M,\sigma]$ is the supercuspidal support of an essentially discrete series representation of $G$ if and only if $\mu$ has a pole at $\sigma$ of order equal to the parabolic rank of $M$, which is the maximum possible order $\mu$ could attain at $\sigma$.
\end{cor}

\begin{proof}
Since an essentially discrete series representation is by  definition the tensor product of a discrete series representation and an unramified character of $G(F)$, this follows by Theorem \ref{HeiThm} and Lemma \ref{essentialinvariance}.
\end{proof}

We now arrive at the main theorem of this section:

\begin{thm}\label{GenCompRed}
Let $\pi$ be a smooth irreducible representation of $G(F)$ with supercuspidal support $[M,\sigma]$, and let $P=MU$ be a parabolic with Levi $M$.
Assume that $G$ has a classical local Langlands correspondence as in Section \ref{ClassicalLLC},\footnote{It is not necessary for $G$ to have a full $B(G)$-parametrized LLC; it is only necessary for the L-packets of $G$ itself to satisfy Properties \ref{ClassLLCProp}.  The L-packets of $G_b$ when $b\neq1$ are not used in the theorem.} and that the L-parameter $\phi$ of $\pi$ is generic in the sense of Definition \ref{GenLParamDef}.  
Then $i_P^G(\sigma)$ is completely reducible.
\end{thm}

\begin{rem}

In \cite[Appendix A]{HamEis} Hamann studies intertwining operators for principal series, and gives sufficient conditions for the parabolic induction of a character of a maximal torus to be irreducible.  Essentially, the condition is equivalent to requiring $S_\phi$ be a torus, which can be viewed as a stronger form of generosity that he calls \textit{generic regular} (the actual definition consists of generosity and the requirement that the character not be fixed by any Weyl group element).
    
\end{rem}

\begin{proof}
Without loss of generality we may assume $P$ is standard.  
Write $\sigma$ as $\sigma_0\otimes\chi$ where $\sigma_0$ is unitary and $\chi$ is an unramified character of $M(F)$.
We define $$A_L:=\bigcap\limits_{\{\alpha:\langle\mathfrak{R}(\chi), \check{\alpha}\rangle= 0\}} \text{Ker}(\alpha)^\circ$$ and the semistandard Levi subgroup $L:=C_G(A_L)$.
Choose $\lambda\in a_M^*$ such that $\chi(\cdot)=q^{-\langle\lambda,H_M(\cdot)\rangle}$.

\begin{lem}
$\textup{Re}(\chi)$ can be uniquely extended to a positive unramified character of $L(F)$.  Therefore there is a canonical isomorphism $i_{P\cap L}^L(\sigma)\cong i_{P\cap L}^L(\sigma_0\otimes\textup{Im}(\chi))\otimes\textup{Re}(\chi)_L$.
\end{lem}
\begin{proof}
Let $p:a_M\to a_L$ be the canonical projection with kernel $a_M^L$ the span of the coroots $\check\alpha$ used to define $L$.
Since $\langle\mathfrak{R}(\chi),\cdot\rangle: a_M\to\mathbb{C}$ annihilates $a_M^L$ it factors through $a_L$, so $\langle\mathfrak{R}(\chi),-\rangle =p^*\langle\mathfrak{R}(\chi)_L,-\rangle$ for some unique $\mathfrak{R}(\chi)_L\in a_L^*$.
Then $\text{Re}(\chi)_L(\cdot):=q^{-\langle \mathfrak{R}(\chi)_L,H_L(\cdot)\rangle}$ is the desired extension of $\text{Re}(\chi)$ to $L(F)$.
The intertwining operator in the last claim is given by multiplication by $\text{Re}(\chi)_L^{-1}$.
\end{proof}

Since $\sigma_0\otimes\text{Im}(\chi)$ is unitary, $i_{P\cap L}^L(\sigma_0\otimes\text{Im}(\chi))$, and hence $i_{P\cap L}^L(\sigma)\cong i_{P\cap L}^L(\sigma_0\otimes\text{Im}(\chi))\otimes\textup{Re}(\chi)_L$, is completely reducible.
Since temperedness is preserved under parabolic induction, $i_{P\cap L}^L(\sigma_0\otimes\text{Im}(\chi))$ is tempered.
Let $Q'$ be the semistandard parabolic subgroup of $G$ with Levi $L$ such that $\langle\mathfrak{R}(\chi),\check{\alpha}\rangle>0$ for every root $\alpha$ of $U_{Q'}$, and let $P' = (P\cap L)U_{Q'}$.
For each irreducible constituent $\tilde\sigma$ of $i_{P'\cap L}^L(\sigma_0\otimes\text{Im}(\chi))$, by the Langlands classification $i_{Q'}^G(\tilde{\sigma}\otimes\textup{Re}(\chi)_L)$ has a unique irreducible quotient.
By Proposition \ref{mufiniteimpliesirred}, $i_{Q'}^G(\tilde\sigma\otimes \text{Re}(\chi)_L)$ is reducible only if $\mu^L_G$ acquires a pole at $\tilde\sigma\otimes \text{Re}(\chi)_L$.
Suppose this to be the case.
Then by Proposition \ref{proddecomp} there must exist a root $\alpha$ of $A_L$ such that $\mu^L_{L_\alpha}$ has a pole at $\tilde\sigma\otimes \text{Re}(\chi)_L$, and therefore $J_{Q'\cap L_\alpha|\overline{Q'}\cap L_\alpha}\circ J_{\overline{Q'}\cap L_\alpha|Q'\cap L_\alpha}$ has a zero at $ \tilde\sigma\otimes \text{Re}(\chi)_L$.

Let $\textbf{X}_\alpha$ denote the set of positive reduced roots of $A_M$ which are roots of $L_\alpha$ but not $L$ (ie those restricting to a positive multiple of $\alpha$), and let $\tilde{P}_\alpha'$ denote the parabolic $(P'\cap L)\cdot \overline{U}_{Q'\cap L_\alpha}$ of $L_\alpha$.
A definition chase shows \[J_{\overline{Q'}\cap L_\alpha|Q'\cap L_\alpha}(i_{Q'\cap L_\alpha}^{L_\alpha}i_{P'\cap L}^L(\sigma))\cong J_{\tilde{P}_\alpha'|P'\cap L_\alpha}(i_{P'\cap L_\alpha}^{L_\alpha}(\sigma))\] as operators when the defining integrals converge, and in general the isomorphism follows since both operators are rational.
Combining this with the analogous isomorphism in the other direction it follows that 
\begin{align*}
J_{Q'\cap L_\alpha|\overline{Q'}\cap L_\alpha}\circ J_{\overline{Q'}\cap L_\alpha|Q'\cap L_\alpha}(i_{Q'\cap L_\alpha}^{L_\alpha}i_{P'\cap L}^{L}(\sigma)) \cong J_{P'\cap L_\alpha|\tilde{P}_\alpha'}\circ J_{\tilde{P}_\alpha'|P'\cap L_\alpha}(i_{P'\cap L_\alpha}^{L_\alpha}(\sigma))
\end{align*}

By \cite[IV.3(4)]{Wal} and the definition of the $\mu$ function the right side is, up to a nonzero constant independent of $\sigma$, isomorphic to $$(\prod\limits_{\beta\in\textbf{X}_\alpha}\mu_{M_{\beta}}(\sigma)^{-1})i_{P'\cap L_\alpha}^{L_\alpha}(\sigma)\cong (\prod\limits_{\beta\in\textbf{X}_\alpha}\mu_{M_{\beta}}(\sigma)^{-1})\bigoplus\limits_ii_{Q'\cap L_\alpha}^{L_\alpha}(\tilde\sigma_i\otimes\text{Re}(\chi)_L)$$

On the other hand, left side is
\begin{align*}
J_{Q'\cap L_\alpha|\overline{Q'}\cap L_\alpha}\circ J_{\overline{Q'}\cap L_\alpha|Q'\cap L_\alpha}(i_{Q'\cap L_\alpha}^{L_\alpha}i_{P'\cap L}^{L}(\sigma))&\cong J_{Q'\cap L_\alpha|\overline{Q'}\cap L_\alpha}\circ J_{\overline{Q'}\cap L_\alpha|Q'\cap L_\alpha}(i_{Q'\cap L_\alpha}^{L_\alpha}\bigoplus\limits_i\tilde\sigma_i\otimes\text{Re}(\chi)_L) \\&\cong \bigoplus\limits_iJ_{Q'\cap L_\alpha|\overline{Q'}\cap L_\alpha}\circ J_{\overline{Q'}\cap L_\alpha|Q'\cap L_\alpha}(i_{Q'\cap L_\alpha}^{L_\alpha}(\tilde\sigma_i\otimes\text{Re}(\chi)_L))\\ &\cong\bigoplus\limits_i\mu^L_{L_\alpha}(\tilde\sigma_i\otimes\text{Re}(\chi)_L)^{-1}i_{Q'\cap L_\alpha}^{L_\alpha}(\tilde\sigma_i\otimes\text{Re}(\chi)_L)
\end{align*}

And equating the two expressions it follows that if $\mu^L_{L_\alpha}(\tilde\sigma\otimes\text{Re}(\chi)_L)^{-1}$ vanishes then so must $\mu_{M_\beta}(\sigma)^{-1}$ for some $\beta\in \textbf{X}_\alpha$.

By Corollary \ref{HeiCor} the pole of $\mu_{M_\beta}$ at $\sigma$ has order at most $1$, and is therefore simple.
Applying Corollary \ref{HeiCor} again, $i_{P'\cap M_\beta}^{M_\beta}(\sigma)$ has an essentially discrete series subquotient $\tilde{\pi}_\beta$.
By Property \ref{ClassLLCProp} (4) the semisimple L-parameters of $\sigma$, $\tilde{\sigma}\otimes\textup{Re}(\chi)_L$, $\tilde{\pi}_\beta$, and $\pi$ all agree after composing along the respective inclusions from $^LM$, $^LL$, and $^LM_\beta$ into $^LG$.
By Property \ref{ClassLLCProp}(2) the L-parameter $\phi_{\tilde{\pi}_\beta} $ of $\tilde{\pi}_\beta$ is essentially discrete.
Since the semisimple L-parameter of $\tilde\pi_\beta$ factors through $M$, its L-parameter must have nontrivial monodromy, and hence $\phi^{\text{ss}}_\pi = \iota_{M_\beta}\circ\phi^{\text{ss}}_{\tilde{\pi}_\beta}$ cannot be generic.

Since $\phi^{\text{ss}}_\pi$ is generic by assumption, we conclude that $\mu_{M_\beta}$ cannot have a pole at $\sigma$, and therefore $i_{Q'}^G(\tilde{\sigma}\otimes\textup{Re}(\chi)_L)$ is irreducible.
It follows that $i_{P'}^G(\sigma)$ is completely reducible.

Now by Proposition \ref{chiconverg} $J_{P'|P}(\sigma)$ and $J_{P|P'}(\sigma)$ are finite.
By \cite[IV.3(4)]{Wal}, up to a nonzero constant independent of $\sigma$, both $J_{P|P'}\circ J_{P'|P}$ and $J_{P'|P}\circ J_{P|P'} $ are given by $  \prod \mu_\alpha^{-1} $ where the product is over the reduced roots $\alpha$ of $U$ for which $\langle\mathfrak{R}(\chi),\check\alpha\rangle<0$.
To show $J_{P'|P}(\sigma)$ is an isomorphism it therefore suffices to show every factor $\mu_\alpha^{-1}(\sigma)$ is nonzero, which follows since $\mu_\alpha(\sigma)$ does not have a pole.
Therefore $i_P^G(\sigma)\cong i_{P'}^G(\sigma)$ is completely reducible.

\end{proof}

\end{document}